\documentclass[reqno, a4paper, 10pt]{amsart}
\usepackage{amssymb}
\usepackage{mathrsfs}
\usepackage{amssymb, url, color, graphicx, amscd, mathrsfs}
\usepackage[colorlinks=true, bookmarks=true, pdfstartview=FitH, pagebackref=true, linktocpage=true, linkcolor = magenta, citecolor = blue]{hyperref}
\usepackage[nodayofweek]{datetime}
\usepackage{graphicx}
\usepackage{times}
\usepackage[subnum]{cases}
\usepackage{tabularx}
\usepackage{enumitem} 

\DeclareMathAlphabet{\mathpzc}{OT1}{pzc}{m}{it}

\newtheorem{theorem}{Theorem}[section]
\newtheorem{lemma}[theorem]{Lemma}
\newtheorem{corollary}[theorem]{Corollary}
\newtheorem{proposition}[theorem]{Proposition}

\theoremstyle{remark}
\newtheorem{remark}[theorem]{Remark}

\numberwithin{equation}{section}

\DeclareMathOperator{\R}{{\mathbf R}}

\DeclareMathOperator{\Ric}{Ric}

\DeclareMathOperator{\dist}{dist}

\let\vol\Vol
\newcommand{\dvg}{\, d\mu_{g_0}}

\newcommand{\dvSn}{\, d\mu_{\mathbb S^n} }
\newcommand{\Po}{\mathbf{P}_0}
\newcommand{\PSn}{\mathbf{P}_{\mathbb S^n}}
\renewcommand{\P}{\mathbf{P}}
\let\epsilon\varepsilon

\newcommand{\Cscr}{\mathscr C}
\newcommand{\Escr}{\mathscr E}

\newcommand{\Lscr}{\mathscr L}

\newcommand{\Uscr}{\mathscr U}

\allowdisplaybreaks

\makeatletter
\def\@cite#1#2{[\textbf{#1}\if@tempswa, #2\fi]}
\makeatother

\def\XXint#1#2#3{{\setbox0=\hbox{$#1{#2#3}{\int}$}
 \vcenter{\hbox{$#2#3$}}\kern-.5\wd0}}

\begin{document}

\title[Prescribed $Q$-curvature flow ]{Prescribed $Q$-curvature flow on closed manifolds of even dimension}

\def\cfac#1{\ifmmode\setbox7\hbox{$\accent"5E#1$}\else\setbox7\hbox{\accent"5E#1}\penalty 10000\relax\fi\raise 1\ht7\hbox{\lower1.0ex\hbox to 1\wd7{\hss\accent"13\hss}}\penalty 10000\hskip-1\wd7\penalty 10000\box7 }

\author[Q.A. Ng\^o]{Qu\cfac oc Anh Ng\^o}

\address[Q.A. Ng\^o]{Department of Mathematics, College of Science, Vi\^{e}t Nam National University, H\`{a} N\^{o}i, Vi\^{e}t Nam and Graduate School of Mathematical Sciences, The University of Tokyo, 3-8-1 Komaba, Meguro-ku, Tokyo 153-8914, Japan.}
\email{\href{mailto: Q.A. Ng\^o <nqanh@vnu.edu.vn>}{nqanh@vnu.edu.vn}}
\email{\href{mailto: Q.A. Ng\^o <ngo@ms.u-tokyo.ac.jp>}{ngo@ms.u-tokyo.ac.jp}}

\author[H. Zhang]{Hong Zhang}

\address[H. Zhang]{Department of Mathematics, University of Science and Technology of China, No. 96 Jin Zhai Road, Hefei, Anhui, China 230026 and School of Mathematics, South China Normal University, Guangzhou, China 510631.}
\email{\href{mailto: H. Zhang <matzhang@ustc.edu.cn>}{matzhang@ustc.edu.cn}}
\email{\href{mailto: H. Zhang <zhanghong@scnu.edu.cn>}{zhanghong@scnu.edu.cn}}

\thanks{}



\date{\bf \today \ at \currenttime}

\begin{abstract}
On a closed Riemannian manifold $(M,g_0)$ of even dimension $n \geq 4$, the well-known prescribed $Q$-curvature problem asks whether or not there is a metric $g$ comformal to $g_0$ such that its $Q$-curvature, associated with the GJMS operator $\P_g$, is equal to a given function $f$. Letting $g = e^{2u}g_0$, this problem is equivalent to solving
\[
\P_{g_0} u+Q_{g_0} = f e^{nu},
\]
where $Q_{g_0}$ denotes the $Q$-curvature of $g_0$. The primary objective of the paper is to introduce the following negative gradient flow of the time dependent metric $g(t)$ conformal to $g_0$,
\[
\frac{\partial g (t)}{\partial t}= -2\Big(Q_{g (t)} - \frac{\int_M f Q_{g(t)} d\mu_{g(t)} }{\int_M f^2 d\mu_{g(t)} }f \Big)g(t) \quad \text{ for } t >0,\\
\]
to study the problem of prescribing $Q$-curvature. Since $\int_M Q_g d\mu_g$ is conformally invariant, our analysis depends on the size of $\int_M Q_{g_0} \dvg$, which is assumed to satisfy
\[
\int_M Q_0 \dvg \ne k (n-1)! \, {\rm vol}(\mathbb S^n) \quad \text{ for all } \; k = 2,3,...
\]
The paper is twofold. First, we identify suitable conditions on $f$ such that the gradient flow defined as above is defined to all time and convergent, as time goes to infinity, sequentially or uniformly. Second, we show that various existence theorems for prescribed $Q$-curvature problem can be derived from the convergence of the flow. 
\end{abstract}

\maketitle


\section{Introduction}

On a closed manifold $(M, g )$ of dimension $n \geqslant 3$, a formally self-adjoint geometric differential operator $A_g$ of the metric $g$ is called \textit{conformally covariant of bidegree $(a,b)$} if
\begin{equation}\label{eqCCB}
A_{g_w} (\varphi) = e^{-b w} A_g (e^{aw} \varphi)
\end{equation}
for all $\varphi \in C^\infty (M)$, where $g_w := e^{2w} g $ is a conformal metric to $g$. A typical geometric differential operator in conformal geometry is the \textit{second-order conformal Laplacian} which is defined by 
\[
L_g ( \varphi ):= -\Delta_g\varphi + \frac{n-2}{4(n-1)} R_g \varphi ,
\]
where $R_g$ is the scalar curvature of $g$. Clearly, $L_g$ is conformally covariant of bidegree $((n-2)/2,(n+2)/2)$ in the sense of \eqref{eqCCB} since
\begin{equation}\label{eqConformalLaplacian}
L_{g_w} (\varphi)=e^{-\frac{(n+2)w}2} L_g (e^{\frac{(n-2)w}2} \varphi)
\end{equation}
for all $\varphi \in C^\infty (M)$. If we write $u^{4/(n-2)} = e^{2w}$, then Eq. \eqref{eqConformalLaplacian} reads
\[
L_{g_w} (\varphi)=u^{-\frac{n+2}{n-2}} L_g (u \varphi)
\]
for all $\varphi \in C^\infty (M )$. By setting $\varphi \equiv 1$, we get, under the conformal change of metric $g_u = u^{4/(n-2)} g$, the transformation law of the scalar curvature $R_g$ 
\begin{equation}\label{eqConformalChangeOfScalarCurvature}
-\frac{4(n-1)}{n-2} \Delta_g u + R_g u = R_{g_u} u^\frac{n+2}{n-2}.
\end{equation}
 In the literature, Eq. \eqref{eqConformalChangeOfScalarCurvature} is closely related to the prescribed scalar curvature problem with the prescribed function $R_{g_u}$. This challenging problem has already captured much attention by many mathematicians during the last few decades. Notice that when $R_{g_u}$ is constant, the prescribed scalar curvature problem becomes the famous Yamabe problem. 

The first higher order example of conformal operators, discovered by Paneitz \cite{Paneitz}, is the \textit{fourth-order Paneitz operator} $\P_g^4$ on a $4$-manifold $(M,g)$. It is given by
\begin{equation*}
\P^4_g (\varphi) =(-\Delta_g)^2\varphi -\mbox{div}_g\Big( \big(\frac23R_g g -2\Ric_g \big)d\varphi \Big),
\end{equation*}
where $d$ is the differential and $\Ric_g$ denotes the Ricci tensor. (Note that the superscript $4$ does not denote the dimension of the underlying manifold but the order of the operator. However, the subscript $g$ does denote the metric $g$ being used.) The Paneitz operator $\P_g^4$ is conformally covariant of bidegree $(0,4)$ in the sense of \eqref{eqCCB} because under the conformal change $g_u=e^{2u}g$, there holds
\begin{equation}\label{eqConformalPaneitz}
\P^4_{g_u} (\varphi)=e^{-4u} \P^4_g (\varphi)
\end{equation}
for all $\varphi \in C^\infty (M)$. Upon letting $\varphi \equiv 1$, we deduce from \eqref{eqConformalPaneitz} that
\[
Q_{g_u}^4=e^{-4u}(\P^4_g (u)+Q_g^4),
\]
where the quantity $Q_g^4$ is defined by
\[
Q_g^4=-\frac16(\Delta R_g-R_g^2+3|\Ric_g|^2).
\]
The quantity $Q_g^4$, discovered by Branson \cite{Branson85}, is called the $Q$-curvature associated with the Paneitz operator $\P^4$ on $(M,g)$. One of interesting features of the $Q$-curvature is that it obeys, by the Gauss--Bonnet--Chern theorem, the following identity
\[
\int_M Q_g^4+\frac14 |W_g|_g^2~d\mu_g=8\pi^2\chi(M),
\]
where $\chi(M)$ is the Euler characteristic of $M$ and $W_g$ is the Weyl tensor. Since the Weyl tensor $W_g$ is conformally invariant, the identity above implies that the total $Q$-curvature $\int_M Q_g^4 d\mu_g$ is also conformally invariant; that is the value $\int_M Q_g^4 d\mu_g$ does not depend on the metric $g$ but the conformal class $[g]$ represented by $g$. Therefore, the problem of seeking solutions to \eqref{eqConformalPaneitz} is split into several cases depending on the value of $\int_M Q_g^4 d\mu_g$. Serving as Aubin's threshold in the Yamabe problem, there is a threshold $16\pi^2$ for $\int_M Q_g^4 d\mu_g$ in the $Q$-curvature problem on the 4-manifold $M$. (The number $16\pi^2$, which is $3! \vol (\mathbb S^4)$, is exactly equal to $\int_{\mathbb S^4} Q_{g_{\mathbb S^4}}^4 d\mu_{g_{\mathbb S^4}}$ being calculated with respect to the standard metric $g_{\mathbb S^4}$ on $\mathbb S^4$.) The technical condition $\int_M Q_g^4 d\mu_g < 16\pi^2$ is often used since the standard variational argument can be applied. Notice that such a condition is not so restrictive because there is a large class of satisfied manifolds; see \cite{Gursky}. 

For general even dimension $n\geqslant 4$, Graham, Jenne, Mason and Sparling \cite{GJMS} discovered a similar operator $\P_g^n$ which is also a conformally invariant, elliptic, self-adjoint operator of order $n$ with leading term $(-\Delta)^{n/2}$. This operator is commonly known as the GJMS operator. The work \cite{GJMS} was based on an earlier one by Fefferman and Graham \cite{fg1985}, in which existence of scalar conformal invariants were derived via a new tool called ambient metric construction. Similar to \eqref{eqConformalPaneitz}, the GJMS operator is conformally covariant of bidegree $(0,n)$ in the sense that under the conformal change $g_u=e^{2u}g$, there holds
\begin{equation*}
\P^n_{g_u} (\varphi)=e^{-nu} \P^n_g (\varphi)
\end{equation*}
for all $\varphi \in C^\infty (M )$. As in the fourth order case, the higher order operator $\P^n_g$ has an associated curvature quantity $Q^n_g$; see \cite[page 11]{Branson93}. The quantity $Q^n_g$ obeys a similar transformation law as $Q^4_g$ does, namely
\begin{equation}\label{qcurvartureequation}
Q_{g_u}^n=e^{-nu}(\P^n_g u+Q_g^n).
\end{equation}
An immediate consquence of \eqref{qcurvartureequation} is that $\int_M Q_{g_u}^n d\mu_{g_u}$ is conformal invariant.

Up to this point, we can describe \textit{the prescribed $Q$-curvature problem} on a generic manifold of even dimension. Assume that $(M ,g_0)$ is a closed Riemannian manifold of even dimension $n$ with background metric $g_0$ and that $f$ is a smooth function on $M$. Then one may ask if there exists a pointwise conformal metric $g$, that is $g=e^{2u}g_0 \in [g_0]$ for some smooth function $u$, such that $f$ can be realized as the $Q$-curvature of $g$. Now, thanks to the rule \eqref{qcurvartureequation}, this geometric problem is equivalent to solving the higher order nonlinear equation
\begin{equation}\label{prescribeqcurvatureequation}
 \Po u+Q_0=fe^{nu},
\end{equation}
where we simply set $\Po := \P_{g_0}^n$ and $Q_0:= Q_{g_0}^n$. In the sequel, we also set $\P_g:=\P_g^n$ and $Q_g:=Q_g^n$ for brevity. Back to \eqref{prescribeqcurvatureequation}, if the equation has a solution $u$, then integrating both sides of the equation gives
\[
 \int_M Q_0 \dvg = \int_M f e^{nu} \dvg.
\]
Consequently, the following conditions, corresponding to the value of $\int_MQ_0 \dvg$, on the smooth function $f$ are necessary for the existence of solutions to \eqref{prescribeqcurvatureequation}
\begin{equation}\label{NecessaryCondition}
\left.
\begin{aligned}
&\sup_{x \in M} f(x)>0 & \mathrm{ if } \displaystyle \int_M Q_0 \dvg>0,\\
&\sup_{x \in M} f(x)>0 \quad \mathrm{ and } \quad \inf_{x\in M}f(x)<0 &\mathrm{ if } \displaystyle \int_M Q_0 \dvg=0,\\
& \inf_{x \in M} f(x)<0 & \mathrm{ if } \displaystyle \int_M Q_0 \dvg<0.\\
\end{aligned}
\right\}
\end{equation}
In the higher dimensional cases, the number $(n-1)!\vol (\mathbb S^n)$ is the threshold, which plays a similar role as that of $16\pi^2$ on 4-manifolds. As a result, we call Eq. \eqref{prescribeqcurvatureequation} is in the \textit{subcritical case}, \textit{critical case}, or the \textit{supercritical case} accordingly if $\int_M Q_0 \dvg <, =,~\mbox{or}~> (n-1)!\vol(\mathbb S^n)$. For convenience, we also call Eq. \eqref{prescribeqcurvatureequation} is in the \textit{negative case} or the \textit{null case} if $\int_M Q_0 \dvg<~\mbox{or}~=0$, respectively.

Within the last few decades, prescribed curvature problems hold a central place in conformal geometry staring from the famous Yamabe problem. Because the class of lower-order problems has already studied and a general picture is now quite clear, people have already turned their attention to higher-order problems especially during the last two decades. There are many research works on the Eq. \eqref{prescribeqcurvatureequation}, some of them are listed as follows: {\itshape subcritical case}: \cite{br, bfr, bfr09}; {\itshape critical case}: \cite{wx1998, br03, br06, ms2006, ChenXu-2011, ho2012, lll2012}; and {\itshape supercritical case}: \cite{n2007, dm2008, bfr09, n2015}. The majority of these works make use of an assumption on the operator $\Po$, that is,
\begin{equation}\label{P0}
\text{the operator } \Po~~\text{is positive with kernel consisting of the constant functions}.
\tag{P}
\end{equation}
It is possible to work with a weaker assumption on $\Po$, for instance, one may relax the positivity of $\Po$ as considered in \cite{bfr09}. In this sense, the operator $\Po$ as finite many negative eigenvalue and the analysis is rather involved. For example, the Adams inequality \eqref{eqAdamsTypeInequality} is only valid for any function $u$ with zero projection onto the eigenspace spanned by negative eigenvalues of $\Po$. In the present work, we do not pursue this direction and leave it for future research. 

Among those works above, in particular, Brendle \cite{br} considered the evolution of metrics
\begin{equation}\label{flowbrendle}
\left\{
\begin{aligned}
\frac{\partial g (t)}{\partial t} & = -\Big (Q_{g (t)} - \frac{\int_MQ_0 \dvg }{\int_Mf \dvg }f\Big)g(t) & & \text{ for } t >0,\\
g(0) & = e^{2u_0} g_0,& &
\end{aligned}
\right.
\end{equation}
where $f$ is assumed to be positive everywhere and $u_0$ is a smooth initial data. Then he proved, in the subcritical case, that if the condition \eqref{P0} holds, the flow \eqref{flowbrendle} has a solution which is defined to all time and converges at infinity to a metric $g_\infty$ with
\[
Q_{g_\infty} = \frac{ \int_MQ_0 \dvg }{ \int_Mf \dvg } f.
\]
A direct conclusion of the above is that after shifting by a constant, 
Eq. \eqref{prescribeqcurvatureequation} has a solution in the negative case (with $f$ replaced by $-f$ in Eq. \eqref{prescribeqcurvatureequation} in this case) or if $0<\int_M Q_0 \dvg < (n-1)!\vol(\mathbb S^n)$. Notice that by the positivity of $\Po $ mentioned in the hypothesis \eqref{P0} we mean
\begin{equation}\label{eqPositivityOfP_0}
\int_M u \cdot \Po u \dvg \geqslant \lambda_1 \int_M (u-\overline u)^2 \dvg \geqslant 0
\end{equation}
for all function $u \in C^\infty (M)$. Here $\lambda_1 > 0$ is the first (non-trivial) eigenvalue of the operator $\Po $ and
\[
\overline u = \frac 1{\vol(M,g_0)} \int_M u \dvg
\] 
denotes the average of $u$ calculated with respect to $g_0$. (In the rest of our paper, by $\vol(K)$ we mean the volume of a subset $K$ of $M$ being calculated with respect to $g_0$. If we want to emphasize that the metric $g$ is being used to calculate the volume of $K$, we shall write $\vol(K,g)$.) Furthermore, we also drop the symbol $g_0$ in any functional space if the backgroup metric $g_0$ is being used, for examples, $L^2(M) : = L^2(M, g_0)$, $H^{n/2}(M) := H^{n/2} (M, g_0)$, etc. 

Later on, Baird et al. \cite{bfr} extended Brendle's result to the case in which the candidate curvature $f$ is allowed to change sign when the underlying manifold has dimension 4. They adopt an abstract negative gradient flow which is different from \eqref{flowbrendle}. To be precise, they considered the functional
\begin{equation*}
 J[u]=\int_M u\cdot \Po u \dvg +2\int_MQ_0u \dvg ,
\end{equation*}
on the Sobolev space $H^2(M)$ under the constraint
\[
u\in X=: \Big\{u\in H^2(M): \int_Mfe^{4u}\dvg =\int_MQ_0 \dvg \Big\}.
\]
Then, they studied the following negative gradient flow of functional $J$ with respective to hypersurface $X$ of Sobolev space $H$
\begin{equation}\label{flowpaulbaird}
 \left\{
\begin{aligned}
 \partial_tu& =-\nabla^XJ(u)\\
 u(0) & =u_0\in X,
 \end{aligned}
\right.
\end{equation}
where $\nabla^XJ$ is the projection of $\nabla J$ onto $X$. If the flow \eqref{flowpaulbaird} exists for all time and converges at infinity, then the limit function $u_\infty$ yields a solution to \eqref{prescribeqcurvatureequation} with $n=4$.

The primary objective of the paper is to generalize the results in \cite{bfr} to all even dimensional manifolds and to further extend the work \cite{br} to the supercritical case. In an earlier version of this paper \cite{NZ-arXiv}, we only consider the case $\int_M Q_0 \dvg \leqslant (n-1)!\vol(\mathbb S^n)$. The method used in \cite{NZ-arXiv} is again the negative gradient flow which is analogous to \eqref{flowbrendle}; see also \cite{NX2015}. In the present version, we relax the restriction of $\int_M Q_0 \dvg$; instead, we only require that $\int_M Q_0 \dvg \ne k (n-1)!\vol(\mathbb S^n)$ for all $k \in \{2,3,...\}$. We keep following the strategy used in \cite{NZ-arXiv}. However, in order to deal with arbitrarily large $\int_M Q_0 \dvg$, we shall exploit the connection between $L^2$-bound and the finite energy of the flow, which was already observed in \cite{fr18} for the case of constant $Q$-curvature; see Section \ref{sec-BoundInH^(n/2)}. For examples of manifolds having $\int_M Q_0 \dvg > (n-1)!\vol(\mathbb S^n)$, we refer the reader to the work of Djadli and Malchiodi \cite{dm2004}.

Now, let us describe our flow method in detail and state the main results of the paper. As in \cite{bfr}, we associate any function $u\in H^{n/2} (M )$ the following energy
\begin{equation}\label{functional}
\Escr [u]=\frac n2\int_Mu\cdot \Po u \dvg + n\int_M Q_0u \dvg.
\end{equation}
We assume the condition \eqref{P0}. Since $\Po$ has the leading term $(-\Delta)^{n/2}$, by adopting the convention that $(-\Delta_0 )^{m+1/2} = \nabla_0 (-\Delta_0 )^m $ for all integers $m$, it is not hard to see that for any $u\in H^{n/2} (M )$ we have
\begin{equation}
 \label{eqBoundForIntegralOfuP_ou}
 \int_M |(-\Delta_0)^{n/4} u|^2 \dvg \lesssim \int_Mu\cdot \Po u \dvg \lesssim \int_M |(-\Delta_0)^{n/4} u|^2 \dvg .
\end{equation}
 Here and thereafter, we sometimes use $\lesssim$ to denote an inequality up to a uniform constant which only depends on the geometry of $(M,g_0)$. Notice that by \cite[Proposition 2]{FRobert} we can define an equivalent norm for $H^{p} (M )$ 
\begin{equation}\label{eqNormH^p}
\|u\|_{H^p(M )}^2 = \int_M |(-\Delta_0)^{p/2} u|^2 \dvg + \int_M u^2 \dvg.
\end{equation}
However, in view of \eqref{eqBoundForIntegralOfuP_ou}, we have another convenient and equivalent norm on $H^{n/2} (M )$ 
\begin{equation}\label{eqNormH^n/2}
\| u\| _{H^{n/2}(M )}^2=\int_Mu\cdot \Po u \dvg +\int_Mu^2 \dvg .
\end{equation}
Now, we define our negative gradient flow for the functional \eqref{functional} as follows
\begin{equation}\label{metricflow}
\partial_t g(t)=-2(Q_{g(t)}-\lambda (t) f)g(t),
\end{equation}
where the parameter $\lambda$ is chosen to fix the quantity $\int_Mfd\mu_{g(t)}$ for all time. A direct calculation for $\partial_t \int_Mfd\mu_{g(t)} = 0$ shows that
\begin{eqnarray}\label{derivativeofintegraloff}
\int_Mf(\lambda (t) f-Q_{g(t)})d\mu_{g(t)} = 0.
\end{eqnarray}
Hence, $\lambda$ is given by
\begin{equation}\label{lambda}
 \lambda (t) =\frac{\int_MfQ_{g(t)}d\mu_{g(t)}}{\int_Mf^2d\mu_{g(t)}}.
\end{equation}
One may easily find that the fundamental difference between our flow \eqref{metricflow} and Brendle's flow \eqref{flowbrendle} is the choice of parameter $\lambda(t)$. Brendle uses $\lambda (t) =\int_M Q_0 \dvg / \int_M f \dvg$ to keep the volume $\int_M d\mu_{g(t)}$ unchanged, while we use ours to keep the quantity $\int_M f d\mu_{g(t)}$ unchanged. As a result, given an initial metric $g(0)$ one has
\begin{equation}\label{FixIntoff}
 \int_Mfd\mu_{g(t)}=\int_Mfd\mu_{g(0)} \quad \text{for all } t>0.
\end{equation}
Since the evolution equation \eqref{metricflow} preserves the conformal structure, we may write $g(t) = e^{2u(t)} g_0$ for some real-valued function $u(t)$. Then, the evolution equation for metrics $g(t)$ can be transformed into one for $u(t)$, namely
\begin{equation}\label{eqFlow}
\left\{
\begin{aligned}
\frac{\partial u}{\partial t} &=\lambda (t) f - Q_{g(t)} & & \text{ for } t > 0,\\
u(0) &= u_0, & &
\end{aligned}
\right.
\tag{F}
\end{equation} 
where and throughout this paper, by the initial data $u_0$ we mean a smooth function chosen in such a way that 
\begin{equation} \label{Initialdata}
\left.
\begin{aligned}
& u_0\in Y &\mbox{if}~\int_M Q_0 \dvg \leqslant (n-1)!\vol(\mathbb S^n), \\
& u_0\in C^\infty(M) &\mathrm{ if }~\int_M Q_0 \dvg>(n-1)!\vol(\mathbb S^n),
\end{aligned}
\right\}
\end{equation}
where
\[
Y=:\Big\{u\in H^2(M): \int_Mfe^{nu} \dvg =\int_MQ_0 \dvg \Big\}.
\]

We are ready to state our first result concerning the long-time existence and convergence of the flow \eqref{eqFlow} in the \textit{non-critical} case. In what follows, we set $h^+ = \max \{ h, 0\}$ and $h^- = \max\{-h, 0\}$ for any function $h$ on $M$. 

\begin{theorem}\label{Noncriticalcase}
Let $(M, g_0)$ be a compact, oriented $n$-dimensional Riemannian manifold with $n$ even. Assume that the GJMS operator $\Po $ is positive with kernel consisting of constant functions and the metric $g_0$ satisfies
\[
\int_M Q_0 \dvg \ne k(n-1)!\vol(\mathbb S^n) \quad \text{ for all } k \in \{ 1,2,...\}.
\]
Let the initial data $u_0$ be satisfied \eqref{Initialdata}. We also let $f$ be a non-trivial smooth function $f$ satisfying \eqref{NecessaryCondition}. In addition, we further assume that, in the negative case, there are positive constants $C_0$ depending on $f^-$ and $(M,g_0)$ and $\tau$ depending only on $(M, g_0)$ such that
\[
\sup_M f^+ \leqslant C_0 \exp \big( -\tau \|u_0\|_{H^{n/2}(M)}^2 \big)
\] 
and that $f > 0$ in the supercritical case. Then, for some $u_0$, the flow \eqref{eqFlow} has a smooth solution on $[0,+\infty)$ and converges sequentially in the sense that there exists a smooth function $u_\infty$ and a real number $\lambda_\infty$ such that for a suitable time sequence $(t_j)_j$ with $t_j\rightarrow+\infty$ as $j\rightarrow+\infty$, there hold
\begin{enumerate}[label=\rm (\roman*)]
 \item $ \| u(t_j)-u_\infty \| _{C^\infty(M )}\rightarrow0$, 
 \item $|\lambda(t_j)-\lambda_\infty|\rightarrow0$, and
 \item $ \| Q_{g(t_j)}-\lambda_\infty f \| _{C^\infty(M )}\rightarrow0$
\end{enumerate} 
 as $j\rightarrow+\infty$. Furthermore, there holds $\lambda_\infty=1$ in the subcritical but non-null case. 
\end{theorem}
\begin{remark}
We will actually prove that the initial data $u_0$ can be freely chosen in $Y$ when the subcritical case is considered. While only suitable $u_0$ can guarantee the convergence of the flow in the supercritical case. 
\end{remark}

Now, let us turn to the critical case, namely $\int_M Q_0 \dvg = (n-1)!\vol(\mathbb S^n)$. In this scenario, we only consider the underlying manifold to be the standard $n$-sphere $\mathbb S^n$, equipped with the standard metric $g_{\mathbb S^n}$. Also, we let $G$ be a group of isometries of $\mathbb S^n$. Then a function $f$ is said to be $G$-invariant if it satisfies
\[
f(\sigma x)=f(x) \quad\mbox{for all }~\sigma\in G~\mbox{and}~x\in \mathbb S^n.
\]
Furthermore, we say that a conformal metric $g$ to $g_{\mathbb S^n}$ is $G$-invariant if $u$ is a $G$-invariant function when $g$ is written as $g=e^{2u}g_{\mathbb S^n}$. Let $\Sigma$ be the set of fixed points of $G$, that is,
\[
\Sigma=\{x\in \mathbb S^n: \sigma x=x~~\mbox{for all}~~\sigma\in G\}.
\]
Inspired by the result of \cite[Theorem 3]{bfr-04} concerning the prescribing scalar curvature on the 2-sphere, our result for the prescribing $Q$-curvature on the $n$-sphere is as follows.

\begin{theorem}\label{Criticalcase}
Let $f$ be a smooth and $G$-invariant function on $ \mathbb S^n$ with $\sup_{x\in \mathbb S^n}f>0$. Also, let $u_0\in Y$ be a $G$-invariant initial data. If 
\begin{enumerate}[label=\rm (\alph*)]
 \item either $\Sigma=\emptyset$ 
 \item or
 \begin{equation} \label{Initialdata1}
 \sup_{x\in\Sigma}f(x)\leqslant(n-1)!\exp\Big(-\frac{\Escr[u_0]}{(n-1)!\vol(\mathbb S^n)}\Big),
 \end{equation}
\end{enumerate} 
then the flow \eqref{eqFlow} has a solution which is defined to all time. In addition, there exists a smooth function $u_\infty$ such that, for a suitable time sequence $(t_j)_j$ with $t_j\rightarrow+\infty$ as $j\rightarrow+\infty$, there hold 
\begin{enumerate}[label=\rm (\roman*)]
 \item $ \| u(t_j)-u_\infty \| _{C^\infty(M )}\rightarrow0$, 
 \item $|\lambda(t_j)-1|\rightarrow0$, and
 \item $ \| Q_{g(t_j)}-f \| _{C^\infty(M )}\rightarrow0$
\end{enumerate} 
 as $j\rightarrow+\infty$. 
\end{theorem}
 
Note that when $M = \mathbb S^n$, in addition to the first condition in \eqref{NecessaryCondition} there is a further necessary condition which states that if \eqref{prescribeqcurvatureequation} has a solution $u$, then $f$ must satisfy
\begin{equation}\label{KW-obstruction}
\int_{\mathbb S^n} \langle \nabla f , \nabla x_j \rangle e^{nu}\dvSn =0
\end{equation}
for all $j=1,2,...,n+1$. The above condition is essentially due to Chang and Yang \cite[page 205]{ChangYang1992}, which is a higher dimensional analogue of the well-known obstruction due to Kazdan and Warner.
 
\begin{remark}
In view of Theorems \ref{Noncriticalcase} and \ref{Criticalcase} above, we do not have the uniform but sequential convergence of the flow \eqref{eqFlow}. This is supported by the fact that ``bubbling'' phenomena of the flow \eqref{eqFlow} may appear, resulting in many different sub-sequential limits of the flow as $t\rightarrow+\infty$. For interested readers, one may refer to the recent work \cite{NZ} on such a situation on $4$-manifolds in the null case, which is a higher-order analogue of a similar situation studied by Struwe \cite{Struwe} for the Gaussian curvature flow on the two dimensional torus.
\end{remark}
 
Although the uniform convergence in time $t$ is not guaranteed in general, it indeed holds true in some special cases. 
 
\begin{theorem}\label{Uniformconvergence}
Assume all hypotheses in Theorems \ref{Noncriticalcase} and \ref{Criticalcase}. If 
\begin{enumerate}[label=\rm (\alph*)]
 \item either $\int_MQ_0 \dvg\ne 0$ and the problem \eqref{prescribeqcurvatureequation} has unique solution
 \item or $\int_MQ_0\dvg <0$ and $f\leqslant0$ everywhere,
\end{enumerate}
then all the convergences in Theorems \ref{Noncriticalcase} and \ref{Criticalcase} are uniform in time, namely
\begin{enumerate}[label=\rm (\roman*)]
 \item $ \| u(t)-u_\infty \| _{C^\infty(M )}\rightarrow0$, 
 \item $|\lambda(t)-\lambda_\infty|\rightarrow0$, and
 \item $ \| Q_{g(t)}-\lambda_\infty f \| _{C^\infty(M )}\rightarrow0$
\end{enumerate} 
as $t\rightarrow+\infty$, and the convergence corresponding to the case (b) is exponentially fast. Finally, there holds $\lambda_\infty=1$ in the subcritical but non-null case or when $M=\mathbb S^n$. 
\end{theorem}

A direct consequence of the convergence of the flow \eqref{eqFlow} is the existence of solutions to the prescribed $Q$-curvature equation \eqref{prescribeqcurvatureequation}.

\begin{corollary}\label{maincorollary}
Let $(M, g_0)$ be a compact, oriented $n$-dimensional Riemannian manifold with $n$ even. Assume that the GJMS operator $\Po $ is positive with kernel consisting of constant functions. Let $f$ be a non-constant, smooth function on $M$. Then, we have the following claims:
\begin{enumerate}[label=\rm (\roman*)]
\item Assume that $$\int_MQ_0 \dvg<0.$$ 
Then there is a positive constant $C_0$ depending on $f^-$ and $(M,g_0)$ such that there exists a conformal metric to $g_0$ with $Q$-curvature $f$ whenever 
\[
\sup_M f^+ \leqslant C_0.
\]

\item Assume that $$\int_MQ_0 \dvg=0.$$ 
Then there exists a conformal metric to $g_0$ with $Q$-curvature $\alpha f$ for $\alpha\in\{-1,1\}$ provided $f$ is sign-changing with $\int_Mfd\mu_{g_0}\neq 0$. Moreover, if 
\[
\int_M f \dvg < 0,
\]
then there exists a conformal metric to $g_0$ with $Q$-curvature $f$, namely $\alpha =1$.

\item Assume that
\[
0<\int_MQ_0 \dvg < (n-1)!\vol(\mathbb S^n).
\]
Then there exists a conformal metric to $g_0$ with $Q$-curvature $f$ if and only if 
$$\sup_M f >0.$$

\item Let $M=\mathbb S^n$ and $f$ be a $G$-invariant function on $\mathbb S^n$ with $\sup_{\mathbb S^n}f>0$. If 
\begin{enumerate}[label=\rm (\alph*)]
 \item either $\Sigma=\emptyset$ 
 \item or there exist $y_0\in\Sigma$ and $r_0>0$ such that 
\end{enumerate}
\begin{equation}\label{SupfonSigma}
\sup_{x\in\Sigma}f(x)\leqslant \frac 1{\vol (\mathbb S^n)}\max\Big\{ \int_{\mathbb S^n}f\circ\phi_{y_0,r_0}\dvSn ,0\Big\},
\end{equation}
then there exists a $G$-invariant conformal metric to $g_0$ with $Q$-curvature $f$. Here $\phi_{y,r}: \mathbb S^n \to \mathbb S^n$ is a conformal diffeomorphism given by
\[
\phi_{y, r} (x) = \pi_y^{-1} ( r \pi_y(x) ),
\] 
where $r>0$, $y \in \mathbb S^n$, and $\pi_y : \mathbb S^n \to \mathbb R^n$ is the stereographic projection from the north pole $y$ to the equatorial plane of $\mathbb S^n$. In particular, upon choosing $r=1$ we conclude that if
$$
\sup_{x\in\Sigma}f(x)\leqslant \frac 1{\vol (\mathbb S^n)}\max\Big\{ \int_{\mathbb S^n}f\dvSn ,0\Big\},
$$
then there exists a $G$-invariant conformal metric to $g_0$ with $Q$-curvature $f$.

\item Assume that $f>0$ and that 
\[
(n-1)!\vol(\mathbb S^n) <\int_MQ_0 \dvg \neq k(n-1)!\vol(\mathbb S^n),~~k=\{2, 3, \dots\}.
\]
Then there always exists a conformal metric to $g_0$ with $Q$-curvature $f$.
\end{enumerate}
\end{corollary}

\begin{remark}
We have the following remarks:
\begin{enumerate}
 \item Our results in Corollary \ref{maincorollary} shares some similarities with other existing results in the literature such as those in \cite{bfr09, fr12} for manifolds of even dimension. Although the configurations used in these two works are similar to us with a weaker condition on the operator $\Po$, the prescribed function $f$ has a fixed sign. The work \cite{bfr09} also considered the supercritical case. However, some sort of symmetric assumption on $f$ has been made in \cite{bfr09} while our $f$ is a generic positive smooth function. Therefore, more or less, Corollary \ref{maincorollary} generalizes certain results in \cite{bfr09, fr12} in several directions.
 
 \item In view of the result in Part (ii), there is a multiple constant $\alpha$ which cannot be removed in general as the sign of $\lambda_\infty$ cannot be determined. However, if limiting ourselves to $4$-mainfolds, then Baird et al. were able to determine the sign of $\lambda_\infty$ under an extra condition; see \cite[Corollary 2.4]{bfr}. Unfortunately, their proof seems to be effective only for $4$-dimension. Later on, Ge and Xu \cite{gx2008} used a variational approach to replace the extra condition in \cite{bfr} by a natural condition $\int_M f \dvg < 0$. In Appendix \ref{apd-GeXu}, we modify the argument of Ge and Xu to give a slightly different proof of existence for this case. Interestingly, this new proof allows us to recover the result of Ge and Xu by using our flow.
 
 \item In Part (iv), the conformal diffeomorphism $\phi_{y,r}$ is well-studied; see \cite[Section 5]{YYLi}. Also, the two inequalities in Part (iv) are sharp. Indeed, let us show the sharpness of 
\[
\sup_{x\in\Sigma}f(x)\leqslant \frac 1{\vol (\mathbb S^n)}\max\Big\{ \int_{\mathbb S^n}f\dvSn ,0\Big\}.
\]
Let 
$$f(x)= \varepsilon x_{n+1}$$ and consider $G$ the group of isometries which fixes the north and south poles of $\mathbb S^n$, namely the two points $(0,...,0, \pm 1) \in \R^{n+1}$. Then the function $f$ is $G$-invariant and
\[
\sup_{x\in\Sigma}f(x) = \varepsilon.
\]
On one hand, because the function $f$ is odd with respect to the origin, we must have
\[
\int_{\mathbb S^n}f\dvSn = 0.
\] 
On the other hand, the function $f$ is not the $Q$-curvature of any metric conformal to $g_{\mathbb S^n}$ by the obstruction \eqref{KW-obstruction} of Kazdan--Warner.
 
 \item The result in Part (v) can be regarded as a generalization of the existence result obtained in \cite{fr18}. We note that in the critical and supercritical cases, the analysis is delicate since blow-up phenomena may occur; see \cite{fr18}. 
\end{enumerate}
\end{remark}

To close this section, we would like to mention the organization of the paper. Section \ref{sec-FlowEquationAndEnergyFunctional} is devoted to a flow of conformal factor $u$ and the energy functional associated with the flow. Sections \ref{sec-BoundInH^(n/2)} and \ref{sec-BoundInH^n} are devoted to $H^{n/2}$- and $H^n$-bounds of $u(t)$ in the maximal interval of existence $[0, T)$. Putting these bounds together and by performing integral estimates, we obtain $H^{2k}$-bounds for $u(t)$ in Section \ref{sec-BoundInH^2k} for any $k >2$. This helps us to conclude that the flow is defined to all time. To study the convergence of the flow, uniform boundedness of the flow is required. We spend Section \ref{sec-GlobalBoundedness} to prove certain global boundedness of the flow. Putting all these preparation together, we study convergence of the flow in Section \ref{sec-Convergence}. Proofs of Theorems \ref{Noncriticalcase} and \ref{Criticalcase} and Corollary \ref{maincorollary} are put in the final section; see Section \ref{sec-ProofsOfMain}

\tableofcontents


\section{The flow equation and its energy functional}
\label{sec-FlowEquationAndEnergyFunctional}

The purpose of this section is to derive a flow equation for the conformal factor $u$. Using this flow equation, we shall show that the energy functional $\Escr$ is non-increasing along the flow. Other properties of the energy functional are also presented. Thanks to \eqref{qcurvartureequation}, we also obtain
\[
\partial_ t u(t) = - \P_{g(t)} u-e^{-nu}Q_0+\lambda (t) f
\]
for any $t > 0$ in the interval of existence, which is non-empty. By an easy calculation, we also obtain the formulas for $\partial_t Q_{g(t)}$ and $\partial_t \lambda (t)$ along the flow as shown below.

\begin{lemma} \label{flowequationsofuandQ}
Suppose that $u(t)$ is a solution to the flow \eqref{eqFlow}. Then, we have
\[
\partial_ t Q_{g(t)}=-\P_{g(t)}(Q_{g(t)}-\lambda (t) f)+nQ_{g(t)}(Q_{g(t)}-\lambda (t) f)
\]
and
\[
\lambda' (t)=\Big (\int_Mf^2d\mu_{g(t)} \Big)^{-1} \int_M \big(\P_{g(t)} f-n\lambda (t) f^2 \big)\big(\lambda (t) f-Q_{g(t)}\big)d\mu_{g(t)},
\]
where, as always, $\P_{g(t)} =e^{-nu}\Po $ is the GJMS operator with respective to the conformal metric $g(t)$.
\end{lemma}

\begin{proof}
 Since this is just a routine calculation, we leave the proof to the reader.
\end{proof}

\begin{lemma}\label{energydecay}
 Suppose that $u(t)$ is a solution to the flow \eqref{eqFlow}. Then, one has 
\[
\frac{d\Escr [u]}{dt}=-n\int_M(Q_{g(t)}-\lambda (t) f)^2d\mu_{g(t)}.
\]
In particular, the energy functional $\Escr [u]$ is non-increasing along the flow.
\end{lemma}

\begin{proof}
 It follows from the flow equation \eqref{eqFlow}, \eqref{qcurvartureequation}, and \eqref{derivativeofintegraloff} that
 \begin{align*}
\frac{d\Escr [u]}{dt} =&\frac{n}{2}\int_Mu_t \cdot \Po u+u\cdot \Po u_t \dvg +n\int_MQ_0u_t \dvg \\
 =&-n\int_M(Q_{g(t)}-\lambda (t)f)Q_{g(t)}d\mu_{g(t)}\\
 =&-n\int_M(Q_{g(t)}-\lambda (t)f)^2d\mu_{g(t)}+n\lambda(t) \int_Mf(\lambda (t)f-Q_{g(t)})d\mu_{g(t)}\\
 =&-n\int_M(Q_{g(t)}-\lambda (t) f)^2d\mu_{g(t)}.
 \end{align*}
The proof is complete.
\end{proof}

Before going further, let us recall some important inequalities that will be used in the paper. First, we recall the Gagliardo--Nirenberg interpolation inequality
\begin{equation}\label{eqGagliardoNirenbergInequality}
\int_M |\nabla^j \varphi|^p \dvg \leqslant \Cscr^p \Big( \int_M |\nabla^m \varphi|^r \dvg \Big) ^{ps/r} \Big( \int_M | \varphi |^q \dvg\Big) ^{p(1-s)/q}
\end{equation}
for some $\Cscr>0$ with $0<s< 1 \leqslant q, r \leqslant +\infty$ satisfying $1/p = j/n + (1/r - m/n) \alpha + (1-s)/q$.

Since the operator $\Po $ is self-adjoint and positive with kernel consisting of constant functions, we can apply Adams' inequality \cite[Theorem 2]{adams} to get
\begin{equation}\label{eqAdamsTypeInequality}
\int_M \exp\Big( \frac{n!\vol(\mathbb S^n)}{2}\frac{(u-{\overline u})^2}{\int_Mu\cdot \Po u \dvg } \Big) \dvg \leqslant \Cscr_A
\end{equation}
for some constant $\Cscr_A>0$. A detailed explanation for the validity of \eqref{eqAdamsTypeInequality} can be found, for example, in \cite[page 330]{br}. As a consequence of \eqref{eqAdamsTypeInequality}, we obtain the following Trudinger-type inequality
\[
\begin{split}
\int_M e^{\alpha(u-{\overline u}) } \dvg \leqslant & \Cscr_A \exp \Big( \frac{\alpha^2}{2n!\vol(\mathbb S^n)}\int_Mu\cdot \Po u \dvg \Big)
\end{split}
\]
for all real number $\alpha$. An equivalent form of the Trudinger-type inequality is the following
\begin{equation}\label{TrudingerInequality}
\begin{split}
\int_M e^{\alpha u} \dvg \leqslant & \Cscr_A \exp \Big( \frac{\alpha^2}{2n!\vol(\mathbb S^n)}\int_Mu\cdot \Po u \dvg + \frac{\alpha}{\vol (M)} \int_M u \dvg \Big).
\end{split}
\end{equation}
Since
\[
\frac{\alpha}{\vol (M)} \int_M u \dvg
\leqslant \frac{\alpha^2}{2n!\vol(\mathbb S^n)} \int_M u^2 \dvg +\frac { n! \vol(\mathbb S^n)}{2 \vol (M)},
\]
we obtain from \eqref{TrudingerInequality} the following
\begin{equation}\label{TrudingerInequalityWithHnorm}
\int_M e^{\alpha u} \dvg \leqslant \Cscr_A \exp \Big(\frac { n! \vol(\mathbb S^n)}{2 \vol (M)} \Big) 
\exp \Big( \frac{\alpha^2}{2n!\vol(\mathbb S^n)} \|u\|_{H^{n/2}(M)}^2\Big).
\end{equation}
Finally, when working on $ \mathbb S^n$, instead of using Adams' inequality \eqref{eqAdamsTypeInequality}, we shall use the following sharp Beckner's inequality
\begin{equation}\label{BecknerInequality}
\begin{split}
\frac 1{\vol(\mathbb S^n)} \int_{ \mathbb S^n} e^{nu} \dvg \leqslant \exp \Big( \frac n{2(n-1)!\vol(\mathbb S^n)} \int_{ \mathbb S^n} u\cdot \Po u\dvg + \frac n{\vol(\mathbb S^n)} \int_{ \mathbb S^n} u \dvg \Big);
\end{split}
\end{equation}
see \cite[Eq. (4.1'')]{ChangYang1992}. 
Similar to Trudinger's inequality \eqref{TrudingerInequalityWithHnorm}, we obtain the following Beckner inequality in terms of norm
\begin{equation}\label{BecknerInequalityWithHnorm}
\frac 1{\vol(\mathbb S^n)} \int_{ \mathbb S^n}  e^{n u} \dvg \leqslant \exp \Big(\frac { n! }{2 } \Big) 
\exp \Big( \frac{n}{2(n-1)!\vol(\mathbb S^n)} \|u\|_{H^{n/2}(\mathbb S^n)}^2\Big).
\end{equation}
By writing $\alpha u =n (\alpha u/n)$, we obtain the following Beckner-type inequality
\[
\begin{split}
\frac 1{\vol(\mathbb S^n)} \int_{ \mathbb S^n} e^{\alpha u} \dvg \leqslant \exp \Big( \frac {\alpha^2}{2 n!\vol(\mathbb S^n)} \int_{ \mathbb S^n} u\cdot \Po u\dvg + \frac \alpha{\vol(\mathbb S^n)} \int_{ \mathbb S^n} u \dvg \Big),
\end{split}
\]
for any $\alpha>0$, which takes a similar form as that of \eqref{TrudingerInequality}. However, we do not use the Beckner-type inequality in this paper.


\section{Boundedness of \texorpdfstring{$u(t)$}{Lg} in \texorpdfstring{$H^{n/2}(M)$}{Lg} for $0 \leqslant t < T$}
\label{sec-BoundInH^(n/2)}
 
In this section, we aim to show that $u(t)$ is bounded in $H^{n/2}(M,g_0)$ for all $t$ in the maximal interval of existence. We will split our argument into four cases depending on the sign of $\int_M Q_0 \dvg$. As we shall see later, such a $H^{n/2}$-bound for $u(t)$ is uniform in the non-supercritical case while the bound depends on the maximal time of existence in the supercritical case; see Proposition \ref{Hn/2bd-supercritical}.

\subsection{The negative case}
\label{subsec-BoundInH^(n/2)TheCaseIntegralQ<0}

To achieve our goal, we first need an analogue of \cite[Lemma 4.1]{bfr} whose proof is provided in Appendix \ref{apd-nonconcentration} for completeness.

\begin{lemma}\label{nonconcentration}
 Let $K$ be a measurable subset of $M$ with $\vol (K)>0$. Then there exist two constants $\alpha>1$ depending on $M$ and $g_0$ and $C_K>1$ depending on $M$, $g_0$, and $\vol (K)$ such that
\[
\int_Me^{nu}\dvg \leqslant C_K \exp\Big( \alpha\| u_0\| ^2_{H^{n/2}(M)} \Big) \max\Big\{\Big(\int_Ke^{nu}\dvg \Big)^\alpha,1\Big\}.
\]
\end{lemma}

With help of this lemma, we are able to obtain a uniform bound for $\vol(M,g(t))$ along the flow.

\begin{lemma}\label{lem-upperboundofvolume}
 Suppose the flow \eqref{eqFlow} with initial data $u_0$ is defined on the maximal interval of existence $[0,T)$ for some $T>0$. Also, assume that there exists a constant $C_0$ depending on $f^-$ and $(M,g_0)$, and a constant $\tau$ depending only on $(M,g_0)$ such that 
\[
\sup_Mf^+\leqslant C_0 \exp \big( -\tau \|u_0\|_{H^{n/2}(M)}^2 \big).
\] 
Then, there exists a uniform constant $\gamma>0$ such that
\[
\int_Me^{nu (t)}\dvg \leqslant \gamma
\]
 for all $t\in[0,T)$.
\end{lemma}

\begin{proof}
Define
\[
K=\Big\{x\in M : f(x)\leqslant \frac12\inf_Mf\Big\}.
\]
Notice that $K\neq\emptyset$ due to the fact that $\inf_Mf<0$. Hence, by the continuity of $f$, we conclude $\vol(K)>0$. Since $u_0\in Y$ and $f = f^+ - f^-$, we have
\begin{align*}
 -\int_MQ_0\dvg & =-\int_Mfe^{nu_0}\dvg 
\leqslant (\sup_M f^-)\int_Me^{nu_0}\dvg ,
\end{align*}
which implies that
\[
\frac 1 {-\sup_Mf^-} \int_MQ_0\dvg \leqslant \int_Me^{nu_0}\dvg .
\]
From \eqref{TrudingerInequalityWithHnorm} we have
\begin{equation}\label{initialvolume}
\begin{split}
 \int_Me^{nu_0}\dvg
 \leqslant &C \exp \big( C \|u_0 \|^2_{H^{n/2}(M)} \big),
\end{split}
 \end{equation}
 where $C> 1$ is a constant depending on $M$ and $g_0$. (Here, we can freely assume that $C>1$.) Therefore, we get
 \begin{equation}\label{hn2normofu0}
 \frac 1{-\sup_Mf^-}\int_MQ_0\dvg \leqslant C \exp \big( C\| u_0\| ^2_{H^{n/2}(M)} \big).
 \end{equation}
To conclude the proof, we shall show that the inequality \eqref{initialvolume} is stable along the flow in the sense that $\int_M e^{n u} \dvg$ is bounded from above by some multiple of the right hand side of \eqref{initialvolume}. 
With the set $K$ chosen as above, let $C_K$ and $\alpha$ be the constants appearing in Proposition \ref{nonconcentration}. Set
\[
\gamma=2C_K(8C )^\alpha \exp \big( (C +1)\alpha\| u_0\| ^2_{H^{n/2}(M)} \big).
\]
As we shall see later, this constant is exactly the upper bound for $\int_M e^{n u} \dvg$ . Now we choose the two constants $\tau$ and $C_0$ in the lemma in such a way that
\[
C_0 \exp \big( -\tau\| u_0\| ^2_{H^{n/2}(M)} \big)= 2 C \gamma^{-1} \exp \big(C \| u_0\|_{H^{n/2}(M)}^2 \big)\sup_Mf^-,
\]
which can be computed precisely to get
\[
C_0=8^{-\alpha}C_K^{-1}C^{1-\alpha}\sup_Mf^- .
\]
and
\[
\tau= \alpha(C+1)-C.
\]
Our choice of $C_0$ and $\tau$ is to ensure that
\begin{equation}\label{bdforf^+overf^-}
\gamma \frac{\sup_M f^+}{ \sup_M f^- } \leqslant 2 C \exp (C \| u_0\|_{H^{n/2}(M)}^2).
\end{equation}
Now, in the rest of the proof, we claim that
 \begin{equation}\label{upperboundofvolume1}
 \int_Me^{nu(t)}\dvg \leqslant \gamma,
 \end{equation}
 for all $t\in[0,T)$. To see this, we let
\[
I=\Big\{t\in[0,T): \int_Me^{nu(s)}\dvg \leqslant \gamma ~~\mbox{ for all }~~s\in[0,t]\Big\}.
\]
From \eqref{initialvolume}, we have that $0\in I$. Set $\beta=\sup I$ and suppose that $\beta<T$. Then by continuity, we have
 \begin{equation}
 \label{volumeattimebeta}
 \int_Me^{nu(\beta)}\dvg =\gamma.
 \end{equation}
 Notice that $u(t)\in Y$ for all $t\in[0,T)$. Hence, in particular, $u(\beta)\in Y$. In the following, we split our argument into two cases.

\noindent{\bf Case 1}. Suppose that
\[
\int_Mf^+e^{nu(\beta)}\dvg \leqslant \frac12\int_Mf^-e^{nu(\beta)}\dvg.
\]
Then the fact that $u(\beta)\in Y$ gives
 \begin{align*}
 -\int_MQ_0\dvg 
 & =\int_Mf^-e^{nu(\beta)}\dvg -\int_Mf^+e^{nu(\beta)}\dvg \geqslant \frac12\int_Mf^-e^{nu(\beta)}\dvg .
 \end{align*}
 This together with the fact that $2 f^-\geqslant \sup_Mf^-$ for all $x\in K$ yields
\begin{equation}
 \label{volumeattimebeta1}
 \int_Ke^{nu(\beta)}\dvg \leqslant \frac{4}{-\sup_Mf^-}\int_MQ_0\dvg .
\end{equation}
Combining \eqref{hn2normofu0} and \eqref{volumeattimebeta1} gives
\[
\int_Ke^{nu(\beta)}\dvg \leqslant 8C \exp \big( C\| u_0\| ^2_{H^{n/2}(M)} \big).
\]
From Lemma \ref{nonconcentration}, it follows that
\begin{align*}
 \int_Me^{nu(\beta)}\dvg \leqslant &C_K \exp \big( \alpha\| u_0\| ^2_{H^{n/2}(M)} \big) \max\Big\{\Big(\int_Ke^{nu(\beta)}\dvg \Big)^\alpha,1\Big\}\\
\leqslant &C_K \exp \big( \alpha\| u_0\| ^2_{H^{n/2}(M)} \big) \max \Big\{ (8C)^\alpha \exp \big( \alpha C\| u_0\| ^2_{H^{n/2}(M)} \big) ,1 \Big\}\\
\leqslant &C_K(8C)^\alpha \exp \big( (C+1)\alpha\| u_0\| ^2_{H^{n/2}(M)} \big) \\
= &\gamma/2,
\end{align*}
which contradicts \eqref{volumeattimebeta}.

\noindent{\bf Case 2}: Suppose that
\[
\int_Mf^+e^{nu(\beta)}\dvg >\frac12\int_Mf^-e^{nu(\beta)}\dvg.
\]
This situation corresponds to the case when $f^+$ plays some role and therefore we need certain smallness of $\sup_M f$. Again by using the fact that $2f^-(x)\geqslant- \inf_Mf$ for all $x\in K$ and \eqref{volumeattimebeta}, we can estimate
\begin{align*}
 \int_Ke^{nu(\beta)}\dvg \leqslant &\frac{2}{\sup_Mf^-}\int_Mf^-e^{nu(\beta)}\dvg \\
 \leqslant &\frac{4}{\sup_Mf^-}\int_Mf^+e^{nu(\beta)}\dvg \leqslant \frac{4\gamma\sup_Mf^+}{\sup_Mf^-}.
\end{align*}
By \eqref{bdforf^+overf^-}, we then have
\[
\int_Me^{nu(\beta)}\dvg \leqslant 8C \exp \big( C\| u_0\| ^2_{H^{n/2}(M)} \big).
\]
Thus, by arguing as in Case 1, we obtain a contradiction once more. Hence, the claim \eqref{upperboundofvolume1} holds and the lemma is proved.
\end{proof} 

\begin{lemma}
Suppose that the flow \eqref{eqFlow} is defined on $[0,T)$. Then, there exists a uniform constant $C>0$ such that
\[
\| u(t)\| _{H^{n/2}(M)}\leqslant C
\]
for all $t\in[0,T)$.
\end{lemma}

\begin{proof}
From Lemma \ref{lem-upperboundofvolume} and Jensen's inequality, it follows that
 \begin{equation}
 \label{upperboundofbaru1}
 {\overline u}\leqslant \frac1n\log \Big( \frac{1}{\vol(M)} \int_Me^{nu}\dvg \Big) \leqslant \frac1n\log\frac{\gamma}{\vol(M)}.
 \end{equation}
 By Lemma \ref{energydecay}, we get
 \begin{equation}
 \label{energyupperbound}
\frac n2 \int_Mu\cdot \Po u\dvg +n\int_MQ_0(u-\overline{u})\dvg +n{\overline u}\int_MQ_0\dvg \leqslant \Escr [u_0].
 \end{equation}
Substituting \eqref{upperboundofintegralofQ_0(u-baru)1} into \eqref{energyupperbound} yields
\begin{equation}
 \label{upperboundofintegralofuP_0u1}
\frac n4 \int_Mu\cdot \Po u\dvg+n{\overline u}\int_MQ_0\dvg \leqslant \Escr [u_0]+\frac{n}{\lambda_1}\int_MQ_0^2\dvg .
 \end{equation}
 By \eqref{upperboundofbaru1} and \eqref{upperboundofintegralofuP_0u1}, we have
 \begin{equation}
 \label{upperboundofintegralofup0u1}
 \int_Mu\cdot \Po u\dvg \leqslant \underbrace {-\frac 4 n\log\frac{\gamma}{\vol(M)}\int_MQ_0\dvg +\frac{4\Escr [u_0]}{n}+\frac{4}{\lambda_1}\int_MQ_0^2\dvg }_{=:C_1}.
 \end{equation}
By the Poincar\'e-type inequality \eqref{eqPositivityOfP_0}, we have
\[
\int_M(u-{\overline u})^2\dvg \leqslant \frac{1}{\lambda_1}\int_Mu\cdot \Po u\dvg \leqslant \frac{C_1}{\lambda_1}.
\]
To finish the proof of the lemma, it remains to get a lower bound for ${\overline u}$. Indeed, from the positivity of $\Po $, \eqref{upperboundofintegralofuP_0u1}, and $\int_MQ_0\dvg <0$, we know that
\[
{\overline u}\geqslant\frac{(n/\lambda_1) \int_MQ_0^2\dvg +\Escr [u_0]}{n\int_MQ_0\dvg }.
\]
We thus conclude that there exists a uniform constant $C_2>0$ such that
 \begin{equation}
 \label{upperboundofl2normofu1}
 \int_Mu^2\dvg \leqslant C_2 .
 \end{equation}
Combining \eqref{upperboundofintegralofup0u1} and \eqref{upperboundofl2normofu1} yields the assertion.
\end{proof}

\subsection{The null case}
\label{subsec-BoundInH^(n/2)TheCaseIntegralQ=0}

In this case, because $u_0\in Y$ and $\partial_t \int_Mfd\mu_{g(t)}=0$, we conclude that $\int_Mfd\mu_{g(t)}=0 $ for all $t\geqslant 0$. Furthermore, in the null case, we always have $\int_M Q_{g(t)} d\mu_{g(t)} = 0$ and $\int_M Q_0 \dvg = 0$. Using these properties, we obtain the following result.

\begin{lemma}\label{volumepreserve}
Along the flow \eqref{eqFlow}, the volume of $g(t)$ denoted by $\vol (M,g(t))$ is preserved. Hence, if we set $V^0=\int_Me^{nu_0}\dvg $, then $\int_Me^{nu(t)}\dvg =V^0$ for all $t$.
\end{lemma}

\begin{proof}
Clearly, we have
\begin{align*}
 \frac{d}{dt}\int_Me^{nu(t)}\dvg =&n\int_Mu_td\mu_{g(t)}=n\int_M \big( \lambda (t) f-Q_{g(t)} \big) d\mu_{g(t)} 
 =0.
\end{align*}
Hence $\int_Me^{nu(t)}\dvg$ is constant along the flow; this proves the assertion.
\end{proof}

Using the preceding lemma, we can control $u(t)$ in $H^{n/2}$-norm as shown below.

\begin{lemma}\label{lem-BoundInH^{n/2}-zero}
There exists a uniform constant $C>0$ such that
\[
\| u(t) \|_{H^{n/2}(M)}\leqslant C
\]
for all time $t$ in the maximal interval of existence.
\end{lemma}

\begin{proof}
From Lemma \ref{energydecay}, it follows that
\begin{equation}\label{nullenergyupperbound}
\frac n2\int_Mu\cdot \Po u\dvg+n\int_MQ_0(u-{\overline u})\dvg \leqslant \Escr [u_0].
\end{equation}
Using the Poincar\'e-type inequality \eqref{eqPositivityOfP_0}, we get
\[
\int_M(u-{\overline u})^2\dvg \leqslant \frac{1}{\lambda_1}\int_Mu\cdot \Po u\dvg,
\]
which together with Young's inequality implies that
\begin{equation}\label{upperboundofintegralofQ_0(u-baru)1}
\Big|n\int_MQ_0(u-{\overline u})\dvg \Big|\leqslant\frac n4\int_Mu\cdot\Po u\dvg +\frac{n}{\lambda_1}\int_MQ_0^2\dvg .
\end{equation}
By substituting \eqref{upperboundofintegralofQ_0(u-baru)1} into \eqref{nullenergyupperbound}, we have
\begin{equation}\label{nullupperboundofupu}
\int_Mu\cdot\Po u\dvg \leqslant\frac4n\Escr[u_0]+\frac{4}{\lambda_1}\int_MQ_0^2\dvg ,
\end{equation}
Hence, in view of \eqref{eqNormH^n/2}, to bound $\| u(t) \|_{H^{n/2}(M)}$, it suffices to bound $\int_M u(t) \dvg$. By Jensen's inequality and Lemma \ref{volumepreserve}, we have that
 $$e^{n{\overline u}}\leqslant \frac 1{\vol(M)} \int_Me^{nu}\dvg =\frac {V^0}{\vol(M)}.$$
Therefore, we can bound $\overline u$ from above as follows
\[
\overline u \leqslant \frac1n\log \Big(\frac{V^0}{\vol(M)}\Big).
\]
To bound $\overline u$ from below, we apply the Trudinger-type inequality \eqref{TrudingerInequality} and \eqref{nullupperboundofupu} to get
\[
\int_M \exp\big(n(u-{\overline u})\big) \dvg \leqslant \Cscr_A \exp\Big( \frac{2\Escr [u_0]}{(n-1)!\vol(\mathbb S^n)}+\frac{2n}{\lambda_1(n-1)!\vol(\mathbb S^n)}\int_MQ_0^2\dvg \Big).
\]
Since $\int_Me^{nu}\dvg =V^0$, we conclude that
\[
\exp \big(-n \overline u \big) \leqslant \frac{\Cscr_A}{V^0} \exp \Big( \frac{2\Escr [u_0]}{(n-1)!\vol(\mathbb S^n)}+\frac{2n}{\lambda_1(n-1)!\vol(\mathbb S^n)}\int_MQ_0^2\dvg \Big).
\]
Hence, we get a lower bound for $\overline u$ as follows
\[
\overline u \geqslant-\Big(\frac1n\log\Big(\frac{\Cscr_A}{V^0}\Big)+\frac{2}{\lambda_1(n-1)!\vol(\mathbb S^n)}\int_MQ_0^2\dvg +\frac{2\Escr [u_0]}{n!\vol(\mathbb S^n)}\Big).
\]
Our proof is thus complete by combining all estimates above.
\end{proof}

\subsection{The case \texorpdfstring{$0<\int_MQ_0\dvg <(n-1)!\vol(\mathbb S^n)$}{Lg}}
\label{subsec-BoundInH^(n/2)TheCaseIntegral0<Q<S^n}

Notice that $\int_Mfe^{nu}\dvg =\int_MQ_0\dvg $ for all $t\geqslant 0$. Therefore,
$$\int_MQ_0\dvg =\int_Mfe^{nu}\dvg \leqslant (\sup_Mf) \int_Me^{nu}\dvg .$$
Hence,
\begin{equation}\label{volumelowerbound}
 \int_Me^{nu}\dvg \geqslant \frac 1{\sup_Mf} \int_MQ_0\dvg .
\end{equation}
Combining \eqref{TrudingerInequality} and \eqref{volumelowerbound} gives
\begin{align*}
 \frac{n/2}{(n-1)!\vol(\mathbb S^n)} \int_M u\cdot \Po u\dvg \geqslant&\log\int_M \exp \big( n(u-{\overline u}) \big) \dvg -\log \Cscr_A\\
\geqslant&\log\Big(\frac 1{\Cscr_A\sup_M f } \int_MQ_0\dvg \Big)-n{\overline u}.
\end{align*}
Thus, we have
\begin{equation}\label{barulowerbound}
 {\overline u}\geqslant-\frac{n/2}{n!\vol(\mathbb S^n)}\int_Mu\cdot \Po u\dvg +\frac1n\log\Big(\frac 1{\Cscr_A\sup_M f } \int_MQ_0\dvg \Big).
\end{equation}
Using \eqref{barulowerbound} and Lemma \ref{energydecay}, we get
\begin{equation}\label{estimateofintegralofuP_0u}
\begin{split}
 \Escr [u_0] \geqslant& \frac n2\int_Mu\cdot \Po u\dvg +n\int_MQ_0(u-\overline{u})\dvg +n\overline{u}\int_MQ_0\dvg \\
\geqslant&\frac n2\Big(1-\frac{\int_MQ_0\dvg }{(n-1)!\vol(\mathbb S^n)}\Big)\int_Mu\cdot \Po u\dvg+n\int_MQ_0(u-\overline{u})\dvg \\
& +\int_MQ_0\dvg \log\Big(\frac 1{\Cscr_A\sup_M f } \int_MQ_0\dvg \Big).
\end{split}
\end{equation}
Since $\int_MQ_0\dvg <(n-1)!\vol(\mathbb S^n)$, by Young's inequality and \eqref{eqPositivityOfP_0} we have the following
\begin{equation}\label{upperboundofintegralofQ_0(u-baru)}
\begin{split}
\Big|n\int_MQ_0(u-\overline{u})\dvg \Big| \leqslant & \frac n4\Big(1-\frac{\int_MQ_0\dvg }{(n-1)!\vol(\mathbb S^n)}\Big)\int_Mu\cdot \Po u\dvg\\
& +\frac {n}{\lambda_1} \frac{ \int_MQ_0^2\dvg }{ 1- (\int_MQ_0\dvg \big)/\big( (n-1)!\vol(\mathbb S^n)\big)} .
\end{split}
\end{equation}
 By substituting \eqref{upperboundofintegralofQ_0(u-baru)} into \eqref{estimateofintegralofuP_0u}, we conclude that there exists a uniform constant $C_1>0$ such that
\begin{equation}\label{upperboundofintegralofup0u}
\int_Mu\cdot \Po u\dvg \leqslant C_1.
\end{equation}
Now substituting \eqref{upperboundofintegralofup0u} into \eqref{barulowerbound} yields that
\begin{equation}\label{lowerbounderofbaru}
 {\overline u}\geqslant C_2,
\end{equation}
for some uniform constant $C_2$. To obtain an upper bound for $\overline{u}$, we notice that \eqref{upperboundofintegralofQ_0(u-baru)1} still holds in this case. Hence by substituting \eqref{upperboundofintegralofQ_0(u-baru)1} into the first inequality in \eqref{estimateofintegralofuP_0u}, by the positivity of $\Po $, and by the fact that $\int_MQ_0\dvg >0$, we get
\begin{equation}\label{upperboundofbaru}
{\overline u}\leqslant \frac{(n/\lambda_1) \int_MQ_0^2\dvg +\Escr [u_0]}{n\int_MQ_0\dvg }.
\end{equation}
By plugging \eqref{upperboundofintegralofup0u}, \eqref{lowerbounderofbaru}, and \eqref{upperboundofbaru} into \eqref{eqPositivityOfP_0}, we conclude that there exists a uniform constant $C_3>0$ such that
\begin{equation}\label{upperboundofl2normofu}
 \int_Mu^2\dvg \leqslant C_3.
\end{equation}
The assertion follows by combining \eqref{upperboundofintegralofup0u} and \eqref{upperboundofl2normofu}.

\subsection{The case \texorpdfstring{$M=\mathbb S^n$}{Lg}}

Now we consider the case where $(M,g_0)$ is the standard sphere $\mathbb S^n$ equipped with the standard metric $g_{\mathbb S^n}$. In particular, one has that $Q_0 \equiv (n-1)!$. As we have already seen before, the strict inequality $\int_M Q_0 \dvg < (n-1)! \vol(\mathbb S^n)$ is crucial in Subsections \ref{subsec-BoundInH^(n/2)TheCaseIntegralQ<0}--\ref{subsec-BoundInH^(n/2)TheCaseIntegral0<Q<S^n}. However, this is no longer the case in this new setting. Therefore, to obtain the uniform boundedness of $u(t)$ in $H^{n/2}(\mathbb S^n)$, we need the following concentration-compactness lemma whose proof is provided in Appendix \ref{apd-nonconcentration}.

\begin{lemma}\label{concentrationcompactness}
Suppose that $[0,T)$ is the maximal interval of existence of the flow \eqref{eqFlow}. Then we have the following alternatives:
\begin{enumerate}[label=\rm (\roman*)]
 \item either there exists a uniform constant $C>0$ independent of $T$ such that
\[
\|u(t)\|_{H^{n/2}(\mathbb S^n)} \leqslant C
\] 
for all $t\in[0,T)$
 \item or there exists a sequence $t_k \to T$ as $k \to +\infty$ and a point $x_\infty\in \mathbb S^n$ such that 
\begin{equation}\label{concentration}
\lim_{k \to +\infty}\int_{B_r(x_\infty)}fe^{nu(t_k)}\dvSn =(n-1)!\vol(\mathbb S^n),
\end{equation}
for all $r>0$. Here $B_r(x_\infty)$ is the geodesic ball in $\mathbb S^n$ with radius $r$ and centered at $x_\infty$. In addition, there holds
\begin{equation}\label{concentration1}
\lim_{k \to +\infty}\int_{B_r(y)}fe^{nu(t_k)}\dvSn =0,
\end{equation}
for all $y\in \mathbb S^n\backslash\{x_\infty\}$ and all $0\leqslant r< \dist (y,x_\infty)$.
\end{enumerate}
\end{lemma}

Note that Lemma \ref{concentrationcompactness} is a higher dimensional analogue of \cite[Lemma 4.2]{bfr}. With help of this lemma and under the hypotheses of Theorem \ref{Criticalcase}, we can conclude that $u(t)$ is uniformly bounded in $H^{n/2} (\mathbb S^n)$.

\begin{lemma}\label{Hn/2bd-critical}
Suppose that the flow \eqref{eqFlow} is defined on $[0,T)$. Then, under the hypotheses of Theorem \ref{Criticalcase}, there exists a uniform constant $C>0$ independent of $T$ such that
\begin{equation*}
 \|u(t)\|_{H^{n/2}(\mathbb S^n)}\leqslant C,
\end{equation*}
for all $t\in[0,T)$.
\end{lemma}

\begin{proof}
Since $u_0$ is $G$-invariant and the solution $u(t)$ is unique, it is not hard to see that $u(t)$ is also $G$-invariant. Moreover, the uniqueness of $u(t)$ also helps us to assume that 
\begin{equation}
 \label{strictinequalityofenergy}
 \Escr[u(t)]<\Escr[u_0]
\end{equation}
for $t>0$. We now prove the lemma by way of contradiction. Suppose that $u(t)$ is not uniformly bounded in $H^{n/2}(\mathbb S^n)$ for $t\in[0,T)$. It follows from Lemma \ref{concentrationcompactness} that there exists some point $x_\infty\in \mathbb S^n$ satisfying \eqref{concentration} for all $r>0$. We have the following two possible cases:

\noindent\textbf{Case 1}. The case $\Sigma=\emptyset$. In this scenario, we can find some $\sigma\in G$ such that $\sigma(x_\infty)\neq x_\infty$. From the fact that $u(t)$ is $G$-invariant we find that
\begin{equation*}
 \lim_{k\to +\infty}\int_{B_r(\sigma(x_\infty))}fe^{nu(t_k)}\dvSn =\lim_{k\to +\infty}\int_{B_r(x_\infty)}fe^{nu(t_k)}\dvSn =(n-1)!\vol(\mathbb S^n),
\end{equation*}
which contradicts \eqref{concentration1}.

\noindent\textbf{Case 2}. The case $\Sigma\neq\emptyset$. Observe that if $x_\infty\not\in\Sigma$, then we arrive at a contradiction in the same way as above. Therefore, we are left with $x_\infty\in\Sigma$. In this scenario, we can estimate
\begin{equation}\label{integralonball}
\begin{split}
 \int_{B_r(x_\infty)}fe^{nu(t_k)}\dvSn \leqslant &\big( \sup_{x\in B_r(x_\infty)}f \big) \int_{B_r(x_\infty)}e^{nu(t_k)}\dvSn \\
 \leqslant &\max\Big(\sup_{x\in B_r(x_\infty)}f,0\Big)\int_{\mathbb S^n}e^{nu(t_k)} \dvSn
\end{split}
\end{equation}
for all $r>0$. From Beckner's inequality \eqref{BecknerInequality} and Lemma \ref{energydecay}, it follows that
\begin{equation}\label{upperboundofvolume}
 \frac{1}{\vol(\mathbb S^n)}\int_{\mathbb S^n}e^{nu(t_k)}\dvSn \leqslant \exp\Big(\frac{\Escr[u(t_1)]}{(n-1)!\vol(\mathbb S^n)}\Big).
\end{equation}
Plugging \eqref{upperboundofvolume} into \eqref{integralonball} and letting $k\to +\infty$ yield
\[
(n-1)!\leqslant\max\Big(\sup_{x\in B_r(x_\infty)}f,0\Big)\exp\Big(\frac{\Escr[u(t_1)]}{(n-1)!\vol(\mathbb S^n)}\Big) .
\]
Letting $r \searrow 0$ in the above inequality gives
\[
(n-1)!\leqslant\max\big(f(x_\infty),0\big)\exp\Big(\frac{\Escr[u(t_1)]}{(n-1)!\vol(\mathbb S^n)}\Big) .
\]
Consequently, we have $f(x_\infty)>0$ and therefore we get
\[
(n-1)!\leqslant f(x_\infty)\exp\Big(\frac{\Escr[u(t_1)]}{(n-1)!\vol(\mathbb S^n)}\Big) .
\]
Thanks to \eqref{strictinequalityofenergy}, we deduce from the preceding estimate that
\[
f(x_\infty)>(n-1)!\exp\Big(-\frac{\Escr[u_0]}{(n-1)!\vol(\mathbb S^n)}\Big),
\]
which contradicts the hypothesis of Theorem \ref{Criticalcase}. The proof is now complete.
\end{proof}

As can be seen from the proof of Lemma \ref{Hn/2bd-critical}, the condition \eqref{Initialdata1} for $\sup_\Sigma f$ plays some role. It is natural to ask what happens if we relax this condition. The following result provides an answer for this question.

\begin{proposition}\label{Hn/2bd-critical-description}
Suppose that the flow \eqref{eqFlow} is defined on $[0,T)$. Then there holds
\[
\inf_{0 \leqslant t < T} \Escr [u(t)] > -\infty.
\] 
\end{proposition}

\begin{proof}
We prove by way of contradiction. Suppose that 
\[
\inf_{0 \leqslant t < T} \Escr [u(t)]=-\infty.
\]
Because $|\int_{\mathbb S^n} u \dvSn|^2 \lesssim \int_{\mathbb S^n} u^2 \dvSn$, it is easy to see that the case (i) in Lemma \ref{concentrationcompactness} cannot happen. Hence, there must exist a sequence of $(t_k)$ with $t_k \to T$ as $k \to +\infty$ and a point $x_\infty\in \mathbb S^n$ such that \eqref{concentration} holds for all $r>0$. On the other hand, because the energy $\Escr [u(t_k)]$ is monotone decreasing, we then have
\[
\lim_{k \to +\infty} \Escr [u(t_k) ] = -\infty.
\]
Since for the case $\Sigma=\emptyset$ we simply argue as before to get a contradiction, we are left with the case $\Sigma\neq\emptyset$. Furthermore, we are also left with $x_\infty\in\Sigma$. Now we use Beckner's inequality \eqref{BecknerInequality} in the following way
\[
 \frac{1}{\vol(\mathbb S^n)}\int_{\mathbb S^n}e^{nu(t_k)}\dvSn \leqslant \exp\Big(\frac{\Escr[u(t_k)]}{(n-1)!\vol(\mathbb S^n)}\Big).
\]
Then we combine this with \eqref{integralonball} and letting $k \to +\infty$ in the resulting inequality to get
\[
(n-1)!\leqslant\max\Big(\sup_{x\in B_r(x_\infty)}f,0\Big)
 \lim_{k\to +\infty} \exp\Big(\frac{\Escr[u(t_k)]}{(n-1)!\vol(\mathbb S^n)}\Big) = 0.
\]
This is a contradiction and Proposition \ref{Hn/2bd-critical-description} is proved. 
\end{proof}

As we shall see later, once the energy functional $\Escr$ has a lower bound, the time-depending $H^{n/2}$-bound will hold, and then follow Brendle's argument in \cite{br} to get the global existence of the flow. This means that the condition \eqref{Initialdata1} is not necessary for the global existence. The said scheme will be carried out in the supercritical case in the next subsection.

\subsection{The supercritical case}

Now we turn out our attention to the supercritical case. In this scenario, it is not clear for us to get the uniform $H^{n/2}$-bound by direct estimates as we did before. As suggested by Proposition \ref{Hn/2bd-critical-description}, our argument here is to carefully exploit the connection between energy bound and $H^{n/2} (M)$-norm of the flow and we can merely obtain the time-depending $H^{n/2}$ bound in this subsection. 

Recall that we have assumed that $f>0$, however, it is worth noting that the positivity of $f$ is not necessary in Lemmas \ref{lemSizeOfIntegralOfLambda} and \ref{lowerbdofmeasureofAt} below, where we only need to assume that $\int_M fd\mu_{g(0)}>0$; see Remark \ref{rmk-positivityf} below. 

Now, as always, we let $T$ be the maximal time of existence of the flow \eqref{eqFlow} and let $T_0 \in [0, T)$ be arbitrary. Given any initial data $u_0\in C^\infty(M)$
we also let 
$$\tilde{\lambda}(t)=\lambda(t)-\frac{\int_MQ_0\dvg }{\int_Mfd\mu_{g(0)}}.$$
where $\lambda$ is given in \eqref{lambda}.
\begin{lemma}\label{lemSizeOfIntegralOfLambda}
Given any $\Uscr > 0$, if
\[
\int_M e^{n u_0} \dvg \leqslant \Uscr, \quad \Escr[u_0] <\Uscr, \quad 
\inf_{t \in [0,T_0]} \Escr[u(t)] > -\Uscr,
\]
then it follows that
\[
\int_0^{T_0} \big|\widetilde\lambda(t) \big| dt 
\leqslant n \Uscr (1+T_0) \Big(\int_M fd\mu_{g(0)}\Big)^{-1},
\]
or equivalently
\[
\int_0^{T_0}|\lambda(t)| dt \leqslant \Big( n \Uscr + \Big|\int_M Q_0 \dvg \Big| \Big)(1+T_0) \Big(\int_M fd\mu_{g(0)}\Big)^{-1}.
\]
\end{lemma}

\begin{proof}
First, we notice the following estimate 
\begin{align*}
 \frac{d}{dt}\int_M e^{nu} \dvg&= n\int_M u_te^{nu} \dvg
 \leqslant n\Big(\int_M u_t^2e^{nu} \dvg\Big)^{1/2} \Big (\int_M e^{nu} \dvg\Big)^{1/2}.
\end{align*}
Solving this differential inequality and using Lemma \ref{energydecay} we obtain
\begin{align}\label{volumeupbd}
\begin{aligned}
 \int_Me^{nu} \dvg&\leqslant \Big[\Big(\int_M e^{nu_0} \dvg\Big)^{1/2}+ \frac n2\int_0^{T_0}\Big(\int_M u_t^2e^{nu} \dvg\Big)^{1/2} dt\Big]^2 \\
 &\leqslant 2\Uscr+ \frac{n^2}2 T_0\int_0^{T_0}\int_M u_t^2e^{nu} \dvg dt \\
 &= 2\Uscr- \frac n2 T_0\int_0^{T_0}\frac{d}{dt}\Escr[u(t)]dt \\
 &= 2\Uscr + \frac n2 T_0 \big( \Escr[u_0] - \Escr[u(T_0)]\big) \\
 &\leqslant n \Uscr (1+T_0)
\end{aligned}
\end{align}
for any $t \in[0, T_0]$. On the other hand, by the flow equation \eqref{eqFlow}, \eqref{FixIntoff} and the conformal invariant of $\int_MQ_{g(t)}d\mu_{g(t)}$, we deduce that
\[\begin{split}
\int_M u_te^{nu} \dvg &= \lambda(t) \int_M fd\mu_{g(t)} - \int_M Q_{g(t)} d\mu_{g(t)} = \widetilde\lambda(t)\int_M fd\mu_{g(0)}.
\end{split}\]
This together with Lemma \ref{energydecay}, and \eqref{volumeupbd} gives
\begin{align*} 
\begin{aligned}
\Big(\widetilde\lambda(t)\int_M fd\mu_{g(0)}\Big)^2&= \Big|\int_M u_te^{nu} \dvg\Big|^2 \\
 &\leqslant \int_M u_t^2e^{nu} \dvg\int_M e^{nu} \dvg\\
 &\leqslant - n \Uscr (1+T_0)\frac{d}{dt}\Escr[u(t)].
\end{aligned}
\end{align*}
Keep in mind that $\int_M fd\mu_{g(0)}> 0$. The estimate above then yields that
\[\begin{split}
\int_0^{T_0} \big|\widetilde\lambda(t)\big|dt 
& \leqslant T_0^{1/2}\Big(\int_0^{T_0} \widetilde\lambda^2 (t)dt\Big)^{1/2} \leqslant n \Uscr (1+T_0) \Big(\int_M fd\mu_{g(0)}\Big)^{-1}.
\end{split}\]
Now by using triangle inequality, we easily obtain the last inequality. This completes the proof.
\end{proof}

The goal of our next step is to show that there exists a suitable initial data $u_0\in C^\infty(M)$ such that the energy has a lower bound along the flow \eqref{eqFlow} whenever it exists.
\begin{proposition}\label{propEnergyBound-Supercritical}
Suppose 
\[
(n-1)!\vol(\mathbb S^n) < \int_M Q_0\dvg \neq k(n-1)!\vol(\mathbb S^n),~~k\in\{2, 3, \dots\},
\]
and assume that $f>0$ everywhere. Then there exists some $u_0 \in C^\infty (M)$ such that the solution $u(t)$ to the flow \eqref{eqFlow} with the initial data $u_0$ enjoys the following
\[
\inf_{t \in [0,T)} \Escr[u(t)] > -\infty.
\]
\end{proposition}

\begin{proof}
We prove by contradiction. Given an initial data $u\in C^\infty(M)$ and let $\varphi(t,u)$ be the solution of \eqref{eqFlow} with $\varphi(0,u)=u$, namely
\[
\left\{
\begin{aligned}
\partial_t\varphi&= -(\Po\varphi+Q_0)e^{-n\varphi}+\lambda_\varphi(t)f , \quad t > 0,\\
\varphi(0,u)&= u,
\end{aligned}
\right.
\]
where 
$$\lambda_\varphi(t)=\frac{\int_Mf(\Po\varphi+Q_0) \dvg}{\int_Mf^2e^{n\varphi} \dvg}.$$
Let $[0,T_u)$ be the maximal existence interval of $\varphi (\cdot, u)$. We assume, by contradiction, that
\begin{equation}
 \label{energynobd}
 \inf_{t\in[0,T_u)}\Escr[\varphi(t,u)]=-\infty
\end{equation}
for any smooth initial data $u\in C^\infty(M)$. Let $E:=C^\infty(M)$ be equipped with its natural $C^\infty$-topology. We also define the sub-level set
\[
E_{-\gamma}:=\Big\{u\in E:\Escr[u]\leqslant-\gamma\Big\},
\]
where $\gamma$ is a positive constant. It follows from the decay property of $\Escr$ that $E_{-\gamma}$ is invariant under the flow $\varphi(t)$, that is, if $u\in E_{-\gamma}$, then $\varphi(t,u)\in E_{-\gamma}$ for all $t\in[0,T_u)$. Moreover, under the assumption on $\int_M Q_0\dvg $ we can make use of \cite[Proposition 4.4]{n2007} to show that there exists a sufficiently large constant $\gamma$ such that $E_{-\gamma}$ is not contractible. Let us fix such the constant $\gamma$. 

Next, we show that the flow $\varphi(t,u)$ defines a deformation retract from $E$ to $E_{-\gamma}$. Indeed, for each $u\in E$, by \eqref{energynobd} we can define
\[
t_u=\min\Big\{t\in[0,T_u):\Escr[\varphi(t,u)]\leqslant-\gamma\Big\}.
\]
It follows from the continuity of $\varphi$ that
$$\Escr[\varphi(t_u,u)]=-\gamma.$$
We then extend $\varphi$ on $[0,+\infty)$ by considering
\[
\widehat\varphi(t,u)=\left\{
\begin{aligned}
 &\varphi(t,u) & & \text{ if } t\in[0,t_u),\\ 
 &\varphi(t_u,u) & & \text{ if } t\geqslant t_u.
\end{aligned}
\right.
\]
In view of the decay property of the energy along the flow, we can obtain that $\widehat\varphi$ is continuous on $[0,+\infty)\times E$. Now, we define the homotopy map $H:[0,1]\times E\rightarrow E$ by
\[
H(t,u)=\left\{
\begin{aligned}
& \widehat\varphi \big(\frac{t}{1-t},u \big)& & \text{ if } t\in[0,1),\\ 
& \widehat\varphi(t_u,u) & & \text{ if } t=1.
\end{aligned}
\right.
\]
Let $u\in E$ be arbitrary. Then it is easy to see that 
\[
H(0,u)=\widehat\varphi(0,u)= \varphi(0,u)=u
\]
and 
\[
H(1,u) = \widehat\varphi(t_u,u) = \varphi(t_u,u) \in E_{-\gamma}.
\] 
Moreover, if we further let $u\in E_{-\gamma}$, then by the decay property of the energy along the flow $\varphi (\cdot, u)$ we deduce that $\varphi (t, u) \in E_{-\gamma}$ for all $t$. Thus, $t_u = 0$ and therefore by definition
\[
\widehat\varphi(t,u) = \varphi(0,u) = u 
\]
for any $t \geqslant 0$, giving the conclusion that $H(t,u)= u$ for any $t \in [0,1]$. This shows that $E_{-\gamma}$ is a strong deformation retract of $E$, and we obtain a contradiction, since $E$ is contractible as a topological vector space. 
\end{proof}

 To go further, we set
\[
A_t=\Big\{x\in M: u(x,t)\geqslant\alpha_0\Big\},
\]
where 
\begin{equation}\label{alpha0}
\alpha_0=\frac1n\log\Big(\frac{\int_M f d\mu_{g(0)} }{2 \vol(M) \sup_Mf}\Big).
\end{equation} 
Clearly, $\alpha_0$ is well-defined. We are going to estimate the size of $A_t$.
 
\begin{lemma}\label{lowerbdofmeasureofAt}
Given any $\gamma_0 > 0$, there exists a constant $C_1 > 0$ depending only on $\gamma_0$, $\sup_M |f|$, $n$, $\vol(M)$, $\int_M Q_0 \dvg$, and $\int_M f d\mu_{g(0)}$ such that if
\[
\int_M e^{n u_0} \dvg \leqslant \gamma_0, \quad \|u_0\|_{H^{n/2}(M)}\leqslant\gamma_0, \quad 
\inf_{t\in[0,T_0]}\Escr[u(t)]\geqslant-\gamma_0,
\]
then
\[
\vol (A_t) \geqslant\exp\big(-C_1 e^{C_1T_0}\big)
\]
for all $t\in[0,T_0]$.
\end{lemma}

\begin{proof}

By the assumption, we have
\[
\Escr[u ] = \frac n2\int_Mu\cdot \Po u \dvg + n\int_M Q_0u \dvg \geqslant-\gamma_0
\]
for all $t\in[0,T_0]$. This together with the conformal invariant of $\int_M Q_{g(t)} d\mu_{g(t)}$, the sign of $\int_M f d\mu_{g(0)}$, and the positivity of the operator $\Po$ implies that
\begin{align}\label{derivativeofintegralofuenu-negative}
\begin{aligned}
 \frac{d}{dt}\int_Mue^{nu} \dvg =&\int_Mu_te^{nu} \dvg+n\int_Mue^{nu}u_t \dvg \\
 =&\lambda(t)\int_Mfe^{nu} \dvg-\int_MQ_0 \dvg+n\lambda(t)\int_Mfue^{nu} \dvg \\
 &-n\int_M \big( u\cdot\Po u+Q_0u \big) \dvg \\
 =& \widetilde\lambda (t) \int_M f d\mu_{g(0)}-\frac n2\int_Mu\cdot\Po u \dvg-\Escr[u ] \\
 &+n\lambda(t)\int_Mfue^{nu} \dvg \\
 \leqslant& \widetilde\lambda (t) \int_M f d\mu_{g(0)}+n\lambda(t)\int_Mfue^{nu} \dvg+\gamma_0.
\end{aligned}
\end{align}
We split the term involving $\int_Mfue^{nu} \dvg$ as follows
\begin{align}\label{splitintopm-negative}
\begin{aligned}
\lambda(t)\int_Mfue^{nu} \dvg&= \lambda(t)\int_Mfu^+e^{nu} \dvg-\lambda(t)\int_Mfu^-e^{nu} \dvg =: \frac{I-II}n.
\end{aligned}
\end{align}

\noindent\textbf{Estimate of I}. To estimate this term, we first recall the elementary inequality $se^s\geqslant-1$, which holds for any real $s$. From this we know that the following pointwise estimates
\[
0 \leqslant u^-e^{nu} = u^-e^{-nu^-} e^{nu^+} \leqslant 1/n
\] 
hold. This together with the mean value theorem for integrals implies that 
\begin{align}\label{estimateofpositivepart-negative}
\begin{aligned}
 I&= n\lambda(t)f(p(t))\int_Mu^+e^{nu} \dvg \\
 &= n\lambda(t)f(p(t))\int_Mue^{nu} \dvg+n\lambda(t)f(p(t))\int_Mu^-e^{nu} \dvg \\
 &\leqslant n\lambda(t)f(p(t))\int_Mue^{nu} \dvg+|\lambda(t)|\sup_M |f| \vol(M).
\end{aligned}
\end{align}
where $p(t)$, depending on the time $t$, is some point in $M$.

\noindent\textbf{Estimate of II}. As in \eqref{estimateofpositivepart-negative}, we have the estimate
\begin{equation}\label{estimateofnegativepart-negative}
 II\leqslant n|\lambda(t)|\int_M\big|f\big|\big|u^-e^{nu}\big| \dvg\leqslant|\lambda(t)|\sup_M |f| \vol(M).
\end{equation}
Substituting \eqref{estimateofpositivepart-negative} and \eqref{estimateofnegativepart-negative} into \eqref{splitintopm-negative} yields
\begin{equation*}
n\lambda(t)\int_Mfue^{nu} \dvg\leqslant n\lambda(t)f(p(t))\int_Mue^{nu} \dvg+2|\lambda(t)|\sup_M\big|f\big| \vol(M).
\end{equation*}
Putting \eqref{derivativeofintegralofuenu-negative}, \eqref{splitintopm-negative}, \eqref{estimateofpositivepart-negative}, and \eqref{estimateofnegativepart-negative} together we arrive at the differential inequality
\begin{equation*}
\begin{split}
\frac{d}{dt}\int_Mue^{nu} \dvg\leqslant & n\lambda(t)f(p(t))\int_Mue^{nu} \dvg\\
&+\widetilde\lambda (t) \int_M f d\mu_{g(0)} + 2\sup_M |f| \vol(M) |\lambda (t)| + \gamma_0.
\end{split}
\end{equation*}
We set
\[
\sigma(t)=\int_Mue^{nu} \dvg.
\]
Then solving the differential inequality above gives
\[\begin{split}
\sigma(t)&\leqslant \sigma(0)\exp\Big(n\int_0^t\lambda(s)f(p(s))~ds\Big)\\
&+\int_0^t 
\left[
\begin{aligned}
&\widetilde\lambda (t) \int_M f d\mu_{g(0)} +\\
& 2\sup_M |f| \vol(M) |\lambda(s)| + \gamma_0 
\end{aligned}
\right] 
\exp\Big(n\int_s^t\lambda(\tau)f(p(\tau))~d\tau\Big)~ds
\end{split}\]
for all $t\in[0,T_0]$. Observe that
\[
\sigma(0)=\int_Mu_0e^{nu_0} \dvg \leqslant \int_M e^{(n+1)u_0} \dvg.
\]
Thanks to \eqref{TrudingerInequalityWithHnorm}, we conclude that $\sigma(0)$ is bounded from above by some number depending only on $\gamma_0$. In addition, we also have
\[
n\int_s^t\lambda(s)f(p(s)) ds \leqslant (n\sup_M |f|) \int_0^{T_0}|\lambda(t)| dt 
\]
for all $0 \leqslant s \leqslant t \leqslant T_0$. Therefore, the above estimate for $\sigma (t)$ can be simplified as 
\begin{equation}\label{upbdofsigma-negative}
 \sigma(t)\leqslant C_0 \Big(1+ \Big| \int_M f d\mu_{g(0)} \Big| \int_0^{T_0} \big| \widetilde\lambda (t) \big| dt +\int_0^{T_0}|\lambda(t)| dt\Big)\exp\Big(n\sup_M |f|\int_0^{T_0}|\lambda(t)| dt\Big)
\end{equation}
for some constant $C_0> 0$ depending only on $\gamma_0$, $\sup_M |f|$, $n$, and $\vol(M)$. Combining Lemma \ref{lemSizeOfIntegralOfLambda} and \eqref{upbdofsigma-negative} gives
\begin{equation}\label{upbdofsigma1-negative}
 \sigma(t)\leqslant C_0\exp(C_0T_0)
\end{equation}
for all $t\in[0,T_0]$ and for some new constant $C_0$ depending only on $\gamma_0$, $\sup_M |f|$, $n$, $\vol(M)$, $\int_M Q_0 \dvg$, and $\int_M f d\mu_{g(0)}$. Now, it follows from the inequality $se^s\geqslant-1$ that 
\[
\int_Aue^{nu} \dvg\leqslant \sigma(t) + \vol (M)
\]
for all $t\in[0,T_0]$ and for all $A\subset M$. Combining this and \eqref{upbdofsigma1-negative} gives
\begin{equation}\label{upbdofsigma2-negative}
\int_Aue^{nu} \dvg\leqslant C_0\exp(C_0T_0)
\end{equation}
for some new constant $C_0$ still depending only on $\gamma_0$, $\sup_M |f|$, $n$, $\vol(M)$, $\int_M Q_0 \dvg$, and $\int_M f d\mu_{g(0)}$. Now, we define
\[
\varphi : (0, +\infty) \to [-1/e, +\infty), \quad s \mapsto s\log s.
\] 
Then $\varphi$ is convex on $(0,+\infty)$ and it satisfies 
\[
s=\frac{\varphi(\mu s)}{\varphi(\mu)}-\frac{\varphi(s)}{\log\mu}
\]
for $\mu>0$ and $s>0$. Since $\varphi(s)\geqslant-1$ for $s>0$, we obtain
\begin{equation}\label{upbdofs}
 s\leqslant\frac{\varphi(\mu s)}{\varphi(\mu)}+\frac{1}{\log\mu}.
\end{equation}
Now we apply Jensen's inequality together with \eqref{upbdofsigma2-negative} to get
\begin{equation}\label{upbdofvarphiofintofenuonAt-negative}
 \varphi \Big(\frac{1}{\vol (A_t)}\int_{A_t}e^{nu} \dvg \Big) \leqslant \frac{1}{\vol (A_t)}\int_{A_t}\varphi\big(e^{nu}\big) \dvg\leqslant\frac{C_0\exp(C_0T_0)}{\vol (A_t)}.
\end{equation}
If $\vol (A_t) \geqslant1$, then the lemma is proved by taking any constant $C_0 > 0$. So, we are left with the case $\vol (A_t)<1$. Using \eqref{upbdofs} with $\mu=1/\vol (A_t)$ and $s=\int_{A_t}e^{nu} \dvg$ and by \eqref{upbdofvarphiofintofenuonAt-negative} we have 
\begin{equation} \label{upbdofintofenuonAt-negative}
\begin{split}
\int_{A_t}e^{nu} \dvg 
& \leqslant -\frac{\vol (A_t)}{\log \vol (A_t) }\varphi\Big(\frac{1}{\vol (A_t)}\int_{A_t}e^{nu} \dvg\Big) - \frac{1}{\log \vol (A_t) }\\
& \leqslant -\frac{C_0\exp(C_0T_0)}{\log \vol (A_t) }.
\end{split}
\end{equation}
Keep in mind that $\int_M f d\mu_{g(0)} > 0$. Clearly,
\[\begin{split}
\frac 1{\sup_M f} \int_M f d\mu_{g(0)} 
& \leqslant \int_{A_t}e^{nu} \dvg+\int_{M\backslash A_t}e^{nu} \dvg .
\end{split}\]
However, it follows from the definition of $\alpha_0$ in \eqref{alpha0} that
\[
\int_{M\backslash A_t}e^{nu} \dvg\leqslant e^{n\alpha_0} \vol(M)=\frac{1}{2\sup_Mf}\int_M f d\mu_{g(0)} .
\]
Therefore, combining the previous two estimates gives
\[
\int_{A_t}e^{nu} \dvg\geqslant\frac{1}{2\sup_Mf}\int_M f d\mu_{g(0)},
\]
This and \eqref{upbdofintofenuonAt-negative} yields
\[
\log\frac{1}{\vol (A_t)}\leqslant C_0\exp(C_0T_0)
\]
or equivalently,
\[
\vol (A_t) \geqslant\exp\big(-C_0e^{C_0T_0}\big).
\]
The proof is thus complete.
\end{proof}

To end this subsection, we try to get the time-depending bound for $u(t)$ in $H^{n/2}$-norm.

\begin{proposition}\label{Hn/2bd-supercritical}
Assume that $(n-1)!\vol(\mathbb S^n) <\int_MQ_0 \dvg \neq k(n-1)!\vol(\mathbb S^n),~~k\in\{2, 3, \dots\}$ and that $f>0$ everywhere. Then there exist an initial data $u_0\in C^\infty(M)$ and a positive constant $C_0$ depending only on $f$, $Q_0$, $n$, and $(M,g_0)$ such that
\begin{equation}\label{Hn2bdofu}
 \sup_{t\in[0,T)}\|u(t)\|_{H^{n/2}(M)}\leqslant\exp(C_0e^{C_0T}).
\end{equation}
\end{proposition}
\begin{proof}
To prove the bound, we need to control $ \int_Mu\cdot\Po u$ and $\|u \|_{L^2(M)}$. For any $t\in[0,T_0] \subset [0,T)$ we have
\begin{equation}\label{splitofbaru}
 \Big|\int_Mu \dvg\Big| \leqslant \Big|\int_{A_t}u \dvg\Big|+\Big|\int_{M\backslash A_t}u \dvg\Big| =: I+II.
\end{equation}
To estimate $I$, we simply make use of the elementary inequality $s\leqslant e^s$ and \eqref{volumeupbd} to get
\[
\int_{A_t}u \dvg\leqslant\frac1n\int_{A_t}e^{nu} \dvg\leqslant\frac1n\int_{M}e^{nu} \dvg\leqslant n \gamma_0(1+T_0).
\]
On the other hand, by the definition of $A_t$ we have
\[
\int_{A_t}u \dvg\geqslant-|\alpha_0| \vol( M ).
\]
Therefore, by letting $C_1 = n \gamma_0 + |\alpha_0| \vol (M)$, we get
\[
I\leqslant C_1 (1+T_0).
\]
For the estimate of $II$, we use H\"older's inequality to get 
\[
II\leqslant \vol(M\backslash A_t)^{1/2} \|u(t)\|_{ L^2 (M) }.
\]
Plugging the estimates of $I$ and $II$ into \eqref{splitofbaru} yields
\[
\Big|\int_Mu \dvg\Big|\leqslant \vol(M\backslash A_t)^{1/2} \|u \|_{ L^2 (M) }+C_1(1+T_0).
\]
This together with the Young's inequality implies that for any $\epsilon>0$
\begin{equation}\label{bdofbarusquare}
\Big(\int_Mu \dvg\Big)^2\leqslant(1+\epsilon)\vol(M\backslash A_t)~\|u \|_{ L^2 (M) }^2+(1+\epsilon^{-1})C_1(1+T_0)^2.
\end{equation}
Now, it follows from Poincar\'{e}'s inequality \eqref{eqPositivityOfP_0} that
\begin{equation}\label{L2bdofu}
\|u \|_{ L^2 (M) }^2\leqslant\frac{1}{\lambda_1}\int_Mu \cdot\Po u \dvg+\frac{1}{\vol (M)}\Big(\int_Mu \dvg \Big)^2.
\end{equation}
Combining \eqref{bdofbarusquare} and \eqref{L2bdofu} gives
\[
\Big(1-\frac{(1+\epsilon) \vol(M\backslash A_t) }{ \vol( M)}\Big)\|u \|_{L^2(M)}^2\leqslant\frac{1}{\lambda_1}\int_Mu \cdot\Po u \dvg+(1+\epsilon^{-1})C_1(1+T_0)^2,
\]
that is,
\begin{equation*}
\begin{split}
 \big(\vol(A_t)- \epsilon \vol(M\backslash A_t) \big)\|u \|_{L^2(M)}^2 \leqslant &\frac{ \vol( M)}{\lambda_1} \int_Mu \cdot\Po u \dvg \\
 & +(1+\epsilon^{-1}) \vol( M) C_1(1+T_0)^2.
\end{split}
\end{equation*}
Notice that from Lemma \ref{lowerbdofmeasureofAt} we have $\vol(A_t) \geqslant\exp(-C_0e^{C_0T_0})$. Hence, by choosing 
\[
\epsilon=\exp(-C_0e^{C_0T_0})/(2 \vol(M) )
\]
and observing the fact that $\vol(M\backslash A_t) \leqslant \vol(M)$ we get
\[
\vol(A_t)- \epsilon \vol(M\backslash A_t) \geqslant \frac 12 \exp(-C_0e^{C_0T_0}),
\]
which implies that
\begin{equation}\label{L2bdofu2}
\|u \|_{L^2(M)}^2\leqslant C_2\Big(\int_Mu \cdot\Po u \dvg+1\Big)\exp(C_0e^{C_0T_0})
\end{equation}
for some constant $C_2 > 0$ depending only on $C_0$, $C_1$, $T_0$, $\vol (M)$, and $\lambda_1$. In view of the energy bound $-\gamma_0 \leqslant \Escr[u (t)] \leqslant \Escr[u_0]$ and H\"older's inequality, we can bound
\begin{equation}\label{bdofintofuPou}
 \int_Mu\cdot\Po u \dvg\leqslant C_3 \|u \|_{L^2(M)}+C_3.
\end{equation}
for some constant $C_3>0$ depending only on $\gamma_0$ and $\int_M Q_0^2 \dvg$. Substituting \eqref{bdofintofuPou} into \eqref{L2bdofu2} gives
\[
\|u \|_{L^2(M)}^2\leqslant C_2 (C_3+1)\big(\|u \|_{L^2(M)}+1\big)\exp(C_0e^{C_0T_0}).
\]
However, this is enough to conclude that
\begin{equation}\label{L2bdofu3}
 \|u \|_{L^2(M)}\leqslant\exp(C_0e^{C_0T_0}).
\end{equation}
for some constant $C_0>0$ depending only on $\gamma_0$, $f$, $(M,g_0)$, and $n$. Going back to \eqref{bdofintofuPou}, the bound in \eqref{L2bdofu3} immediately gives us a bound for $ \int_Mu\cdot\Po u$; hence concluding \eqref{Hn2bdofu} as claimed. The proof is complete.
\end{proof} 

\begin{remark}\label{rmk-positivityf}
 In the statement of Proposition \ref{Hn/2bd-supercritical}, the extra condition $f>0$ everywhere is technical because we do not know if the initial data $u_0$ found by Proposition \ref{propEnergyBound-Supercritical} enjoys the requirement $\int_M f d\mu_{g(0)} \ne 0$, which is important if one wants to know the sign of the limit $\lambda_\infty$. This forces us to assume $f$ having a sign, which is assumed to be positive due to the assumption on $\int_M Q_0\dvg$. In other words, there always holds
 \[
 \int_M f e^{nu_0} \dvg > 0.
 \] 
\end{remark}
 

\section{Boundedness of \texorpdfstring{$u(t)$}{Lg} in \texorpdfstring{$H^n (M)$}{Lg} for $0 \leqslant t < T$} 
\label{sec-BoundInH^n}

Now we improve our almost-uniform boundedness of $u(t)$ in $H^{n/2}(M)$ established in Section \ref{sec-BoundInH^(n/2)} to the same in higher-order Sobolev space. In this section, we improve the boundedness of $u(t)$ up to $H^n(M)$. 

The first step in this procedure is to show that $\lambda(t)$ is pointwise bounded on the maximal interval of existence of the flow. This can be regarded as an improvement of Lemma \ref{lemSizeOfIntegralOfLambda}.

\begin{lemma} \label{boundoflambda}
For the subcritical case or $M=\mathbb S^n$, then there exists a universal constant $\Cscr_\ell>0$ such that
\[
|\lambda (t)|\leqslant \Cscr_\ell.
\]
While, for the supercritical case, there exists a constant $\Cscr_\ell(T)>0$ depending on the maximal time of existence $T$, such that
\[
|\lambda (t)|\leqslant \Cscr_\ell (T)
\]
for all $t \in [0,T)$.
\end{lemma}

\begin{proof}
From the definition of $\lambda (t)$ in \eqref{lambda}, to bound $\lambda (t)$ it suffices to bound $\int_MfQd\mu_{g(t)}$ from above and $\int_Mf^2d\mu_{g(t)}$ from below positively away from zero. By the expression of $Q$-curvature in \eqref{qcurvartureequation} we conclude that
\begin{equation}\label{boundofintegraloffq}
\begin{split}
\Big|\int_MfQ_{g(t)}d\mu_{g(t)}\Big| =&\Big|\int_Mf(\Po u+Q_0) \dvg \Big| = \Big|\int_M \big[u\cdot \Po f+fQ_0 \big] \dvg \Big| \\
\leqslant&\sup_M\big(|\Po f|\big)\| u (t)\| _{H^{n/2}(M )}\sqrt{\vol(M)}+\sup_M|fQ_0|\vol(M).
\end{split}
\end{equation}
Thus \eqref{boundofintegraloffq} and the boundedness of $u(t)$ in $H^{n/2}(M )$ established in Section \ref{sec-BoundInH^(n/2)} provide us an upper bound for $\int_MfQ_{g(t)}d\mu_{g(t)}$. To bound $\int_Mf^2d\mu_{g(t)}$ from below, first we denote $\sigma=\int_Mf^2 \dvg >0$. Then, by Jensen's inequality we can estimate
\begin{equation}\label{lowerboundofintegraloff2}
\begin{split}
\int_Mf^2d\mu_{g(t)} =&\int_Me^{nu}f^2 \dvg =\sigma \Big( \frac 1\sigma \int_Me^{nu}f^2 \dvg \Big) \\
 \geqslant&\sigma\exp\Big( \frac n\sigma \int_Muf^2 \dvg \Big) \\
 \geqslant&\sigma\exp\Big(- \frac n{2\sigma}\int_M \big[ u^2+f^4 \big] \dvg \Big) \\
\geqslant&\sigma\exp\Big(-\frac n{2\sigma} \Big(\| u\| ^2_{H^{n/2}(M )}+\sup_M(f^4)\vol(M)\Big)\Big).
\end{split}
\end{equation}
Again the boundedness of $u(t)$ in $H^{n/2}(M )$ tells us that $\int_Mf^2d\mu_{g(t)}$ is bounded from below. Now, the proof follows from \eqref{boundofintegraloffq} and \eqref{lowerboundofintegraloff2}.
\end{proof}

Next, we follow the argument in \cite{br} to bound $\|u \|_{H^n(M)}$. To this end, first by standard elliptic estimates we know that
\[
\|u \|_{H^n(M)} \lesssim \|\Po u \|_{L^2(M)} + \|u \|_{L^2(M)} .
\]
Using the continuous embedding $H^{n/2}(M) \hookrightarrow L^2(M)$ together with the boundedness of $u(t)$ in $H^{n/2}(M)$ we deduce that
\[
\|u \|_{H^n(M)} \lesssim \|\Po u \|_{L^2(M)} +1 .
\]
Hence, to bound $\|u \|_{H^n(M)}$, it suffices to bound $\int_M (\Po u)^2 \dvg$ from above. Let $T_0 \in (0, T)$ be fixed, positive real number and assume that $u(t)$ is defined on $[0,T_0]$. For brevity, we let
\[
v = -e^{nu/2} \big( Q_{g(t)} - \lambda (t) f\big).
\]
Then it is not hard to verify that
\[
\left\{
\begin{aligned}
\frac{\partial}{\partial t} u & = e^{-nu/2} v,\\
\Po u & =-e^{nu/2} v - Q_0 + e^{nu} \lambda (t) f, \\
\frac{\partial}{\partial t} \Po u & = \Po \Big( \frac{\partial}{\partial t} u \Big) = \Po \big( e^{-nu/2} v \big).
\end{aligned}
\right.
\]
Consequently, we can compute the time derivative of $\int_M (\Po u)^2 \dvg$ as follows
\begin{equation}\label{eqHighOrderDecomposition1}
\begin{split}
\frac d{dt} \Big( \int_M (\Po u)^2 \dvg\Big) = & 2 \int_M \Po u \cdot \Po \big( e^{-nu/2} v \big) \dvg \\
 = & -2\int_M e^{nu/2} v \cdot \Po \big( e^{-nu/2} v \big) \dvg\\
& - 2 \int_M Q_0 \Po \big( e^{-nu/2} v \big) \dvg \\
& + 2\lambda (t) \int_M e^{nu} f \Po \big( e^{-nu/2} v \big) \dvg .
\end{split}
\end{equation}
To be able to examine this time derivative, we further estimate \eqref{eqHighOrderDecomposition1} as follows
\begin{equation*}
\begin{split}
\frac d{dt} \Big( \int_M (\Po u)^2 \dvg\Big) = &-2\int_M (-\Delta_0 )^{n/4} \big(e^{nu/2} v \big) (-\Delta_0 )^{n/4} \big( e^{-nu/2} v \big) \dvg\\
& - 2 \int_M \big( (-\Delta_0 )^{n/4} Q_0\big) (-\Delta_0 )^{n/4} \big( e^{-nu/2} v \big) \dvg \\
& + 2\lambda (t) \int_M (-\Delta_0 )^{n/4} \big( e^{nu} f \big) (-\Delta_0 )^{n/4} \big( e^{-nu/2} v \big) \dvg \\
& + \text{lower order terms}.
\end{split}
\end{equation*}
Note that the right-hand side of \eqref{eqHighOrderDecomposition1} involves derivatives of $u$ and $v$ of order at most $n/2$ and the total number of derivatives is at most $n$. In view of Lemma \ref{boundoflambda}, we obtain
\begin{equation}\label{eqDerivativesOfP_0u}
\begin{split}
\frac d{dt} \Big( \int_M (\Po u)^2 \dvg\Big) \leqslant & -2\int_M \big( (-\Delta_0 )^{n/4} v \big)^2 \dvg \\
&+\Cscr \sum_{k_1,...,k_m} \int_M | \nabla_0^{k_1} v| \, | \nabla_0^{k_2} v| \, | \nabla_0^{k_3} u| \cdots |\nabla^{k_m} u| \dvg\\
&+\Cscr \sum_{l_1,...,l_m} \int_M | \nabla_0^{l_1} v| \, | \nabla_0^{l_2}u| \cdots |\nabla^{l_m} u| e^{\alpha (l_1,...,l_m) u} \dvg.
\end{split}
\end{equation}
Note that in \eqref{eqDerivativesOfP_0u}, the first sum is taken over all $m$-tuples $(k_1, ..., k_m)$ with $m \geqslant 3$ satisfying $0 \leqslant k_i \leqslant n/2$ for $1 \leqslant i \leqslant 2$ and $1 \leqslant k_i \leqslant n/2$ for $3 \leqslant i \leqslant m$ with $\sum_{i=1}^m k_i \leqslant n$ and the second sum is taken over all $m$-tuples $(l_1, ..., l_m)$ with $m \geqslant 1$ satisfying $0 \leqslant l_1 \leqslant n/2$ and $1 \leqslant l_i \leqslant n/2$ for $2 \leqslant i \leqslant m$ with $\sum_{i=1}^m l_i \leqslant n$.

Next, we estimate the right hand side of \eqref{eqDerivativesOfP_0u} term by term. Following \cite[Section 4]{br}, we obtain the following two fundamental estimates
\begin{equation}\label{propEstimateOfTheFirstSum}
\begin{split}
- \frac 12 \int_M \big( (-\Delta_0 )^{n/4} v \big)^2 \dvg &+ \Cscr \sum_{k_1,...,k_m} \int_M | \nabla_0^{k_1} v| | \nabla_0^{k_2} v| | \nabla_0^{k_3} u| \cdots |\nabla^{k_m} u| \dvg\\
& \leqslant C \|v\|_{ L^2 (M) }^2 \big( \|u\|_{H^n(M )}^2 +1 \big)
\end{split}
\end{equation}
and
\begin{equation}\label{propEstimateOfTheSecondSum}
\begin{split}
- \frac 12 \int_M \big( (-\Delta_0 )^{n/4} v \big)^2 \dvg &+ \Cscr \sum_{l_1,...,l_m} \int_M | \nabla_0^{l_1} v| | \nabla_0^{l_2} w| \cdots |\nabla^{l_m} u| e^{\alpha u} \dvg\\
& \leqslant C \big( \|v\|_{ L^2 (M) }^2 +1 \big) \big( \|u\|_{H^n(M )}^2 +1 \big).
\end{split}
\end{equation}
Then combining \eqref{eqDerivativesOfP_0u}, \eqref{propEstimateOfTheFirstSum}, and \eqref{propEstimateOfTheSecondSum} gives
\[
\frac d{dt} \Big( \int_M (\Po u)^2 \dvg\Big) \leqslant C \big( \|v\|_{ L^2 (M) }^2 +1 \big) \big( \|u\|_{H^n(M )}^2 +1 \big)
\]
for some constant $C$ independent of $t$. Observe that
\[
\|v\|_{ L^2 (M) }^2 = 4 \int_M e^{nu} \big( Q_{g(t)} - \lambda (t) f \big)^2 \dvg
\]
and there holds $\|u\|_{H^n(M)}^2 \lesssim \int_M (\Po u)^2 \dvg + 1$. Hence,
\[
\frac d{dt} \Big( \int_M (\Po u)^2 \dvg + 1 \Big) \leqslant C \Big( 1 + \int_M \big( Q_{g(t)} - \lambda (t) f \big)^2 d\mu_{g(t)} \Big) \Big( 1 + \int_M (\Po u)^2 \dvg \Big),
\]
which implies
\[
\frac d{dt} \log \Big( \int_M (\Po u)^2 \dvg + 1 \Big) \leqslant C \Big( 1 + \int_M \big( Q_{g(t)} - \lambda (t) f \big)^2 d\mu_{g(t)} \Big) .
\]
Upon integrating both sides of the preceding inequality over $[0,T_0]$ and using Lemma \ref{energydecay}, we conclude that
\begin{equation}\label{eqP_0InL^2}
\int_M (\Po u)^2 \dvg \leqslant C(T_0)
\end{equation}
for all $t \in [0, T_0]$. Thus, combining \eqref{eqPositivityOfP_0} and \eqref{eqP_0InL^2} gives the uniform boundedness of $u(t)$ in $H^n(M )$ as claimed. Thus, we have just proved the following result.

\begin{lemma}\label{lemH^nBoundedInFixedTime}
For any solution $u(t)$ to the flow \eqref{eqFlow}, there exists a constant $\Cscr (T)>0$ such that 
\[
\sup_{0 \leqslant t < T} \|u(t)\|_{H^n(M )} \leqslant \Cscr (T).
\]
\end{lemma}

 Note that if the solution $u(t)$ to the flow is uniformly bounded in $H^n(M )$ for every fixed time interval $[0, T_0]$, then by Sobolev's inequality it is easy to show that along the flow $u(t)$ is also pointwise bounded in $[0, T_0]$. Therefore we have the following simple corollary.

\begin{corollary}\label{corL^inftyBound}
For any solution $u(t)$ to \eqref{eqFlow}, there exists a constant $\Cscr (T)>0$ such that
\[
\sup_{ 0 \leqslant t< T} |u(t)| \leqslant \Cscr (T). 
\]
\end{corollary}

In the next section, we turn our $H^n$-boundedness into even higher-order Sobolev spaces. Having such a boundedness, we obtain a global existence of the flow \eqref{eqFlow}.


\section{Boundedness of \texorpdfstring{$u(t)$}{Lg} in \texorpdfstring{$H^{2k}(M)$}{Lg} for $0 \leqslant t < T$ and the long time existence of the flow}
\label{sec-BoundInH^2k}

The aim of this section is to strengthen the $H^n$-boundedness of $u(t)$ obtained in Section \ref{sec-BoundInH^n} above. What we are going to prove is that $u(t)$ is bounded in $H^{2k} (M )$ for any $k>2$, which is crucial to claim that the flow is defined to all time. Again, it is worth noting that the argument below does not depend on the size of $\int_M Q_0 \dvg$. To realize this, we follow the argument in \cite{br}.

In view of \eqref{eqNormH^p}, to estimate $u(t)$ in $H^{2k}$-norm, we need to estimate $\int_M |(-\Delta_0 )^k u |^2 \dvg$. Using the flow equation $\partial_t u(t) =\lambda (t) f - Q_{g(t)} $, for some constant $\Cscr>0$, we obtain
\begin{equation}\label{eqHigherOrder}
\begin{split}
\frac{d}{dt} \left( \int_M |(-\Delta_0 )^k u |^2 \dvg\right) = & 2 \int_M (-\Delta_0 )^k u (-\Delta_0 )^k \big( \partial_t u \big) \dvg \\
= & - 2\int_M \big( (-\Delta_0 )^k u \big) (-\Delta_0 )^k \big( e^{-nu} \Po u \big) \dvg \\
& - 2\int_M \big( (-\Delta_0 )^k u \big) (-\Delta_0 )^k \big( e^{-nu} Q_0 \big) \dvg \\
& + 2\lambda (t) \int_M \big( (-\Delta_0 )^k u \big) (-\Delta_0 )^k f \dvg \\
 \leqslant & -2\int_M e^{-nu} \big| (-\Delta_0 )^{k+n/4} u \big|^2 \dvg \\
&+\Cscr \sum_{k_1,...,k_m} \int_M | \nabla^{k_1} u| \cdots |\nabla^{k_m} u| \dvg
\end{split}
\end{equation}
for all $t \geqslant 0$. Here the sum is taken for all tuples $(k_1,...,k_m)$ with $m \geqslant 3$ satisfying $1 \leqslant k_i \leqslant 2k+n/2$ and $\sum_{i=1}^m k_i \leqslant 4k+n$. We now choose real numbers $p_i \in [2,+\infty)$ such that $k_i \leqslant 2k +n/{p_i}$ and $\sum_{i=1}^m 1/{p_i} = 1$. Then we set
\[\theta_i = \max \left\{ \frac{k_i - n/p_i - n/2 }{2k-n/2}, 0 \right\}\]
for each $i=1,..,m$. Since $m \geqslant 3$, we immediately have $\sum_{i=1}^m \theta_i < 2$. Now we estimate \eqref{eqHigherOrder}
First, in view of Corollary \ref{corL^inftyBound}, we can further estimate the left hand side of \eqref{eqHigherOrder} to obtain
\begin{equation}\label{eqHigherOrderNew}
\begin{split}
\frac{d}{dt} \left( \int_M |(-\Delta_0 )^k u |^2 \dvg\right) \leqslant & - \Cscr \int_M |(-\Delta_0 )^{k+n/4} u |^2 \dvg \\
&+\Cscr \sum_{k_1,...,k_m} \int_M | \nabla^{k_1} u| \cdots |\nabla^{k_m} u| \dvg
\end{split}
\end{equation}
for some new constant $\Cscr$. By repeatedly using \eqref{eqHigherOrderNew} for suitable $k$ and H\"{o}lder's inequality, we arrive at
\begin{equation}\label{eqHigherOrderH2k}
\begin{split}
\frac{d}{dt} \|u\|_{H^{2k} (M )}^2 \leqslant & - \Cscr \|u\|_{H^{2k+\frac n2} (M )}^2 \\
&+\Cscr \sum_{k_1,...,k_m} \| \nabla^{k_1} u\|_{L^{p_1}(M )} \cdots \|\nabla^{k_m} u \|_{L^{p_m}(M )}
\end{split}
\end{equation}
for some new constant $\Cscr$. Thanks to \eqref{eqGagliardoNirenbergInequality}, we obtain
\[
\| \nabla^{k_i} u\|_{L^{p_i} (M )} \leqslant \Cscr \|u\|_{H^{k_i - \frac n{p_i} + \frac n2}(M )} .
\]
Again, we apply \eqref{eqGagliardoNirenbergInequality} to obtain
\[
\|u\|_{H^{k_i - \frac n{p_i} +\frac n2}(M )} \leqslant \Cscr \|u\|_{H^n (M )}^{1-\theta_i} \|u\|_{H^{2k+\frac n2}(M )}^{\theta_i} .
\]
By combining the last two estimates above and the boundedness of $u(t)$ in $H^n(M )$ established in Lemma \ref{lemH^nBoundedInFixedTime}, we get
\[
\| \nabla^{k_i} u\|_{L^{p_i} (M )} \leqslant \Cscr \|u\|_{H^{2k+ n/2}(M )}^{\theta_i}.
\]
This together with \eqref{eqHigherOrderH2k} implies that
\begin{equation}\label{eqHigherOrderH2kNew}
\begin{split}
\frac{d}{dt} \|u\|_{H^{2k} (M )}^2 \leqslant & - \Cscr \|u\|_{H^{2k+\frac n2} (M )}^2 +\Cscr \sum_{k_1,...,k_m} \|u\|_{H^{2k+\frac n2} (M )}^{\theta_1 + \cdots \theta_m}.
\end{split}
\end{equation}
Since $\sum_{i=1}^m \theta_i < 2$, from \eqref{eqHigherOrderH2kNew} we have by Young's inequality that
\[
\frac{d}{dt} \|u\|_{H^{2k} (M )}^2 \leqslant - \Cscr \|u\|_{H^{2k+\frac n2} (M )}^2 +\Cscr \leqslant - \Cscr \|u\|_{H^{2k} (M )}^2 +\Cscr,
\]
where in the last estimate, we have used the embedding $H^{2k+1} (M ) \hookrightarrow H^{2k} (M )$. From this, it is routine to get
\[ 
\|u\|_{H^{2k} (M )} \leqslant \Cscr_H(k,T)
\]
for some constant $\Cscr_H(k,T)$ as claimed. Thus, we have just finished proving the following result.

\begin{lemma}\label{lemH^2kBoundedInTime}
For any solution $u(t)$ to the flow \eqref{eqFlow}, there exists a constant $\Cscr_H(k, T)$ depending on $k$ and $T$ such that
\[\sup_{0 \leqslant t < T} \|u(t)\|_{H^{2k} (M )} \leqslant \Cscr_H(k,T) \]
for any $k>1$.
\end{lemma}

An immediate consequence of Lemma \ref{lemH^2kBoundedInTime} above is that the flow \eqref{eqFlow}, with appropriate initial data, cannot blow up in finite time.

\begin{proposition}\label{corFlowExistsForAllTime}
Given an initial data $u_0$ satisfying \eqref{Initialdata} or \eqref{Initialdata1} when $M=\mathbb S^n$ and $\Sigma\neq\emptyset$, the flow \eqref{eqFlow} has a smooth solution which is defined to all time.
\end{proposition}






\section{Improved $H^{n/2}$-bound for $u(t)$ in the supercritical case}
\label{sec-UniformBoundInH^2k}

Up to the previous section, we have already seen that $u(t)$ is uniform bounded in $H^{n/2}(M)$ in any fixed time interval and the flow exists to all time. Furthermore, if either $\int_M Q_0 \dvg < (n-1)!\vol (\mathbb S^n)$ or $M = \mathbb S^n$, the $H^{n/2}$-bound of $u(t)$ does not depend on the length of the time interval; see Section \ref{sec-BoundInH^(n/2)}. This section is devoted to showing a similar result in the supercritical case. 

To this purpose, let us mention the following concentration-compactness result due to A. Fardoun, R. Regbaoui in \cite{fr18}. Although the authors in \cite{fr18} only consider 4-manifolds, their proof can be easily adapted to the general even dimensional case . 

\begin{proposition}[A. Fardoun and R. Regbaoui]\label{concentration-compacness}
Let $h\in C^0(M)$ and $(u_k,h_k)_k$ be a sequence in $H^{n/2}(M)\times L^1(M)$ satisfying the equation
\[
\Po u_k+h_k=f_ke^{nu_k}
\]
with
\[
\int_Mf_ke^{nu_k} \dvg=\int_M h_k \dvg
\]
and
\[
h_k \to h
\]
in $L^1(M)$ as $k\rightarrow+\infty$, where $f_k\in C^0(M)$ and $C^{-1}\leqslant f_k\leqslant C$ for some positive constant $C$. Then either
\begin{enumerate}[label=\rm (\roman*)]
 \item the sequence $(e^{|u_k|})_k$ is bounded in $L^p(M)$ for all $p\in[1,+\infty)$ or

 \item for a subsequence, still denoted by $(u_k)_k$, there exist a finite number of points $x_1, \dots, x_m\in M$ and integers $l_1, \dots, l_m\in\mathbb{N}^*$ such that
\[
\sum_{j=1}^ml_j=\frac{(n-1)!\vol(\mathbb S^n)}{\int_M h \dvg}
\]
and as $k\rightarrow+\infty$
\[
e^{nu_k} \to \frac{(n-1)!\vol(\mathbb S^n)}{\int_M h \dvg}\sum_{j=1}^ml_j\delta_{x_j}
\]
in the sense of measure, where $\delta_{x}$ is the Dirac mass at the point $x\in M$.
\end{enumerate}
\end{proposition}

Using the above result, we are able to control $u(t)$ in $H^{n/2}$-norm, which can be thought of as an improvement of Proposition \ref{Hn/2bd-supercritical}.

\begin{lemma}
Let $(n-1)!\vol(\mathbb S^n)<\int_M Q_0 \dvg\neq k(n-1)!\vol(\mathbb S^n),~~k\in\{2, 3, \dots\}$. Then, there exists a uniform constant $\Cscr_{n/2}>0$ such that
\[
\sup_{t \geqslant 0}\|u(t)\|_{H^{n/2}(M)}\leqslant \Cscr_{n/2}.
\]
\end{lemma}

\begin{proof}
 In view of Corollary \ref{byproduct}, we have
\[
\int_0^{+\infty}\int_Mu_t(t)^2e^{nu(t)} \dvg dt <+\infty.
\]
This implies that there exists a sequence $t_k\in[k,k+1]$ for all $k\in\mathbb{N}^*$ such that
\[
\lim_{k\rightarrow+\infty}\int_M|u_t(t_k)|^2e^{nu(t_k)} \dvg=0.
\]
Since $f>0$ everywhere, it is easy to see that
\[
\int_M e^{nu(t_k)} \dvg \leqslant \frac 1{\inf_M f} \int_M f d\mu_{g(0)}.
\]
Besides, as in the proof of Lemma \ref{lemSizeOfIntegralOfLambda}, there holds
\[
\begin{aligned}
\Big(\lambda(t_k) - \frac{ \int_M Q_0 \dvg }{ \int_M fd\mu_{g(0)}} \Big)^2 \Big(\int_M fd\mu_{g(0)}\Big)^2 &\leqslant \int_M |u_t(t_k)|^2 e^{nu(t_k)} \dvg\int_M e^{nu(t_k)} \dvg.
\end{aligned}
\]
Putting the above three estimates together, we deduce that
\[
\lambda(t_k) \to \frac{ \int_M Q_0 \dvg }{ \int_M fd\mu_{g(0)}}
\]
as $k \to +\infty$. Now, if we set
\[
u_k=u(t_k),~~h_k=u_t(t_k)e^{nu_k}+Q_0,\quad\mbox{and}\quad f_k=\lambda(t_k)f,
\]
then we have
\begin{equation}\label{eqforuk}
\Po u_k+h_k=f_ke^{nu_k} 
\end{equation}
with $f_k\in C^0(M)$ satisfying
\[
\frac 1 2 \frac{ \int_M Q_0 \dvg }{ \int_M fd\mu_{g(0)}} \inf_M f
\leqslant f_k\leqslant 2 \frac{ \int_M Q_0 \dvg }{ \int_M fd\mu_{g(0)}} \sup_Mf
\]
and
\[
\int_M f_ke^{nu_k} \dvg=\int_M h_k \dvg.
\]
In addition, we obtain 
\[
\|h_k-Q_0\|_{L^1(M)} \leqslant \Big(\int_Me^{nu(t_k)} \dvg\Big)^{1/2} \Big(\int_M|u_t(t_k)|^2e^{nu(t_k)} \dvg\Big)^{1/2} \rightarrow0
\]
as $k\rightarrow+\infty$. Therefore, under the condition $\int_MQ_0 \dvg\notin (n-1)!\vol(\mathbb S^n) \mathbb N$ we know that the alternative (ii) in Proposition \ref{concentration-compacness} cannot occur. Hence, for any $p \geqslant 1$ but fixed, we know that
\begin{equation}\label{LpNormofeukbd}
 \int_Me^{p|u_k|} \dvg\leqslant C_p
\end{equation}
for some constant $C_p >0$ depending on $p$. Hence, it is not hard to verify that
\begin{equation}\label{LpNormoffkbd}
 \|f_k\|_{L^p}\leqslant C_p
\end{equation}
for all $p\in[1,2)$. Using \eqref{LpNormofeukbd} and \eqref{LpNormoffkbd} we apply the standard elliptic regularity theory to \eqref{eqforuk} and get that $(u_k)_k$ is bounded in $W^{n,p}(M)$ for all $p\in[1,2)$. Keep in mind that Sobolev's embedding theorem $W^{n,p}(M)\subset H^{n/2}(M)\cap C^\alpha(M)$ holds true for all $\alpha\in[0,1-1/p]$. Hence, we have proved that $(u_k)_k$ is bounded in $H^{n/2}(M)$, that is,
\begin{equation} \label{uniformHn2bdofuk}
 \|u_k\|_{H^{n/2}(M)}\leqslant C,
\end{equation}
where $C$ is a positive constant depending only on $n, \gamma, u_0, Q_0$ and $M$. Now, we define \[
v_k(t):=u(t+t_k).
\]
Clearly, $v_k$ solves
\[
\left\{
\begin{aligned}
 e^{nv_k}\partial_tv_k&= -(\Po v_k+Q_0)+\lambda(t+t_k)fe^{v_k},\\
 v_k(0)&= v_k.
\end{aligned}
\right.
\]
In view of \eqref{uniformHn2bdofuk} we know that $\|v_k(0)\|_{H^{n/2}(M)} < +\infty$. From this and
\[
\inf_{t\in[0,1]}\Escr[v_k(t)]=\inf_{t\in[t_k,t_k+1]}\Escr[u(t)] > - \infty,
\]
we can apply Proposition \ref{Hn/2bd-supercritical} with $T_0=1$ to get that
\[
\sup_{t\in[0,1]}\|v_k(t)\|_{H^{n/2}(M)}\leqslant \exp(C_0e^{C_0}),
\]
which is equivalent to
\[
\sup_{t\in[t_k,t_k+1]}\|u(t)\|_{H^{n/2}(M)}\leqslant\exp(C_0e^{C_0}).
\]
Since $t_k\in[k,k+1]$ for all $k\in\mathbb{N}^*$, we have
\[
\sup_{t\in[0,+\infty)}\|u(t)\|_{H^{n/2}(M)}\leqslant\exp(C_0e^{C_0}).
\]
This conclude the supercritical case.
\end{proof}

Once we have the time-independent $H^{n/2}$-bound for $u(t)$, we can follow the proof of Lemma \ref{boundoflambda} to bound $\lambda (t)$ for all time, leading us to the following result.

\begin{lemma} \label{uniformboundoflambda}
There exists a universal constant $\Cscr_\ell>0$ such that
\[
|\lambda (t)|\leqslant \Cscr_\ell
\]
for all $t \geqslant 0$.
\end{lemma}


\section{Boundedness of $u(t)$ in \texorpdfstring{$C^\infty(M)$}{Lg} for all time}
\label{sec-GlobalBoundedness}

In this section, we shall derive the uniform boundedness of our flow in $C^\infty$-norm. The crucial step is to show the convergence of the quantity below 
$$F_2(t)=:\int_M(Q_{g(t)}-\lambda(t)f)^2d\mu_{g(t)}$$
as time goes to infinity. To do so, we first need the following estimates
\begin{corollary}\label{byproduct}
We have the following claims:
\begin{enumerate}[label=(\rm \roman*)]
 \item There is an universal constant $C>0$ such that
\begin{equation}\label{IntF_2bd}
\int_0^{+\infty} \int_M(\lambda(t)f - Q_{g(t)} )^2 d\mu_{g(t)} dt\leqslant C.
\end{equation}

 \item For any real number $\alpha$, there is a uniform constant $C(\alpha)>0$ such that
\begin{equation}
 \label{Intealphaubd}
 \int_Me^{\alpha u(t)}\dvg \leqslant C(\alpha).
\end{equation}
for all $t>0$.
\end{enumerate}
\end{corollary}

\begin{proof}
For Part (i), from the proofs in the Section \ref{sec-BoundInH^(n/2)} and the global existence result Proposition \ref{corFlowExistsForAllTime}, one can easily deduce that
\[
\inf_{t\geqslant0}\Escr[u(t)]>-\infty.
\]
This together with Lemma \ref{energydecay} implies \eqref{IntF_2bd}. For Part (ii), notice that we have showed, in the Sections \ref{sec-BoundInH^(n/2)} and \ref{sec-UniformBoundInH^2k} , the time-independent bound for our flow in $H^{n/2}$-norm. This together with \eqref{TrudingerInequalityWithHnorm} and Proposition \ref{corFlowExistsForAllTime} gives \eqref{Intealphaubd}. The proof is complete. 
\end{proof}

Recall that we have obtained the time-independent bound for $\lambda(t)$; see Lemmas \ref{boundoflambda} and \ref{uniformboundoflambda}. This uniform bound is crucial to the convergence of $F_2(t)$.

\begin{lemma}\label{EnergyDecayInL^2}
We have
$
F_2(t) \rightarrow 0 
$
as $t \to +\infty$.
\end{lemma}
\begin{proof}
Let $0 < \varepsilon \ll 1$ be arbitrary and choose $t_0 \geqslant 0$ in such a way that $F_2 (t_0) \leqslant \varepsilon$. We shall show that $F_2 (t) \leqslant 3\varepsilon$ for all $t \geqslant t_0$. By way of contradiction, we may assume that there is a finite number $t_1$ defined by
\[
t_1 := \inf\{t \geqslant t_0 : F_2 (t) \geqslant 3\varepsilon\}.
\]
Let us now consider the flow in $[0, t_1]$ and thoughout this proof, by $C$ we mean a positive constant independent of $t_0$ and $t_1$ which may change from line to line. Then by using \eqref{qcurvartureequation}, Lemma \ref{boundoflambda}, the fact that $F_2(t)\leqslant 3\epsilon$ for $t_0\leqslant t\leqslant t_1$ and the estimate
\[
|Q_{g(t)}|^2 \leqslant 2 |Q_{g(t)} - \lambda (t) f|^2 +2 |\lambda (t)|^2 f^2
\]
we deduce that
\begin{equation}\label{F2bound}
\int_M e^{-nu} \big(Q_0 + \Po u \big)^2 \dvg = \int_M |Q_{g(t)}|^2 d\mu_{g(t)} \leqslant C
\end{equation}
for all $t_0 \leqslant t \leqslant t_1$. Notice that by \eqref{Intealphaubd} one has $\int_M e^{3nu} \dvg \leqslant C$ for all $t \geqslant 0$. Then H\"older's inequality implies that
\[
\int_M |Q_0 + \Po u|^{3/2} \dvg \leqslant \Big( \int_M e^{-nu} \big(Q_0 + \Po u \big)^2 \dvg \Big)^{3/4} \Big( \int_M e^{3nu} \dvg \Big)^{1/4}.
\]
From all three estimates above, we deduce that
\begin{equation}\label{P3/2bound}
\int_M | \Po u|^{3/2} \dvg \leqslant C
\end{equation}
for all $t_0 \leqslant t \leqslant t_1$. By using Sobolev's inequality, we obtain $
|u (t)| \leqslant C$ for all $t_0 \leqslant t \leqslant t_1$. It follows from Lemma \ref{flowequationsofuandQ} that $Q_{g(t)} - \lambda (t) f $ satisfies the evolution equation
\[\begin{split}
\dfrac{\partial}{\partial t} \big( Q_{g(t)} - \lambda (t) f ) =& - \P_{g(t)}(Q_{g(t)}-\lambda (t) f)+nQ_{g(t)}(Q_{g(t)}-\lambda (t) f) + \lambda' (t)f.
\end{split}\]
From this we can estimate the time derivative of $F_2 (t)$ as follows
\[\begin{split}
\frac d{dt} F_2 (t) =& 2 \int_M \big(Q_{g(t)}-\lambda (t) f \big) \dfrac{\partial}{\partial t} \big( Q_{g(t)} - \lambda (t) f ) d\mu_{g(t)} + n \int_M \big(\lambda (t) f - Q_{g(t)} \big)^3 d\mu_{g(t)} \\
= & - 2 \int_M \big(Q_{g(t)}-\lambda (t) f \big) \P_{g(t)}(Q_{g(t)}-\lambda (t) f) d\mu_{g(t)} \\
&+ 2n \int_M \big(Q_{g(t)}-\lambda (t) f \big)Q_{g(t)}(Q_{g(t)}-\lambda (t) f) d\mu_{g(t)} \\
& + 2 \lambda' (t) \int_M f \big(Q_{g(t)}-\lambda (t) f \big) d\mu_{g(t)} + n \int_M \big(\lambda (t) f - Q_{g(t)} \big)^3 d\mu_{g(t)} \\
= & - 2 \int_M \big(Q_{g(t)}-\lambda (t) f \big) \P_{g(t)}(Q_{g(t)}-\lambda (t) f) d\mu_{g(t)} \\
&+ 2n \lambda(t) \int_M f \big(Q_{g(t)}-\lambda (t) f \big)^2 d\mu_{g(t)} + n \int_M \big(Q_{g(t)} - \lambda (t) f \big)^3 d\mu_{g(t)} .
\end{split}\]
Here we have used \eqref{derivativeofintegraloff} to drop off the term involving $\lambda' (t)$. Next we apply the Gagliardo--Nirenberg interpolation inequality \eqref{eqGagliardoNirenbergInequality} to get
\[
 \int_M \big| \lambda (t) f - Q_{g(t)} \big|^3 d\mu_{g(t)} \leqslant C \Big( \int_M \big(\lambda (t) f - Q_{g(t)} \big)^2 d\mu_{g(t)} \Big) \|\lambda (t) f - Q_{g(t)} \|_{H^{n/2}(M, g(t))},
\]
where the norm can be taken with respect to $g(t)$ due to the uniform boundedness of $u$ on $[t_0, t_1]$. From this, we easily get
\[\begin{split}
 \int_M \big| \lambda (t) f - Q_{g(t)} \big|^3 d\mu_{g(t)} \leqslant & C \Big( \int_M \big(\lambda (t) f - Q_{g(t)} \big)^2 d\mu_{g(t)} \Big) \\
& \times \Big( \int_M \big(Q_{g(t)} - \lambda (t) f\big) \P_{g(t)} \big(Q_{g(t)} - \lambda (t) f \big)d\mu_{g(t)} \Big)^{1/2} \\
&+C \Big( \int_M \big(\lambda (t) f - Q_{g(t)} \big)^2 d\mu_{g(t)} \Big)^{3/2}
\end{split}\]
for some constant $C$ independent of $t$. By using Young's inequality, we further get
\[\begin{split}
 \int_M \big| \lambda (t) f - Q_{g(t)} \big|^3 d\mu_{g(t)} \leqslant & C \Big( \int_M \big(\lambda (t) f - Q_{g(t)} \big)^2 d\mu_{g(t)} \Big)^2 \\
& + \frac 1n \int_M \big(Q_{g(t)} - \lambda (t) f \big) \P_{g(t)} \big(Q_{g(t)} - \lambda (t) f \big) d\mu_{g(t)}\\
&+C \Big( \int_M \big(\lambda (t) f - Q_{g(t)} \big)^2 d\mu_{g(t)} \Big)^{3/2}
\end{split}\]
for some constant $C$ independent of $t$. Thus, we have just shown that
\[
\frac d{dt} F_2 (t) \leqslant C F_2(t)^2 + C F_2(t)^{3/2} + C F_2(t)
\]
for all $t \in [t_0, t_1]$. Since $0< \varepsilon \ll 1$, the preceding inequality gives $F_2' (t) \leqslant C F_2(t)$ for all $t \in [t_0, t_1]$. Integrating both sides of the preceding inequality over $[t_0, t_1]$ gives
\[
2\varepsilon \leqslant F_2(t_1) - F_2(t_0) \leqslant C \int_{t_0}^{t_1} F_2(t) dt.
\]
Now by Corollary \ref{byproduct}(i), we can select $t_0 \gg 1$ in such a way that $C \int_{t_0}^{+\infty} F_2(t) dt<\varepsilon$. Hence we obtain $2\varepsilon < \varepsilon$, which is impossible. Thus we have just obtained $L^2$-convergence of the flow.
\end{proof}

An immediate consequent of Lemma \ref{EnergyDecayInL^2} is that $F_2 (t)$ is bounded and this is exactly the same as that of \eqref{F2bound}, which is limited to a finite time interval, namely we have
\[
\int_M Q_{g(t)}^2 d\mu_{g(t)} \leqslant C
\]
for all $t \geqslant 0$ and for some constant $C>0$. Now combining Lemma \ref{EnergyDecayInL^2} and Corollary \ref{byproduct}(ii) gives a uniform bound for $\int_M |Q_0 + \Po u|^{3/2} \dvg$. Consequently, we obtain a bound similar to \eqref{P3/2bound}, however for all $t \geqslant 0$, namely we have
\[
\int_M | \Po u|^{3/2} \dvg \leqslant C
\]
for all $t \geqslant 0$. Applying Green’s formula gives
\[
u(x)= \overline u + \int_M \Po u(z) G(x,z) d\mu_{g_0(z)},
\]
where $G(x,z)$ is Green's function associated to the operator $\Po$. Green's function $G$ is in $C^\infty( M \times M \setminus \text{diagonal})$ and in the asymptotic expansion of the kernel of $G$, the leading term coincides with that of the Green’s function for $(-\Delta)^{n/2}$ in $\R^n$. To be more precise, up to a constant multiple, we have the following asymptotic for $G$
\[
G(x,z) \sim \log |x-z|+ R(x,z),
\]
where $R \in C^{\infty}(M \times M)$. From this we easily verify that $u \in W^{n, 3/2}(M)$. By Sobolev's inequality, we obtain the following.

\begin{lemma}\label{lemC^0BoundedInAllTime}
For any solution $u(t)$ to the flow \eqref{eqFlow}, there holds
\begin{equation}\label{eqBoundForUAllTime}
\sup_{t \geqslant 0} |u(t)| < +\infty.
\end{equation}
\end{lemma}

Thus, we have just shown that the solution $u(t)$ is uniformly pointwise bounded along the flow. A consequence of \eqref{eqBoundForUAllTime}, we deduce that
\[
\int_M \big( Q_0 + \Po u \big)^2 \dvg \leqslant C
\]
for all $t \geqslant 0$. This gives us a similar bound for $\int_M (\Po u)^2 \dvg$ as in \eqref{eqP_0InL^2}, however, for all time. Hence, we have just improved Lemma \ref{lemH^nBoundedInFixedTime} as follows.

\begin{lemma}
For any solution $u(t)$ to the flow \eqref{eqFlow}, there holds
\[\sup_{t \geqslant 0} \|u(t)\|_{H^n(M )} \leqslant \Cscr_n,\]
where $\Cscr_n>0$ is a uniform constant.
\end{lemma}

By repeating arguments used in Section \ref{sec-BoundInH^2k}, we conclude the following.

\begin{lemma}\label{lemH^2kBoundedInAllTime}
For any solution $u(t)$ to the flow \eqref{eqFlow}, there exists a constant $\Cscr_{2k}>0$ depending on $k$ such that
\[
\sup_{t \geqslant 0} \|u(t)\|_{H^{2k} (M )} \leqslant \Cscr_{2k} 
\]
for any $k>1$.
\end{lemma}



\section{Convergence of the flow}
\label{sec-Convergence}

This section is devoted to various convergences of the flow.

\subsection{Sequential convergence of the flow}

A direct consequence of the $H^{2k}$-boundedness of $u(t)$ obtained in Lemma \ref{lemH^2kBoundedInAllTime} above is the following sequential convergence of the flow \eqref{eqFlow} with an appropriate initial data $u_0$. Of course, we always assume that the hypotheses of Theorems \ref{Noncriticalcase} and \ref{Criticalcase} are fulfilled.

\begin{corollary}\label{sequentialconvergence}
There exist a function $u_\infty\in C^\infty(M )$, a real number $\lambda_\infty$, and a time sequence $(t_j)_j$ with $t_j\rightarrow+\infty$ as $j\rightarrow+\infty$ such that
 \[
\Po u_\infty+Q_0=\lambda_\infty fe^{nu_\infty}
\]
and that the following claims hold:
\begin{enumerate}[label=\rm (\roman*)]
 \item $ \| u(t_j)-u_\infty \| _{C^\infty(M )}\rightarrow0$,
 \item $|\lambda(t_j)-\lambda_\infty|\rightarrow0$, and 
 \item $ \| Q_{g(t_j)}-\lambda_\infty f \| _{C^\infty(M )}\rightarrow0$
\end{enumerate}
as $j\rightarrow+\infty$. Furthermore, if either $0 \ne \int_MQ_0\dvg <(n-1)!\vol(\mathbb S^n)$ or $M=\mathbb S^n$, then there holds $\lambda_\infty=1$. 
\end{corollary}

\begin{proof}
To start our proof, we recall that $u(t)$ is bounded in $H^{2k}(M)$ for arbitrary but fixed $k$. Therefore, by Sobolev's embedding theory, we deduce that there is some smooth function $u_\infty$ and a time sequence $(t_j)_j$ with $t_j\rightarrow+\infty$ as $j\rightarrow+\infty$ such that
\[
 \| u(t_j)-u_\infty \| _{C^\infty(M )}\rightarrow0
\]
as $j\rightarrow+\infty$. This establishes Part (i). The strong convergence $u(t_j) \to u_\infty$ as $j \to +\infty$ also implies that
\[
 e^{n u(t_j)} \to e^{nu_\infty} 
\]
as $j\rightarrow+\infty$. Put
\[
g_\infty = e^{2 u_\infty} g_0, 
\quad Q_\infty= e^{-n u_\infty} \big( \Po u_\infty + Q_0 \big), 
\quad \lambda_\infty = \frac{\int_M f Q_\infty d\mu_{g_\infty}}{\int_M f^2 d\mu_{g_\infty}}.
\]
Part (ii) follows easily because of the strong convergence $u(t_j) \to u_\infty$ as $j \to +\infty$. Now we establish Part (iii). To this purpose, we first establish the following weaker result
\[
\|Q_{g(t_j)}- \lambda_\infty f \|_{L^2(M )} \to 0
\]
as $j \to +\infty$. To see this, we observe from \eqref{eqBoundForUAllTime}, Lemma \ref{EnergyDecayInL^2}, and Lemma \ref{lemC^0BoundedInAllTime} that $\| Q_{g(t_j)}-\lambda(t_j)f \| _{L^2(M)} \to 0$ as $j \to +\infty$ because by writing
\[
\big( Q_{g(t_j)}-\lambda(t_j)f \big)^2=\big( Q_{g(t_j)}-\lambda(t_j)f \big) e^{-nu(t_j)/2}\big( Q_{g(t_j)}-\lambda(t_j)f \big) e^{nu(t_j)/2}
\] 
we can estimate
\[
\begin{split}
\| Q_{g(t_j)}-\lambda(t_j)f \|_{L^2(M)}^2 \leqslant &\| Q_{g(t_j)}-\lambda(t_j)f \|_{L^2(M, g(t_j))}  \\
& \times \Big( \int_M \big( Q_{g(t_j)}-\lambda(t_j)f \big)^2 e^{-nu(t_j)} \dvg\Big)^{1/2} \\
\leqslant & C \| Q_{g(t_j)}-\lambda(t_j)f \|_{L^2(M, g(t_j))} \| Q_{g(t_j)}-\lambda(t_j)f \|_{L^2(M)},
\end{split}
\]
which gives
\[
\| Q_{g(t_j)}-\lambda(t_j)f \|_{L^2(M)} \leqslant C \| Q_{g(t_j)}-\lambda(t_j)f \|_{L^2(M, g(t_j))} \to 0 
\]
as $j \to +\infty$. From this we apply triangle inequality and assertion (i) to get
\[
\| Q_{g(t_j)}- \lambda_\infty f \| _{L^2(M )}\leqslant \| Q_{g(t_j)}-\lambda(t_j)f \| _{L^2(M)}+ |\lambda(t_j)- \lambda_\infty| \| f \|_{L^2(M)} \to 0
\]
as $j \to +\infty$. Now we apply the Gagliardo--Nirenberg inequality in the following way
\begin{align*}
\begin{aligned}
\|Q_{g(t_j)}- \lambda_\infty f \|_{H^k(M )} & \leqslant C \|Q_{g(t_j)}- \lambda_\infty f \|_{H^{2k}(M )}^{1/2} \|Q_{g(t_j)}- \lambda_\infty f \|_{L^2(M )}^{1/2}
\end{aligned}
\end{align*}
for each $k$ fixed. If we rewrite
\[
Q_{g(t_j)}- \lambda_\infty f = e^{-n u(t_j)} (\Po u + Q_0) - \lambda_\infty f ,
\]
then for each $k$ fixed it follows from the boundedness of $u(t)$ in $H^{n+2k} (M)$ that $Q_{g(t_j)}- \lambda_\infty f$ is uniformly bounded in $H^{2k} (M)$. Hence applying Lemma \ref{EnergyDecayInL^2} and Sobolev's embedding theory gives us the conclusion. Finally, to realize the equation satisfied by $u_\infty$, we note that
\[\begin{split}
Q_{g(t_j)} - Q_\infty = &( \Po u(t_j) - \Po u_\infty) e^{-n u(t_j)} \\
& + \Po u_\infty \big( e^{-n u(t_j)} - e^{-n u_\infty} \big) + \big( e^{-n u(t_j)} - e^{-n u_\infty} \big)Q_0,
\end{split}\]
which implies that 
\[
\| Q_{g(t_j)} - Q_\infty \|_{H^k(M)} \to 0
\]
as $j \to +\infty$. Hence
\[
Q_\infty = \lambda_\infty f
\]
everywhere on $M$. This implies that
$$\Po u_\infty+Q_0=\lambda_\infty fe^{nu_\infty}.$$
Finally, we evaluate $\lambda_\infty$. Since $u_0\in Y$ because either $\int_MQ_0\dvg <(n-1)!\vol(\mathbb S^n)$ or $M=\mathbb S^n$, we have
$$\int_Mfe^{u_\infty}\dvg =\int_MQ_0\dvg .$$
Hence, from this it is easy to conclude that $\lambda_\infty=1$ provided $\int_M Q_0 \dvg \ne 0$.
\end{proof}

\subsection{Uniform convergence of the flow}

The main result of this subsection is the following.

\begin{lemma}\label{corExistenceOfSolutionU_infty}
If $\int_M Q_0 \dvg \ne 0$ and the prescribing $Q$-curvature equation
\[
\Po u + Q_0 = fe^{nu}
\]
has a unique solution, the convergence in Corollary \ref{sequentialconvergence} is uniform in time, namely we have
\begin{enumerate}[label=\rm (\roman*)]
 \item $ \| u(t)-u_\infty \| _{C^\infty(M )}\rightarrow0$,
 \item $|\lambda(t)-\lambda_\infty|\rightarrow0$, and
 \item $ \| Q_{g(t)}-\lambda_\infty f \| _{C^\infty(M )}\rightarrow0$
\end{enumerate}
as $t\rightarrow+\infty$. 
\end{lemma}

\begin{proof}
Let $u_\infty$ be found in Corollary \ref{sequentialconvergence}. By way of contradiction, let us assume that there exist some $\varepsilon >0$ and a time sequence $(t_j)_j\subset[0,+\infty)$ with $t_j\rightarrow+\infty$ as $j\rightarrow+\infty$ such that
\[
\| u(t_j)-u_\infty \| _{C^\infty(M )} \geqslant\epsilon
\]
for $j$ large. Now, we consider the associated sequence of functions $u(t_j)$ and sequence of numbers $\lambda(t_j)$. By following the proof of Corollary \ref{sequentialconvergence} and up to a subsequence, still denoted by $(t_j)_j$ we obtain 
\begin{itemize}
 \item [(i)] $ \| u(t_j)- \widetilde u_\infty \| _{C^\infty(M )}\rightarrow0$,
 \item [(ii)] $|\lambda(t_j)- \widetilde \lambda_\infty|\rightarrow0$, and 
 \item [(iii)] $ \| Q_{g(t_j)}- \widetilde \lambda_\infty f \| _{C^\infty(M )}\rightarrow0$
\end{itemize}
as $j\rightarrow+\infty$. Furthermore, the function $\widetilde u_\infty$ solves
\[
\Po \widetilde u_\infty+Q_0 = \widetilde \lambda_\infty fe^{n \widetilde u_\infty}.
\]
It is important to note that in the case $\int_M Q_0 \dvg \ne 0$ our choice of $u_0$ leads us to the positivity of both $\lambda_\infty$ and $\widetilde \lambda_\infty$. Consequently, the two functions $u_\infty + (1/n) \log \lambda_\infty$ and $\widetilde u_\infty + (1/n) \log \widetilde \lambda_\infty$ solve the prescribing $Q$-curvature equation. By the uniqueness property, we deduce that
\[\begin{split}
\int_M f e^{nu_0} \dvg & = \int_M f e^{n u_\infty} \dvg \\
&= e^{\log \widetilde \lambda_\infty - \log \lambda_\infty} \int_M f e^{n \widetilde u_\infty} \dvg = \frac{\widetilde \lambda_\infty}{ \lambda_\infty} \int_M f e^{nu_0} \dvg .
\end{split}\]
Hence $ \widetilde \lambda_\infty = \lambda_\infty$, giving us $\widetilde u_\infty= u_\infty$. This is a contradiction. The proof is complete.
\end{proof}


\subsection{Exponential convergence of the flow}

In the final part of this section, we want to study the uniform convergence of the flow \eqref{metricflow} and its rate under the assumptions $\int_MQ_0 \dvg<0$ and $f\leqslant0$. However, we try to prove a more general result which provides a sufficient condition on the exponentially fast convergence. We should point out that the result below is inspired by Struwe \cite{Struwe}.

\begin{theorem}
Suppose that $u_\infty$ is a strictly relative minimizer in the sense that for some constant $c_0>0$ there holds
\begin{equation}\label{eqStrictRelativeMinimizer}
\int_M \big( h \cdot \Po h - n\lambda_\infty f e^{n u_\infty} h^2 \big) \dvg \geqslant 2c_0 \|h\|_{H^{n/2}(M)}^2
\end{equation}
for all $h \in T_{u_\infty}$, where we let
\[
T_{u} = \Big\{ h \in H^{n/2}(M) : \int_M fe^{nu} h \dvg = 0 \Big\}.
\]
Then $u(t) \to u_\infty$ and $\lambda(t) \to \lambda_\infty$ exponentially fast as $t\to +\infty$ in the following sense: with a constant $C$ depending only on the initial data $u_0$, there holds
\begin{equation*}
\|Q_{g(t)} - \lambda (t) f\|_{L^2(M,g(t))} \leqslant C e^{-2c_0t}
\end{equation*}
for all time and hence
\begin{equation*}
|\lambda (t) - \lambda_\infty| + \|u(t) - u_\infty \|_{C^\infty(M)} \leqslant Ce^{- c_0 t/2}
\end{equation*}
for all $t$. 
\end{theorem}

\begin{proof}
Set 
\[
w(t)=|\lambda (t) - \lambda_\infty| + \|u(t) - u_\infty \|_{H^n(M )}.
\]
Let $0<\epsilon\ll1$ be an arbitrary positive number. From Corollary \ref{sequentialconvergence}, it follows that there exists some $t_0>0$ such that $w(t_0) \leqslant \epsilon$. Then we claim that $w(t)\leqslant 3\epsilon$ for all $t\geqslant t_0$.
We may assume by contradiction that there exists a finite number $t_1$ such that
$$t_1:=\inf\{t\geqslant t_0:w(t)\geqslant3\epsilon\}.$$
This implies that 
\begin{equation}\label{wtsmall}
w(t)\leqslant3\epsilon \quad \mbox{for}~~t_0\leqslant t\leqslant t_1.
\end{equation}
Now, we recall that
\[\begin{split}
\frac 12 \frac d{dt} F_2 (t) = & - \int_M \big(Q_{g(t)}-\lambda (t) f \big) \P_{g(t)}(Q_{g(t)}-\lambda (t) f) d\mu_{g(t)} \\
&+ n \lambda(t) \int_M f \big(Q_{g(t)}-\lambda (t) f \big)^2 d\mu_{g(t)} + \frac n2 \int_M \big(Q_{g(t)} - \lambda (t) f \big)^3 d\mu_{g(t)}.
\end{split}\]
By H\"older's inequality, the embedding $H^{n/2}(M ) \hookrightarrow L^4(M )$, and the equivalence between $g(t)$ and $g_0$, we have
\[\begin{split}
\int_M \big| Q_{g(t)} - & \lambda (t) f \big|^3 d\mu_{g(t)}\\
 \leqslant & \overbrace { \Big(\int_M \big| Q_{g(t)} - \lambda (t) f \big|^2 d\mu_{g(t)} \Big)^{1/2} }^{=o(1)} \Big( \int_M \big| Q_{g(t)} - \lambda (t) f \big|^4 d\mu_{g(t)} \Big)^{1/2} \\
\leqslant & o(1) \int_M \big(Q_{g(t)}-\lambda (t) f \big) \P_{g(t)} \big (Q_{g(t)}-\lambda (t) f \big ) d\mu_{g(t)} \\
& + o(1) \int_M \big(Q_{g(t)}-\lambda (t) f \big)^2 d\mu_{g(t)}
\end{split}\]
with error $o(1) \to 0$ as $t \to +\infty$. Therefore, the time derivative of $F_2(t)$ can be estimated further as follows
\[\begin{split}
\frac 12 \frac d{dt} F_2 (t) \leqslant &- (1+o(1)) \int_M \big(Q_{g(t)}-\lambda (t) f \big) \P_{g(t)}(Q_{g(t)}-\lambda (t) f) d\mu_{g(t)} \\
&+ n \lambda(t) \int_M f \big(Q_{g(t)}-\lambda (t) f \big)^2 d\mu_{g(t)} + o(1)F_2 (t)\\
=&-(1+o(1))d^2\Lscr_u(u_t,u_t)+o(1)F_2 (t),
\end{split}\]
where 
\[
d^2\Lscr_u(u_t,u_t)=\int_M \big( u_t\Po u_t-n\lambda(t) fu_t^2e^{nu} \big)\dvg.
\]
Notice that we can find $\delta(t)\in\mathbb{R}$ such that $u_t+\delta(t)f\in T_{u_\infty}$. In fact, one only needs to solve the following equation with $g_\infty=e^{2u_\infty}g_0$
$$0=\int_M(u_t+\delta(t)f)fd\mu_{g_\infty}=\int_Mu_tf(e^{n(u_\infty-u)}-1)d\mu_{g(t)}+\delta(t)\int_Mf^2d\mu_{g_\infty}.$$
Here we have used the fact that $u_t\in T_{u(t)}$ due to \eqref{derivativeofintegraloff}. It follows from \eqref{wtsmall} and \eqref{lowerboundofintegraloff2} that
$$|\delta|\leqslant C \| u(t)-u_\infty \| _{L^\infty (M)}\sqrt{F_2 (t)}\leqslant C\epsilon \sqrt{F_2 (t)},$$
for $t_0\leqslant t\leqslant t_1$. In consequence, we set $h_0=u_t+\delta(t)f$ for brevity. Then we have
\begin{align*}
 d^2\Lscr_u(u_t,u_t)&=\int_Mu_t\Po u_t-n\lambda(t) fu_t^2e^{nu} \dvg\\
 &=d^2\Lscr_{u_\infty}(h_0,h_0)+I+II \\
 &\geqslant 2c_0 \| h_0 \| ^2_{H^{n/2}(M )}+I+II
\end{align*}
with error terms
\begin{align*}
 I&=\int_M(u_t\Po u_t-h_0\Po h_0) \dvg \\
 &=-\delta\int_M\Po f(2u_t+\delta f) \dvg\\
 &=-2\delta\int_Mu_t\Po f \dvg-\delta^2\int_Mf\Po f \dvg=O(\epsilon F_2 (t))
\end{align*}
and
\begin{align*}
 II&=n\int_M(\lambda_\infty fh_0^2e^{nu_\infty}-\lambda(t)fu_t^2e^{nu}) \dvg\\
 &=n\lambda_\infty\int_Mf(h_0^2e^{nu_\infty}-u_t^2e^{nu}) \dvg-n(\lambda(t)-\lambda_\infty)\int_Mfu_t^2e^{nu} \dvg\\
 &=n\lambda_\infty\int_Mf((h_0^2-u^2_t)+u_t^2(1-e^{n(u-u_\infty)})d\mu_{g_\infty}+O(\epsilon F_2 (t))\\
 &= O(\epsilon F_2 (t))
\end{align*}
for $t_0\leqslant t\leqslant t_1$. Moreover, similar computations and Lemma \ref{lemH^2kBoundedInAllTime} yield
\[
\| h_0 \| ^2_{L^2(M )}= \| u_t \| ^2_{L^2(M )}+O(\epsilon F_2 (t))
\]
for $t_0\leqslant t\leqslant t_1$. Therefore, for sufficiently small $\epsilon>0$ and $t_0\leqslant t\leqslant t_1$ we have
\begin{equation}\label{eqExponentialDecay}
\frac 12 \frac d{dt} F_2 (t) \leqslant -c_0 F_2 (t)
\end{equation}
for some $c_1>0$. Hence we find that
\begin{equation}\label{eqExponentialDecayResult}
F_2(t) \leqslant F_2(t_0) \exp(-2c_0(t-t_0))
\end{equation}
for $t_0\leqslant t\leqslant t_1$. Now, from Lemma \ref{flowequationsofuandQ} we have
\[
\lambda' (t)=\Big (\int_Mf^2d\mu_{g(t)} \Big)^{-1} \Big[ \int_M ( \P_{g(t)}f ) u_td\mu_{g(t)} +n \int_M \lambda (t) f^2 \big(Q_{g(t)} - \lambda (t)f\big)d\mu_{g(t)}\Big].
\]
The term $\int_M ( \P_{g(t)}f ) u_td\mu_{g(t)}$ can be estimated as follows
\[
\Big| \int_M ( \P_{g(t)}f ) u_td\mu_{g(t)} \Big| \leqslant C \|f\|_{C^n} \|u_t\|_{L^2(M,g(t))}.
\]
For the term $ \int_M \lambda (t) f^2 \big(Q_{g(t)} - \lambda (t)f\big)d\mu_{g(t)}$, we can use H\"older's inequality and Lemma \ref{boundoflambda} to get
\[
\Big| \int_M \lambda (t) f^2 \big(Q_{g(t)} - \lambda (t)f\big)d\mu_{g(t)} \Big| \leqslant C\| \lambda (t)f - Q_{g(t)}\|_{L^2(M,g(t))}.
\]
Hence, from \eqref{eqExponentialDecay}, we have
\[
| \lambda' (t) | \leqslant C \sqrt{ F_2(t_0) } \exp(-c_0 (t-t_0)),
\]
and for $t_0 < t \leqslant t_1$ we obtaim
\begin{equation}\label{eqEstimateLambda-Lambda}
|\lambda(t) - \lambda(t_0)| \leqslant \int_{t_0}^{t} |\lambda' (s) |ds \leqslant C \sqrt{ F_2(t_0) }.
\end{equation}
 To estimate $\|u(t)-u_\infty\|_{L^2(M )}$, we use Lemma \ref{lemH^2kBoundedInAllTime} to get 
\[
\|u(t) - u(t_0)\|_{L^2(M )} \leqslant C\int_{t_0}^{t} \|u_t\|_{L^2(M,g(t))} ds \leqslant C \sqrt{ F_2(t_0) }.
\]
This together with the Gagliardo--Nirenberg inequality and Lemma \ref{lemH^2kBoundedInAllTime} implies that
\begin{equation}\label{eqEstimateU-UinL^2}
\begin{split}
 \| u(t)-u(t_0) \| _{H^n(M )} & \leqslant C \| u(t)-u(t_0) \| _{H^{2n}(M )}^{1/2} \| u(t)-u(t_0) \| _{L^2(M )}^{1/2}\\
 & \leqslant CF_2(t_0)^{1/4}.
 \end{split}
\end{equation}
Combining \eqref{eqEstimateLambda-Lambda} and \eqref{eqEstimateU-UinL^2} gives
\begin{equation*}
 |\lambda(t)-\lambda(t_0)|+ \| u(t)-u(t_0) \| _{H^n(M )}\leqslant C F_2(t_0)^{1/4}
\end{equation*}
for $t_0\leqslant t\leqslant t_1$. From this and the fact that $w(t_0)\leqslant\epsilon$, it follows that
\[
w(t)\leqslant C F_2(t_0)^{1/4}+\epsilon
\] 
for $t_0\leqslant t\leqslant t_1$. Now, choose $t_0$ sufficiently large such that $C F_2(t_0)^{1/4}\leqslant\epsilon$. Then, by the definition of $t_1$, we obtain that $3\epsilon\leqslant w(t_1)\leqslant2\epsilon,$
which is a contradiction. Hence, we have $w(t)\leqslant 3\epsilon$ for all $t>t_0$. However, this will imply that \eqref{eqExponentialDecayResult} holds true for all $t>t_0$. Then, one has
$$|\lambda(t)-\lambda(t_0)|+ \| u(t)-u(t_0) \| _{L^2(M )}\leqslant C\sqrt{F(t_0)}$$
for all $t>t_0$. Now, by setting $t=t_j$ in the above equality and then sending $t_j$ to $+\infty$, we see that
\[
|\lambda_\infty-\lambda(t_0)|+ \| u_\infty-u(t_0) \| _{L^2(M )}\leqslant C\sqrt{F(t_0)}
\]
for all sufficiently large $t_0>0$. By the Gagliardo--Nirenberg inequality and Lemma \ref{lemH^2kBoundedInAllTime}, we obtain
\[
\| u_\infty-u(t_0) \| _{H^k(M )}\leqslant C \| u_\infty-u(t_0) \| _{H^{2k}(M )}^{1/2} \| u_\infty-u(t_0) \| _{L^2(M )}^{1/2}\leqslant CF_2(t_0)^{1/4}
\]
for all $k\geqslant1$. Hence, we have
\[
|\lambda_\infty-\lambda(t_0)|+ \| u_\infty-u(t_0) \| _{C^\infty(M )}\leqslant CF_2(t_0)^{1/4}.
\]
Renaming $t_0$ as $t$ and choosing $t_0>0$ such that \eqref{eqExponentialDecayResult} holds for all $t>t_0$, we then complete the proof.
\end{proof}

In general, the condition \eqref{eqStrictRelativeMinimizer} is not easy to test; however, by a fairy simple argument we know that the inequality \eqref{eqStrictRelativeMinimizer} holds true if $f$ is non-vanishing with $f\leqslant 0$ everywhere; see \cite[Lemma 2.2]{Ga17} for a similar argument. Hence, we have the following corollary

\begin{corollary}
 If $\int_MQ_0 \dvg<0$ and $f\leqslant0$ everywhere, then we have
\begin{enumerate}[label=\rm (\roman*)]
 \item $ \| u(t)-u_\infty \| _{C^\infty(M )}\rightarrow0$, 
 \item $|\lambda(t)-\lambda_\infty |\rightarrow0$, and
 \item $ \| Q_{g(t)} - \lambda_\infty f \| _{C^\infty(M )}\rightarrow0$ 
\end{enumerate} 
exponentially fast as $t\rightarrow+\infty$.
\end{corollary}


\section{Proof of main results}
\label{sec-ProofsOfMain}

\subsection{Proof of Theorems \ref{Noncriticalcase}, \ref{Criticalcase} and \ref{Uniformconvergence}}
Theorems \ref{Noncriticalcase} and \ref{Criticalcase} follows from Corollary \ref{sequentialconvergence} while Theorem \ref{Uniformconvergence} follows from Lemma \ref{corExistenceOfSolutionU_infty}. We thus complete the proof of the three theorems. 

\subsection{Proof of Corollary \ref{maincorollary}}
For Part (i), following the statement of Theorem \ref{Noncriticalcase}, the initial data $u_0$ for the flow can be arbitrary. Therefore, one may pick some $u_0$ and fix it. This answers why in this case, the term involving $u_0$ on the right hand side of the following inequality
\[
\sup_M f^+ \leqslant C_0 \exp \big( -\tau \|u_0\|_{H^{n/2}(M)}^2 \big)
\] 
can be absorbed into the constant $C_0$, leading to a simpler inequality for the existence. 

Clearly, Part (iii) is straightforward. 

For Part (v), it follows from Corollary \ref{sequentialconvergence} that the limiting function $u_\infty$ solves
\[
\Po u_\infty+Q_0=\lambda_\infty fe^{nu_\infty}.
\]
Recall that our assumption on the initial data $u_0$ gives $\int_M f e^{n u_0} \dvg > 0$. Therefore, integrating both sides of the equation yields 
\[
\int_M Q_0 \dvg = \lambda_\infty \int_M f e^{n u_\infty} \dvg = \lambda_\infty \int_M f e^{n u_0} \dvg.
\]
Hence $\lambda_\infty > 0$. From this we know that the function
\[
u_\infty + (1/n) \log \lambda_\infty
\]
solves \eqref{prescribeqcurvatureequation}. 

For Part (ii), by choosing a new background metric if necessary, we may assume that $Q_0\equiv0$. Then $u_\infty$ satisfies
$$\Po u_\infty=\lambda_\infty fe^{nu_\infty}.$$
First we rule out the possibility of $\lambda_\infty=0$. Indeed, if this is not the case, then we get from the preceding equation that $\Po u_\infty=0$. From this we conclude that $u_\infty\equiv C$. Thus
\[
0\neq e^{nC} \int_Mf\dvg =\int_Mfe^{nu_\infty}\dvg =\int_MQ_0\dvg =0,
\]
which is a contradiction. Hence, there holds $\lambda_\infty\neq0$. Now, if $\lambda_\infty>0$, then the new function
\[
v_\infty=u_\infty+ \frac 1n\log\lambda_\infty
\]
satisfies $\Po v_\infty=fe^{nv_\infty}$; while if $\lambda_\infty<0$, then the function
\[
v_\infty=u_\infty+\frac 1n\log(-\lambda_\infty)
\] 
satisfies $\Po v_\infty=-fe^{nv_\infty}$. Thus, for either case, $e^{2v_\infty}g_0$ is the metric we need. We now consider the case $\int_M f \dvg < 0$. In this scenario, we shall show that $\lambda_\infty>0$. Indeed, the idea is to choose initial data $u_0$ carefully. Mimicking the idea of \cite{gx2008}, we first let $g_0$ be such that $Q_0 \equiv 0$. Then as shown in Appendix \ref{apd-GeXu} the minimizing problem
\[
\inf_{w \in \widetilde{\mathscr H}} \Escr[u]
\]
over the set $\widetilde{\mathscr H}$, defined in Step 2 of Appendix \ref{apd-GeXu}, admits a solution $w \in \widetilde{\mathscr H}$. Still by Appendix \ref{apd-GeXu}, it is important to note that $w$ is also a minimizer of 
\[
\inf_{w \in \mathscr H} \Escr[u]
\] 
over the set $\mathscr H$ defined in Step 1 of Appendix \ref{apd-GeXu}. Therefore, we get $\int_M f e^{nw} \dvg = 0$ and $\int_M e^{nw} \dvg = \vol(M)$. Let us consider our flow \eqref{eqFlow} with the initial data $w$. Then the flow converges to some non-constant function $u_\infty$ which solves
\[
\Po u_\infty =\lambda_\infty fe^{nu_\infty}.
\]
for some $\lambda_\infty \ne 0$. Suppose that $\lambda_\infty < 0$, then it follows that
\[
\int_M f u_\infty e^{nu_\infty} \dvg= \frac 1{\lambda_\infty} \int_M u_\infty \cdot \Po u_\infty \dvg < 0.
\]
Hence, following the idea in \cite{gx2008}, because $\int_M f e^{n u_\infty} \dvg \geqslant 0$ and $\int_M e^{nu_\infty} \dvg= \vol (M)$ there is some $t \in (0,1)$ closed to $1$ such that
\[
\int_M f e^{n t u_\infty} \dvg > 0
\]
and that
\[
\frac {\vol(M)}2\leqslant \int_M e^{n tu_\infty} \dvg \leqslant \frac {3\vol(M)}2.
\]
Putting these facts together, we deduce that $t u_\infty \in \widetilde{\mathscr H}$. However, it is not hard to see that
\[
0 < \Escr[tu_\infty] = t^2 \Escr[u_\infty] < \Escr[u_\infty] \leqslant \Escr[w],
\]
which contradicts the fact that $w$ is an minimizer of $\Escr$ over $\widetilde{\mathscr H}$. Thus $\lambda_\infty > 0$.

Finally, we consider Part (iv). The idea is to show that there exists suitable initial data $u_0$ such that we can apply Theorem \ref{Criticalcase}. If the set $\Sigma=\emptyset$, then the conclusion follows immediately from Corollary \ref{sequentialconvergence}. Otherwise, the set $\Sigma\neq\emptyset$ and $f$ satisfies the inequality \eqref{SupfonSigma}. We consider the following two cases: 

\noindent\textbf{Case (a1)}. Suppose $\int_{\mathbb S^n}f\circ\phi_{y_0,r_0}\dvSn \leqslant0$. In this case, the inequality \eqref{SupfonSigma} tells us that 
\[
\sup_{x \in \Sigma} f(x) \leqslant 0,
\]
so the condition (b) in Theorem \ref{Criticalcase} obviously holds for any G-invariant initial data $u_0$. Hence the conclusion follows.

\noindent\textbf{Case (a2)}. Suppose $\int_{\mathbb S^n}f\circ\phi_{y_0,r_0}\dvSn >0$. In this scenario, we set 
\[
u_0=\frac{1}{n}\log(\det{d\phi_{y_0,r_0}^{-1}})+C,
\] 
where the constant $C$ is chosen in such a way that
\[
e^{nC}\int_{\mathbb S^n}f\circ\phi_{y_0,r_0}\dvSn =(n-1)!\vol(\mathbb S^n).
\]
With this choice of $C$, we conclude that $u_0\in Y$. Since $y_0\in\Sigma$, it is not hard to see that $u_0$ is $G$-invariant. Moreover, the well-known fact
\[
\Escr \Big[ \frac 1 n \log \big(\det{d\phi_{y_0,r_0}^{-1}} \big) \Big]=0
\]
implies that
\[
\Escr[u_0]=\Escr \Big[ \frac 1 n \log \big(\det{d\phi_{y_0,r_0}^{-1}} \big) \Big] + n!\vol(\mathbb S^n)C = n!\vol(\mathbb S^n)C.
\] 
From this, it follows the inequality \eqref{SupfonSigma} that
\[\begin{split}
\sup_{x\in\Sigma}f(x)\leqslant & \frac 1{\vol(\mathbb S^n)}\int_{\mathbb S^n}f\circ\phi_{y_0,r_0}\dvSn\\
= & (n-1)!\exp(-nC)\\
=&(n-1)!\exp\Big(-\frac{\Escr[u_0]}{(n-1)!\vol(\mathbb S^n)}\Big).
\end{split}\]
Hence the condition (b) in Theorem \ref{Criticalcase} holds and the conclusion follows from Corollary \ref{sequentialconvergence}. \qed 


\begin{remark}
Before closing this work, we would like to summarize our major existence results for the prescribing $Q$-curvature problem as follows:

\begin{table}[h]
\begin{tabularx}{\textwidth}{ c|>{\centering}X |>{\centering}X|>{\centering}X|>{\centering\arraybackslash}X}
\hline 
$\int_M Q_0 $ & $ <0$ & $ =0$ & $\in (0, (n-1)!\vol(\mathbb S^n))$ & $ > (n-1)!\vol(\mathbb S^n)$ \\
\hline
n. cond. & $\inf f< 0$ & $f$ changes sign & $\sup f>0 $ & $\sup f > 0$ \\ \hline
s. cond. & $\sup f < C(f^-) $ & $\int_M f \dvg < 0$ & $\sup f >0 $ & $f > 0$\\ \hline
\end{tabularx}
\end {table}
In the above table, `n. cond.' and `s. cond.' stand for necessary and sufficient conditions respectively. In the null case, as far as we know, there is no existence/non-existence result for the case $\int_M f \dvg \geqslant 0$. As already commented in \cite{bfr} and \cite{gx2008}, there is an example of 4-manifolds with non-negative total $Q$-curvature. In the supercritical case, the positivity of $f$ seems to be technical and we do not know whether or not if this is necessary.

\end{remark}


\section*{Acknowledgments}

Q.A. Ng\^o would like to thank Professor Michael Struwe for sharing his preprint \cite{Struwe} at an early stage and for many fruitful discussions regarding curvature flows. Thanks also go to Professor Andreas Juhl for useful discussion on the GJMS operators. Special thanks also go to the Vietnam Institute for Advanced Study in Mathematics (VIASM) for hosting and support where this work was initiated. The research of Q.A. Ng\^o is supported in part by NAFOSTED under grant No. 101.02-2016.02. He also acknowledges the support of Tosio Kato Fellowship from 2019 to 2020. H. Zhang would like to acknowledge the support from the NUS Research Grant R-146-000-169-112 between Oct 2015 and Jan 2016 when part of the work was done. He would also like to thank the hospitality of Professor Xi'nan Ma during his visit to University of Science and Technology of China where part of the work was done. Both authors would like to thank Professor Xingwang Xu for his constant encouragement.


\appendix
\section{Proof of Lemmas \ref{nonconcentration} and \ref{concentrationcompactness}}
\label{apd-nonconcentration}

\subsection{Proof of Lemma \ref{nonconcentration}}

We prove that for a measurable subset $K$ of $M$ with $\vol (K)>0$, there are two constants $\alpha>1$ and $C_K>1$ such that
\begin{equation}\label{apd-eq0}
\int_Me^{nu (t)}\dvg \leqslant C_K \exp\big( \alpha\| u_0\| ^2_{H^{n/2}(M)} \big) \max\Big\{\Big(\int_Ke^{nu(t)}\dvg \Big)^\alpha,1\Big\}.
\end{equation}
Here $\alpha$ depends on $(M, g_0)$ and $C_K$ depends on $(M, g_0)$ and $\vol (K)$. Fix $t>0$ and denote $u=u(t)$ if no confusion occurs. 

\noindent\textbf{Step 1}. Borrowing the idea in \cite{bfr}, in this step we prove that
\begin{equation}\label{apd-eq1}
\int_M u \dvg \leqslant \big| \Escr [u_0]\big| + \frac C{\vol (K)} + \frac {4\vol(M)}{\vol (K)} \max \Big\{ \int_K u \dvg , 0 \Big\}
\end{equation}
for some (uniform) constant $C>0$. Suppose that $\int_M u \dvg>0$, otherwise \eqref{apd-eq1} is trivial. By the definition of $\Escr$ and Lemma \ref{energydecay}, there holds
\[
\int_M \Po u \cdot u \dvg \leqslant \frac 2n\Escr [u_0 ] - 2 \int_M Q_0 u \dvg.
\]
One way to bound $\int_M u \dvg$ from above is to bound $\int_M u^2 \dvg$ from above. Following this strategy, by Poincar\'e's inequality \eqref{eqPositivityOfP_0}, we first obtain
\begin{equation}\label{apd-eq2}
\int_M u^2 \dvg \leqslant \frac 2{n\lambda_1} \Escr [u_0 ] - \frac 2{\lambda_1} \int_M Q_0 u \dvg + \frac 1{\vol(M)} \Big( \int_M u \dvg \Big)^2.
\end{equation}
Depending on the sign of $\int_K u \dvg$, we consider the following two cases.

\noindent\textit{Case 1}. Suppose that $\int_K u \dvg \leqslant 0$. Then it is easy to see that
\[
\Big( \int_M u \dvg \Big)^2 \leqslant \Big( \int_{M \backslash K} u \dvg \Big)^2 \leqslant \vol(M \backslash K) \int_M u^2 \dvg.
\]
From this and \eqref{apd-eq2} we obtain the following
\begin{equation}\label{apd-eq3}
\frac {\vol(K)}{\vol(M)} \int_M u^2 \dvg \leqslant \frac 2{n\lambda_1} \Escr [u_0 ] - \frac 2{\lambda_1} \int_M Q_0 u \dvg.
\end{equation}
Applying Young's inequality gives
\[
\Big| \frac 2{\lambda_1} \int_M Q_0 u \dvg \Big| \leqslant \frac {2\vol(M)}{\lambda_1^2 \vol(K)} \int_M Q_0^2 \dvg + \frac { \vol(K)}{2 \vol(M)} \int_M u^2 \dvg.
\]
Hence, combining this and \eqref{apd-eq3} gives
\[
 \int_M u^2 \dvg \leqslant \frac {4\vol(M)}{n\lambda_1\vol(K)} \Escr [u_0 ] + \frac 4{\lambda_1^2} \frac{\vol(M)^2}{\vol(K)^2} \int_M Q_0^2 \dvg,
\]
which then yields
\begin{equation}\label{apd-eq4}
\Big( \int_M u \dvg \Big)^2 \leqslant \underbrace { \frac {4\vol(M)^2}{n\lambda_1\vol(K)} \Escr [u_0 ] }_{\leqslant \Escr [u_0 ]^2 + \big( \frac{2\vol(M)^2 }{ n \lambda_1 \vol(K)} \big)^2} + \frac 4{\lambda_1^2} \frac{\vol(M)^3}{\vol(K)^2} \int_M Q_0^2 \dvg.
\end{equation}
Thus we have just shown from \eqref{apd-eq4} that
\[
 \int_M u \dvg \leqslant \big| \Escr [u_0 ] \big| + \frac{2\vol(M)^2}{n \lambda_1 \vol(K)} + \frac 2{\lambda_1 } \frac{\vol(M)^{3/2}}{\vol(K)} \Big(\int_M Q_0^2 \dvg \Big)^{1/2}.
\]
This establishes \eqref{apd-eq1} in the first case.

\medskip
\noindent\textit{Case 2}. Suppose that $\int_K u \dvg > 0$. In this scenario, we first rewrite \eqref{apd-eq2} as follows
\[\begin{split}
\int_M u^2 \dvg \leqslant &\frac 2{n\lambda_1} \Escr [u_0 ] - \frac 2{\lambda_1} \int_M Q_0 u \dvg \\
& + \frac 1{\vol(M)} \Big(\Big( \int_K u \dvg \Big)^2 + \Big( \int_{M \backslash K} u \dvg \Big)^2\Big)\\
& + \frac 2{\vol(M)} \Big( \int_K u \dvg \Big) \Big( \int_{M \backslash K} u \dvg \Big).
\end{split}\]
By Young's inequality and the inequality $$\Big(\int_{M \backslash K} u \dvg \Big)^2 \leqslant \vol(M\backslash K) \int_M u^2 \dvg,$$ we obtain
\[\begin{split}
\frac 2{\vol(M)} \Big( \int_K u \dvg \Big) \Big( \int_{M \backslash K} u \dvg \Big) \leqslant &\frac{2 \vol(M \backslash K)}{\vol(K)\vol(M)} \Big( \int_K u \dvg \Big)^2 \\
&+ \frac{\vol(K)}{2\vol(M)} \Big( \int_M u^2 \dvg \Big).
\end{split}\]
Hence 
\[\begin{split}
\int_M u^2 \dvg \leqslant &\frac 2{n\lambda_1} \Escr [u_0 ] - \frac 2{\lambda_1} \int_M Q_0 u \dvg + \frac{\vol(K) + 2 \vol(M \backslash K)}{\vol(K)\vol(M)} \Big( \int_K u \dvg \Big)^2 \\
&+ \frac{\vol(K) + 2\vol(M\backslash K)}{2\vol(M)} \Big( \int_M u^2 \dvg \Big),
\end{split}\]
which implies
\[\begin{split}
\frac{\vol(K) }{2\vol(M)} \int_M u^2 \dvg \leqslant &\frac 2{n\lambda_1} \Escr [u_0 ] - \frac 2{\lambda_1} \int_M Q_0 u \dvg + \frac{2}{\vol(K) } \Big( \int_K u \dvg \Big)^2
\end{split}\]
Again applying Young's inequality gives
\[
\Big| \frac 2{\lambda_1} \int_M Q_0 u \dvg \Big| \leqslant \frac {4\vol(M)}{ \lambda_1^2 \vol(K)} \int_M Q_0^2 \dvg + \frac {\vol(K)}{4\vol(M)} \int_M u^2 \dvg.
\]
Therefore,
\[\begin{split}
 \int_M u^2 \dvg \leqslant &\frac 8{n\lambda_1} \frac{ \vol(M)}{\vol(K) } \Escr [u_0 ] + \frac {16\vol(M)^2}{ \lambda_1^2 \vol(K)^2} \int_M Q_0^2 \dvg + \frac{8\vol(M)}{\vol(K)^2 } \Big( \int_K u \dvg \Big)^2,
\end{split}\]
which yields
\begin{equation}\label{apd-eq5}
\begin{split}
 \Big( \int_M u \dvg \Big)^2 \leqslant & \overbrace { \frac 8{n\lambda_1} \frac{ \vol(M)^2}{\vol(K) } \Escr [u_0 ] }^{\leqslant \Escr [u_0 ]^2 + \big( \frac{4\vol(M)^2 }{ n \lambda_1 \vol(K)} \big)^2} \\
&+ \frac {16\vol(M)^3}{ \lambda_1^2 \vol(K)^2} \int_M Q_0^2 \dvg + \frac{8\vol(M)^2}{\vol(K)^2 } \Big( \int_K u \dvg \Big)^2 .
\end{split}
\end{equation}
Thus we have just shown from \eqref{apd-eq5} that
\[\begin{split}
 \int_M u \dvg \leqslant \Escr [u_0 ] &+ \frac{4\vol(M)^2}{n \lambda_1 \vol(K)} + \frac 4{\lambda_1 } \frac{\vol(M)^{3/2}}{\vol(K)} \Big(\int_M Q_0^2 \dvg \Big)^{1/2}\\
& + \frac{2\sqrt 2\vol(M) }{\vol(K) } \Big( \int_K u \dvg \Big).
\end{split}\]
This establishes \eqref{apd-eq1} in the second case.

\noindent\textbf{Step 2}. By the definition of $\Escr$ and Lemma \ref{energydecay}, there holds
\[
\int_M \Po u \cdot u \dvg \leqslant \frac 2n\Escr [u_0 ] - 2 \Big(\int_M Q_0 \dvg \Big) \overline u- 2 \int_M Q_0 (u - \overline u) \dvg .
\]
By making use of Young's inequality and Poincar\'e's inequality \eqref{eqPositivityOfP_0}, we obtain
\[\begin{split}
\int_M \Po u \cdot u \dvg \leqslant & \frac 2n\Escr [u_0 ] - 2 \Big(\int_M Q_0 \dvg \Big) \overline u \\
&+ \frac 2{\lambda_1} \Big(\int_M Q_0^2\dvg \Big) +\frac {\lambda_1}2\Big( \int_M (u - \overline u)^2 \dvg \Big)\\
\leqslant & \frac 2n\Escr [u_0 ] - 2 \Big(\int_M Q_0 \dvg \Big) \overline u\\
& +\frac 2{\lambda_1} \Big(\int_M Q_0^2\dvg \Big) + \frac 12 \int_M \Po u \cdot u \dvg .
\end{split}\]
Thus, we have just proved that
\[\begin{split}
\frac 12\int_M \Po u \cdot u \dvg \leqslant & \frac 2n\Escr [u_0 ] - 2 \Big(\int_M Q_0 \dvg \Big) \overline u +\frac 2{\lambda_1} \Big(\int_M Q_0^2\dvg \Big) .
\end{split}\]
Thanks to \eqref{initialvolume}, we can estimate
\[
\begin{split}
 \int_Me^{nu }\dvg \leqslant & \Cscr_A \exp\Big(\frac{n}{2(n-1)!\vol(\mathbb S^n)}\int_Mu \cdot \Po u \dvg + \frac n{\vol(M)} \int_Mu \dvg \Big) \\
\leqslant & \Cscr_A \exp
\left[
\begin{split}
&\frac{2}{ (n-1)!\vol(\mathbb S^n)} \Escr [u_0 ] \\
& + \frac n{\vol(M)} \Big(1 - \frac{2 }{ (n-1)!\vol(\mathbb S^n) } \int_M Q_0 \dvg \Big) \Big( \int_Mu \dvg \Big) \\
& + \frac{n}{ (n-1)!\vol(\mathbb S^n)}\frac 2{\lambda_1} \Big(\int_M Q_0^2\dvg \Big)
\end{split}
\right] .
\end{split} 
\]
Note that
\[
\Escr [u_0 ] \leqslant \frac n2 \|u_0\|_{H^{n/2}(M)}^2 + \int_M Q_0^2\dvg.
\]
Therefore, we can further estimate $\int_Me^{nu }\dvg$ as follows
\begin{equation}\label{apd-eq6}
\begin{split}
 \int_Me^{nu }\dvg \leqslant & \Cscr_A \exp
\left[
\begin{split}
&\frac{n}{ (n-1)!\vol(\mathbb S^n)} \|u_0\|_{H^{n/2}(M)}^2 \\
& +\frac{2}{ (n-1)!\vol(\mathbb S^n)}\Big( 1 + \frac n{\lambda_1} \Big) \int_M Q_0^2\dvg \\
 &+ \frac n{\vol(M)} \Big(1 - \frac{2 }{ (n-1)!\vol(\mathbb S^n) } \int_M Q_0 \dvg \Big) \Big( \int_Mu \dvg \Big) 
\end{split}
\right]\\
\leqslant & C \exp \Big( \frac{n}{ (n-1)!\vol(\mathbb S^n)} \|u_0\|_{H^{n/2}(M)}^2 + B \int_Mu \dvg \Big).
\end{split}
\end{equation}
Here $B$ and $C$ are positive constants depending only on $(M,g_0)$. Combining \eqref{apd-eq1} and \eqref {apd-eq6} gives
\begin{equation*}
\begin{split}
 \int_Me^{nu }\dvg \leqslant & C'_K \exp \Big( A_1 \|u_0\|_{H^{n/2}(M)}^2 + \frac {B_1}{\vol(K)} \max \Big\{ \int_K u \dvg , 0 \Big\} \Big).
\end{split}
\end{equation*}
Again $A_1$ and $B_1$ are positive constants depending on $(M,g_0)$ and $C'_K>0$ depends on $(M,g_0)$ and $\vol(K)$. Clearly, we may assume that $B_1 > n$ by replacing it with $\max \{n+1, B_1\}$, if necessary. Therefore, by letting $\alpha = \max\{A_1, B_1/n\}>1$, we deduce that
\begin{equation}\label{apd-eq8}
\begin{split}
 \int_Me^{nu }\dvg \leqslant & C'_K \exp \big( \alpha \|u_0\|_{H^{n/2}(M)}^2 \big) \exp \Big( \frac {\alpha}{\vol(K)} \max \Big\{ \int_K nu \dvg , 0 \Big\} \Big).
\end{split}
\end{equation}
By Jensen's inequality, we have
\[
\exp \Big( \frac {\alpha}{\vol(K)} \int_K nu \dvg \Big) \leqslant \Big( \frac 1{\vol(K)} \int_K e^{nu} \dvg\Big)^\alpha.
\]
From this we know that
\begin{equation}\label{apd-eq9}
\begin{split}
\exp \Big( \frac {\alpha}{\vol(K)} \max \Big\{ \int_K nu \dvg , 0 \Big\} \Big) = &\max \Big\{ \exp \Big( \frac {\alpha}{\vol(K)} \int_K nu \dvg \Big), 1 \Big\} \\
\leqslant & \max \Big\{ \Big( \frac 1{\vol(K)} \int_K e^{nu} \dvg\Big)^\alpha ,1 \Big\} .
\end{split}
\end{equation}
Plugging \eqref{apd-eq9} into \eqref{apd-eq8} gives
\[\begin{split}
 \int_Me^{nu }\dvg \leqslant & C'_K \exp \big( \alpha \|u_0\|_{H^{n/2}(M)}^2 \big) \\
 & \times \max \Big\{ \Big( \frac 1{\vol(K)} \Big)^\alpha ,1 \Big\} \max \Big\{ \Big( \int_K e^{nu} \dvg\Big)^\alpha ,1 \Big\} \\
= & C_K \exp \big( \alpha \|u_0\|_{H^{n/2}(M)}^2 \big) \max \Big\{ \Big( \int_K e^{nu} \dvg\Big)^\alpha ,1 \Big\},
\end{split}\]
which is the desired estimate \eqref{apd-eq0}. The proof is complete.
\qed

\subsection{Proof of Lemma \ref{concentrationcompactness}}

To prove the lemma, we follow the idea in \cite{bfr-04, bfr} with some modifications. As shown in \cite[Proposition 6]{br06} and in \cite[page 939]{ChenXu-2011}, for each smooth positive function $u(t)$, there exist some point $p(t)\in \mathbb B^{n+1}$ such that the normalized companion of $u(t)$ defined by
\[
w_{p(t)} =u(t) \circ\phi_{p(t)}+\frac1n\log(\det d\phi_{p(t)}),
\]
where $\phi_{p(t)} : \mathbb S^n \to \mathbb S^n$ is the conformal diffeomorphism given by
\[
\phi_{p} (x) = \frac 1{1+2\langle p, x\rangle + |p|^2} \big[ (1-|p|^2) x + 2(1+\langle p, x\rangle )p \big] 
\]
with $p \in \mathbb B^{n+1}$, enjoys the following
\begin{equation}\label{normalizedcondition}
\int_{\mathbb S^n} x_j e^{nw_{p(t)}}\dvSn =0
\end{equation}
for all $j=1,2,...,n+1$. Furthermore, if $u(t)$ depends smoothly on the time $t$, then so does $p(t)$. By conformal invariance there holds
\[
\Escr[w_{p(t)}]=\Escr[u(t)];
\]
see for instance \cite[page 940]{ChenXu-2011}. To emphasize that we are working on $\mathbb S^n$, let us denote $\Po$ by $ \PSn $. Then by Lemma \ref{energydecay} and the preceding identity, we have 
\begin{equation}\label{normalizeenergy}
\frac{n}{2}\int_{\mathbb S^n}w_{p(t)} \cdot \PSn w_{p(t)} \dvSn +n!\int_{\mathbb S^n}w_{p(t)} \dvSn \leqslant \Escr[u_0],
\end{equation}
which implies that
\begin{equation}
 \label{integralofwupperbound}
 \int_{\mathbb S^n}w_{p(t)}\dvSn \leqslant\frac{\Escr[u_0]}{n!}.
\end{equation}
Since $u_0\in Y$, we obtain
\begin{equation*}
 \int_{\mathbb S^n}f\circ \phi_{p(t)} e^{nw_{p(t)}}\dvSn =\int_{\mathbb S^n}fe^{nu(t)} \dvSn =(n-1)!\vol(\mathbb S^n).
\end{equation*}
In particular, this gives
\begin{equation}
 \label{wvolume}
 \int_{\mathbb S^n}e^{nw_{p(t)}}\dvSn \geqslant \frac{(n-1)!\vol(\mathbb S^n)}{\sup_{\mathbb S^n}f}.
\end{equation}
Since $w_{p(t)}$ satisfies \eqref{normalizedcondition}, we can apply an improved Beckner inequality due to Wei and Xu in \cite[Theorem 2.6]{wx1998} and use \eqref{normalizeenergy} and \eqref{wvolume} to get
\begin{equation}\label{wvolumeupperbound}
\begin{split}
 \frac{(n-1)!}{\sup_{ \mathbb S^n}f} & \leqslant\frac{1}{\vol(\mathbb S^n)}\int_{\mathbb S^n}e^{nw_{p(t)}}\dvSn \\
 & \leqslant\exp\Big[\frac{1}{(n-1)!\vol(\mathbb S^n)}\Big(\frac{an}{2}\int_{\mathbb S^n}w_{p(t)} \cdot \PSn w_{p(t)}\dvSn +n!\int_{\mathbb S^n}w_{p(t)}\dvSn \Big)\Big]\\
 & \leqslant\exp\Big[\frac{n(a-1)}{2(n-1)!\vol(\mathbb S^n)}\int_{\mathbb S^n}w_{p(t)} \cdot \PSn w_{p(t)}\dvSn +\frac{\Escr[u_0]}{(n-1)!\vol(\mathbb S^n)}\Big]
\end{split}
\end{equation}
for some $0<a<1$. From this it follows that there exists a uniform constant $C_1>0$ such that
\begin{equation}
 \label{wPwupperbound}
 \int_{\mathbb S^n}w_{p(t)} \cdot \PSn w_{p(t)} \dvSn \leqslant C_1 .
\end{equation}
Plugging \eqref{wPwupperbound} into \eqref{wvolumeupperbound} yields
\begin{equation}
 \label{integralofwlowerbound}
 \int_{\mathbb S^n}w_{p(t)} \dvSn \geqslant \frac{\vol(\mathbb S^n)}{n}\log \Big(\frac{(n-1)!}{\sup_{ \mathbb S^n}f} \Big)-\frac{aC_1}{2(n-1)!}.
\end{equation}
Combining \eqref{integralofwupperbound} and \eqref{integralofwlowerbound}, we obtain that there exists a uniform constant $C_2>0$ such that
\begin{equation}
 \label{integralofwbound}
 \Big|\int_{\mathbb S^n}w_{p(t)} \dvSn \Big|\leqslant C_2,
\end{equation}
It follows from \eqref{wPwupperbound}, \eqref{integralofwbound}, and Poincare's inequality \eqref{eqPositivityOfP_0} that
\begin{equation*}
 \int_{\mathbb S^n}w^2_{p(t)} \dvSn \leqslant C_3
\end{equation*}
for some uniform constant $C_3>0$, which together with \eqref{wPwupperbound} implies that there exists a uniform constant $C_4>0$ such that
\begin{equation}\label{wh2nbound}
\| w_{p(t)} \|_{H^{n/2}(\mathbb S^n)} \leqslant C_4.
\end{equation}
It is worth noticing that all constant $C_i$ with $1 \leqslant i \leqslant 4$ are independent of $T$. Now depending on the size of $\sup_{t \in [0,T)} |p(t)|$, we have two possibilities:

\medskip
\noindent\textbf{Case 1}. Suppose that $\sup_{t\in [0,T)} |p(t)| < 1$. We shall prove that Part (i) of the lemma occurs. Indeed, it follows from \cite[page 943]{ChenXu-2011} that
\[
\det(d\phi_p^{-1}(x) ) = \Big( \frac{1-|p|^2}{1-2\langle p, x \rangle + |p|^2} \Big)^n \quad x \in \mathbb S^n,
\]
which yields
\[
\Big( \frac{1-|p|}{1+|p|} \Big)^n \leqslant\det(d\phi_p^{-1} )\leqslant \Big( \frac{1+|p|}{1-|p|} \Big)^n.
\]
Therefore, under our hypothesis, there is a constant $C \gg 1$ independent of $T$ such that
\[
0<C^{-1} \leqslant\det(d\phi_{p(t)}) \leqslant C.
\]
By Beckner's inequality \eqref{BecknerInequalityWithHnorm} and \eqref{wh2nbound}, there holds
\[
\sup_{[0,T)} \int_{\mathbb S^n} e^{n   |w(t)|} \dvSn \leqslant C_3 .
\]
Now we pass our bound for $w_{p(t)}$ in \eqref{wh2nbound} to a similar bound for $u(t)$. To this purpose, we first note that
\[\begin{aligned}
\int_{\mathbb S^n} e^{n  |u(t)|} \dvSn 
&=\int_{\mathbb S^n} \exp \big( n  |u(t) \circ \phi_{p(t)} | +  \log (\det d\phi_{p(t)} ) \big) \dvSn \\
&\leqslant\int_{\mathbb S^n} \exp \big( n  | w_{p(t)}  | +  2\log (\det d\phi_{p(t)} ) \big) \dvSn \\
&= \int_{\mathbb S^n} e^{n  |w_{p(t)}|} |\det d\phi_{p(t)}  |^2 \dvSn.
\end{aligned}\]
From this and the bounds for $\det(d\phi_{p(t)})$, we deduce that
\[
\sup_{[0,T)} \int_{\mathbb S^n} e^{n  |u(t)|} \dvSn \leqslant C_4 .
\]
In particular, applying Jensen's inequality gives
\begin{equation}
\int_{\mathbb S^n}|u(t)|\dvSn \leqslant C_5,
\label{uintegralbound}
\end{equation}
which together with Lemma \ref{energydecay} implies that
\begin{equation}
\int_{\mathbb S^n}u \cdot \PSn u\dvSn \leqslant C.
\label{uPsnubound}
\end{equation}
Now it follows from Poincar\'e's inequality, \eqref{uintegralbound}, and \eqref{uPsnubound} that
\[
\|u(t)\|_{H^{n/2}(\mathbb S^n)} \leqslant C
\]
for some constant $C>0$ independent of $T$.

\noindent\textbf{Case 2}. Suppose that $\sup_{t\in [0,T)} |p(t)| = 1$.
In this scenario, we can find a sequence $t_k\nearrow T$ in such a way that $|p(t_k)|\nearrow 1$ and that $p(t_k) \to x_\infty$ as $k \to +\infty$ for some point $x_\infty\in \mathbb S^n$. In view of \eqref{wh2nbound}, there exists a subsequence of $t_k$, still denoted by $t_k$, and some function $w_\infty\in H^{n/2}(\mathbb S^n)$ such that 
\[
w_k = : w_{p(t_k)} \rightharpoonup w_\infty
\]
weakly in $H^{n/2}(\mathbb S^n)$ and strongly in $L^2(\mathbb S^n)$. Let $r>0$ be arbitrary but fixed and set 
\[
B_k=(\phi_{p(t_k)})^{-1}(B_r(x_\infty)).
\] 
Then by H\"older's inequality, we get
\begin{equation} \label{fphitkonBkc}
\begin{split}
\Big|\int_{\mathbb S^n\backslash B_k}f\circ\phi_{p(t_k)} e^{nw_k}\dvSn \Big|&\leqslant (\sup_{\mathbb S^n} |f| )\, \Big(\vol(\mathbb S^n\backslash B_k)\int_{\mathbb S^n}e^{2n|w_k|}\dvSn \Big)^{1/2} \\
&\leqslant C \sqrt{\vol (\mathbb S^n\backslash B_k)}.
\end{split}
\end{equation}
To obtain the last inequality in \eqref{fphitkonBkc}, we have used \eqref{wh2nbound} and Beckner's inequality \eqref{BecknerInequalityWithHnorm}. Notice that, one can easily verify that as $k \to +\infty$ there holds
\[
\phi_{p(t_k)} \to \phi_{x_\infty} \equiv x_\infty
\]
uniformly in $\mathbb S^n \setminus B_\delta(x_\infty)$ with any sufficiently small $\delta>0$. Next, we need to estimate $\vol(B_k)$ for large $k$.

\smallskip\noindent\textbf{Claim}. There holds 
\[
\lim_{k \to +\infty} \vol(B_k)=\vol(\mathbb S^n).
\]
\begin{proof}[Proof of Claim]

For each $\varepsilon >0$ sufficiently small but fixed, because $\phi_{p(t_k)} \to x_\infty$ uniformly in $\mathbb S^n \setminus B_\varepsilon(x_\infty)$, there is some $N>1$ independent of $x$ such that
\[
d \big( \phi_{p(t_k)} (x) , x_\infty \big) < r
\]
for all $k \geqslant N$ and all $x \in \mathbb S^n \setminus B_\varepsilon(x_\infty)$. In particular, there holds
\[
 \mathbb S^n \setminus B_\varepsilon(x_\infty) \subset \phi_{p(t_k)} ^{-1} \big( B_r(x_\infty) \big) = B_k
\]
for all $k \geqslant N$. Thus, we have just shown that
\[
 \mathbb S^n \setminus B_\varepsilon(x_\infty) \subset B_k \subset \mathbb S^n,
\]
for large $k$, which immediately implies
\[
\vol( \mathbb S^n \setminus B_\varepsilon(x_\infty) ) \leqslant \liminf_{k \to +\infty} \vol(B_k) \leqslant \limsup_{k \to +\infty} \vol(B_k) \leqslant \vol(\mathbb S^n).
\]
Letting $\varepsilon \searrow 0$ gives the desired result.
\end{proof}

As a consequence of the above Claim, we get
\[
\lim_{k \to +\infty}\vol (\mathbb S^n\backslash B_k)=0.
\] 
Keep in mind that
\[
\int_{B_r(x_\infty)}f\ e^{nu(t_k)}\dvSn = \int_{B_k}f\circ\phi_{p(t_k)} e^{nw_k}\dvSn.
\]
This fact together with \eqref{fphitkonBkc} and the fact $u(t)\in Y$ implies that
\[
\lim_{k \to +\infty}\int_{B_r(x_\infty)}fe^{nu(t_k)}\dvSn =(n-1)!\vol(\mathbb S^n)
\]
for all $r>0$. This establishes \eqref{concentration}. Now, let $y\in \mathbb S^n\backslash\{x_\infty\}$ and $0\leqslant r<\dist (y,x_\infty)$. Then we can choose some $0< s < \dist (y,x_\infty) - r$ and fix it. By the argument leading to \eqref{concentration}, we know that
\[
\lim_{k \to +\infty} \vol (\mathbb S^n \setminus (\phi_{p(t_k)})^{-1}(B_s(x_\infty))) = 0.
\]
However, we have $B_s(x_\infty) \cap B_r(y) = \emptyset$, which implies that 
\[
\phi_{p(t_k)}^{-1}(B_r(y)) \cap \phi_{p(t_k)}^{-1}(B_s (x_\infty))= \emptyset.
\] 
Hence, there holds
\[
\lim_{k \to +\infty}\vol \big( \phi_{p(t_k)}^{-1}(B_r(y)) \big)=0.
\]
However, as in \eqref{fphitkonBkc}, we can estimate
\[\begin{split}
\Big| \int_{B_r(y)}fe^{nu(t_k)}\dvSn \Big| = & \Big| \int_{\phi_{p(t_k)}^{-1}(B_r(y))}f\circ\phi_{p(t_k)} e^{nw_k}\dvSn \Big| \\
\leqslant & (\sup_{\mathbb S^n} |f| ) \Big( \int_{\mathbb S^n} e^{2n|w_k|}\dvSn \Big)^{1/2} \Big( \int_{\phi_{p(t_k)}^{-1}(B_r(y))} \dvSn \Big)^{1/2}.
\end{split}\]
Keep in mind that, by Beckner's inequality \eqref{BecknerInequalityWithHnorm}, the term $\int_{\mathbb S^n} e^{2n|w_k|}\dvSn$ is uniformly bounded. Therefore, we get
\[
\lim_{k \to +\infty}\int_{B_r(y)}fe^{nu(t_k)}\dvSn =0.
\]
This establishes \eqref{concentration1}. Hence Part (ii) of the lemma is proved. \qed

\section{An alternative proof of Ge--Xu's result}
\label{apd-GeXu}

In this appendix, we provide an alternative proof of \cite[Theorem 3.2]{gx2008}. Using our notation and conventions, we shall reprove the following result.

\begin{theorem}[Ge and Xu]\label{thmGeXu}
Let $(M, g_0)$ be a compact, oriented $n$-dimensional Riemannian manifold with $n$ even. Assume that the GJMS operator $\Po $ is positive with kernel consisting of constant functions. Moreover, assume that the metric $g_0$ satisfies
\[
\int_M Q_0 \dvg = 0.
\]
If $f$ is a smooth function on $M$ such that $\int_M f \dvg<0$, then there exists a conformal metric $g \in [g_0]$ such that $Q_g = f$.
\end{theorem}

To prove Theorem \ref{thmGeXu} above, we also use variational techniques; however, the difference between our approach and that of \cite{gx2008} is the constraint. Thanks to the resolution of the problem of prescribing constant $Q$-curvature \cite{n2007}, we can assume at the beginning that $g_0$ has $Q_0 \equiv 0$. From this fact, to conclude Theorem \ref{thmGeXu}, it is equivalent to solving
\begin{equation}\label{eqPDE}
\Po u =f e^{nu}
\end{equation}
for $u$ under the condition $\int_M f \dvg< 0$.

\noindent\textbf{Step 1}. Inspired by Lemma \ref{volumepreserve}, we let
\begin{equation}\label{eqConstraint}
\mathscr H = \Big\{ u \in H^{n/2} (M) : \int_M f e^{nu} \dvg = 0, \int_M e^{nu} \dvg = \vol(M) \Big\}.
\end{equation}
As can be easily observed, the condition $\int_M u \dvg= 0$ in \cite{gx2008} is replaced by the condition $\int_M e^{nu} \dvg = \vol(M)$ in \eqref{eqConstraint}.

As always, we shall minimize the following functional energy
\begin{equation*}
\Escr [u]= \int_Mu\cdot \Po u\dvg
\end{equation*}
in $\mathscr H$. Following the method in \cite{gx2008}, it is not hard to see that the set $\mathscr H$ is non-empty, thanks to $\int_M f \dvg < 0$. Since $\mathscr H$ is non-empty and $\Escr$ is bounded from below by $0$, we know that
$$\mu = \inf_{u \in \mathscr H} \Escr(u)$$
is finite and non-negative. As in \eqref{eqNormH^n/2}, we still use
\[
\| u\| _{H^{n/2}(M)} = \Big( \int_Mu\cdot \Po u\dvg +\int_Mu^2\dvg \Big)^{1/2}
\]
as an equivalent norm on $H^{n/2} (M)$. Suppose that $(u_k)_k \subset \mathscr H$ is a minimizing sequence for $\mu$. Clearly, $\int_Mu_k\cdot \Po u_k\dvg$ is bounded. Hence, to bound the sequence $(u_k)_k$ in $H^{n/2} (M)$, it suffices to bound $\int_M u_k^2 \dvg$. However, this can be easily obtained, in the same fashion of Lemma \ref{lem-BoundInH^{n/2}-zero}, since we can easily bound $\int_M u_k \dvg$. Therefore, up to subsequences, there exists some $u_\infty \in H^{n/2} (M)$ such that
\begin{itemize}
 \item $u_k \rightharpoonup u_\infty$ weakly in $H^{n/2} (M)$ and
 \item $u_k \to u_\infty$ almost everywhere in $M$
\end{itemize}
as $k \to +\infty$. Clearly,
\[
\Escr [u_\infty] \leqslant \liminf_{k \to +\infty} \Escr [u_k].
\]
From this we conclude that $\Escr [u_\infty] = \mu$ and that $u_\infty$ is an optimizer for $\mu$ since $u_\infty \in \mathscr H$, thanks to Trudinger's inequality. By Lagrange's multiplier theorem, there exist two constants $\alpha$ and $\beta$ such that
\begin{equation}\label{eqAfterLagrange}
\int_M \varphi \cdot \Po u_\infty \dvg = \alpha \int_M e^{nu_\infty} \varphi \dvg + \beta \int_M f e^{nu_\infty} \varphi \dvg
\end{equation}
for any test function $\varphi \in H^{n/2} (M)$. Testing \eqref{eqAfterLagrange} with $\varphi \equiv 1$ gives $\alpha = 0$ and hence $\beta \ne 0$.

\medskip\noindent\textbf{Step 2}. To realize that our equation \eqref{eqPDE} has a solution, it is necessary to prove that $\beta>0$. To this purpose, we minimize $\Escr$ over the set
\begin{equation*}
\widetilde{\mathscr H} = \Big\{ u \in H^{n/2} (M) : \int_M f e^{nu} \dvg \geqslant 0, \frac {\vol(M)}2\leqslant \int_M e^{nu} \dvg \leqslant \frac {3\vol(M)}2 \Big\}.
\end{equation*}
Set
$$\widetilde\mu = \inf_{u \in \widetilde{\mathscr H} } \Escr(u).$$
Because $\mathscr H \subset \widetilde{\mathscr H} $ we clearly have $\widetilde\mu \leqslant \mu$. Using the same argument as before, there exists some function $\widetilde u_\infty \in \widetilde{\mathscr H}$ such that $\widetilde\mu = \Escr(\widetilde u_\infty)$. (Note that the positive boundedness away from zero of $\int_M e^{nu} \dvg$ will be enough to guarantee an upper bound for $\overline u$; see Lemma \ref{lem-BoundInH^{n/2}-zero}.) Suppose that $\widetilde\mu < \mu$, again by Lagrange's multiplier theorem, there holds
\begin{equation}\label{eqAfterLagrange2}
\int_M \varphi \cdot \Po \widetilde u_\infty \dvg = 0
\end{equation}
for any $\varphi \in H^{n/2} (M)$. By testing \eqref{eqAfterLagrange2} with $\varphi \equiv \widetilde u_\infty$ we deduce that $ \widetilde u_\infty$ is constant. From this we obtain a contradiction because $\widetilde u_\infty \in \widetilde{\mathscr H} \backslash \mathscr H$ and
\[
0 \leqslant \int_M f e^{n \widetilde u_\infty} \dvg = e^{n\widetilde u_\infty} \int_M f \dvg < 0.
\]
Thus, we have just shown that $\widetilde\mu =\mu$. By contradiction, suppose that $\beta <0$. Then by \eqref{eqAfterLagrange}, we have
\[
\Escr [u_\infty] = \mu = \beta \int_M f e^{nu_\infty} u_\infty \dvg ,
\]
which implies that $ \int_M f e^{nu_\infty} u_\infty \dvg<0$. Set
\[
l(t) = \int_M f e^{n t u_\infty} \dvg.
\]
Clearly, $l(1) = 0$ and $l'(1)<0$. Hence there exists some $\lambda \in (0,1)$ closed to $1$ such that $l(\lambda )>0$ and
\[
\frac {\vol(M)}2\leqslant \int_M e^{n\lambda u_\infty} \dvg \leqslant \frac{3\vol(M)}2 .
\]
In other words, we conclude that $\lambda u_\infty \in \widetilde{\mathscr H}$ which implies that
\[
\widetilde\mu \leqslant \Escr [\lambda u_\infty] = \lambda^2 \Escr [u_\infty] < \Escr [u_\infty] = \widetilde\mu.
\]
This contradiction shows $\beta>0$ as we wish.

\medskip
\noindent\textbf{Step 3}. Once we can show that $\beta>0$, it is immediate to verify that $u_\infty + (\log \beta ) /n$ solves \eqref{eqPDE}. The proof is complete.




\begin{thebibliography}{9999999}

\bibitem[Ada88]{adams}
\textsc{D. Adams},
A sharp inequality of J. Moser for higher order derivatives,
\textit{Ann. of Math.} \textbf{128} (1988) 385--398.

\bibitem[BFR04]{bfr-04}
\textsc{P. Baird, A. Fardoun, R. Regbaoui},
The evolution of the scalar curvature of a surface to a prescribed function,
\textit{Ann. Sc. Norm. Super. Pisa Cl. Sci. (5)} \textbf{3} (2004) 17--38.
 
\bibitem[BFR06]{bfr}
\textsc{P. Baird, A. Fardoun, R. Regbaoui},
$Q$-curvarure flow on $4$-manifolds,
\textit{Calc. Var.} \textbf{27} (2006) 75--104.

\bibitem[BFR09]{bfr09}
\textsc{P. Baird, A. Fardoun, R. Regbaoui},
Prescribed $Q$-curvature on manifolds of even dimension,
\textit{J. Geom. Phys.} \textbf{59} (2009) 221--233.

\bibitem[Bra85]{Branson85}
\textsc{T. Branson},
Differential operators canonically associated to a conformal structure,
\textit{Math. Scand.} \textbf{57} (1985) 293--345.

\bibitem[Bra93]{Branson93}
\textsc{T. Branson},
The functional determinant, Lecture notes series, No. 4, Seoul National University, 1993.

 \bibitem[Bre03]{br}
 \textsc{S. Brendle},
 Global existence and convergence for a higher order flow in conformal geometry,
 \textit{Ann. of Math.} \textbf{158} (2003) 323--343.

\bibitem[Bren03]{br03}
\textsc{S. Brendle},
Prescribing a higher order conformal invariant on $\mathbb S^n$,
\textit{Comm. Anal. Geom.} \textbf{11} (2003) 837--858.

 \bibitem[Bre06]{br06}
 \textsc{S. Brendle},
Convergence of the $Q$-curvature flow on $\mathbb S^4$,
\textit{Adv. Math.} \textbf{205} (2006) 1--32.

\bibitem[CY92]{ChangYang1992}
\textsc{S.-Y.A. Chang, P.C. Yang}, 
Extremal metrics of zeta function determinants on $4$-manifolds,
\textit{Ann. of Math.} \textbf{142} (1992) 171--212.

\bibitem[CX11]{ChenXu-2011}
\textsc{X. Chen, X. Xu},
$Q$-curvature flow on the standard sphere of even dimension,
\textit{J. Funct. Anal.} \textbf{261} (2011) 934--980.

\bibitem[DM04]{dm2004}
\textsc{Z. Djadli, A. Malchiodi},
A fourth order uniformization theorem on some four manifolds with large total $Q$-curvature, 
\textit{C. R. Acad. Sci. Paris, Ser. I} \textbf{340} (2005) 341--346.

\bibitem[DM08]{dm2008}
\textsc{Z. Djadli, A. Malchiodi},
Existence of conformal metrics with constant $Q$-curvature,
\textit{Ann. of Math.} \textbf{168} (2008) 813--858.

\bibitem[FR12]{fr12}
\textsc{A. Fardoun, R. Regbaoui},
$Q$-curvature flow for GJMS operators with non-trivial kernel,
\textit{J. Geom. Phys.} \textbf{62} (2012) 2321--2328.

\bibitem[FR18]{fr18}
\textsc{A. Fardoun, R. Regbaoui},
Compactness properties for geometric fourth order elliptic equations with application to the $Q$-curvature flow,
\textit{J. reine angew. Math.} \textbf{734} (2018) 229--264.

\bibitem[FG85]{fg1985}
\textsc{C. Fefferman, C.R. Graham},
Conformal invariants. The mathematical heritage of \'Elie Cartan (Lyon, 1984). \textit{Ast\'erisque} 1985, Num\'ero Hors S\'erie, 95--116.

\bibitem[Ga17]{Ga17}
\textsc{L. Galimberti} ``Large'' conformal metrics of prescribed $Q$-curvature in th negative case.
\textit{Nonlinear Differ. Equ. Appl.} 2017.

\bibitem[GX08]{gx2008}
\textsc{Y.Ge, X. Xu},
Prescribed $Q$-curvature problem on closed $4$-Riemannian manifolds in the null case,
\textit{Calc. Var. PDE.} \textbf{31} (2008) 549--555.

\bibitem[GJMS92]{GJMS}
\textsc{C.R. Graham, R. Jenne, L. Mason, G. Sparling},
Conformally invariant powers of the Laplacian, I: existence,
\textit{J. London Math. Soc.} \textbf{46} (1992) 557--565.

\bibitem[Gur99]{Gursky}
\textsc{M. Gursky},
The principal eigenvalues of a conformally invariant differential operator
with an application to semilinear elliptic PDE,
\textit{Comm. Math. Phys.} \textbf{207} (1999) 131--143.

\bibitem[Li95]{YYLi}
\textsc{Y.Y. Li},
Prescribing scalar curvature on $\mathbb S^n$ and related problems, part I, 
\textit{J. Differential Equations} \textbf{120} (1995), pp. 319--410.

\bibitem[Ndi07]{n2007}
\textsc{C.B. Ndiaye},
Constant $Q$-curvature metrics in arbitrary dimension,
\textit{J. Funct. Anal.} \textbf{251} (2007) 1--58.

\bibitem[Ndi15]{n2015}
\textsc{C.B. Ndiaye},
Algebraic topological methods for the supercritical $Q$-curvature problem. ,
\textit{Adv. Math.,} \textbf{277} (2015) 56--99.

\bibitem[NX15]{NX2015}
\textsc{Q.A. Ng\^o, X. Xu},
Slow convergence of the Gaussian curvature flow on closed Riemannian surfaces with vanishing Euler characteristic,
2015, \textit{unpublished notes}.

\bibitem[NZ19a]{NZ-arXiv}
\textsc{Q.A. Ng\^o, H. Zhang}, 
\textit{$Q$-curvature flow on closed manifolds of even dimension},
arXiv:1701.02247v2. 

\bibitem[NZ19b]{NZ}
\textsc{Q.A. Ng\^o, H. Zhang}, 
\textit{Bubbling of the prescribed $Q$-curvature equation on 4-manifolds in the null case}, 
arXiv:1903.12054. 

\bibitem[Ho12]{ho2012}
\textsc{P.T. Ho},
Prescribed $Q$-curvature flow on $\mathbb S^n$,
\textit{J. Geom. Phys.} \textbf{62} (2012) 1233--1261.

\bibitem[LLL12]{lll2012}
\textsc{J. Li, Y. Li, P. Liu},
The $Q$-curvature on a $4$-dimensional Riemannian manifold $(M,g)$ with $\int_M QdVg=8\pi^2$,
\textit{Adv. Math.} \textbf{231} (2012) 2194--2223.

\bibitem[MS06]{ms2006}
\textsc{A. Malchiodi, M. Struwe},
$Q$-curvature flow on $\mathbb S^4$,
\textit{J. Differential Geom.} \textbf{73} (2006) 1--44.

\bibitem[Pan83]{Paneitz}
\textsc{S. Paneitz},
A quartic conformally covariant differential operator for arbitrary pseudo-Riemannian manifolds, preprint, 1983.

\bibitem[Rob11]{FRobert}
\textsc{F. Robert},
Admissible $Q-$curvatures under isometries for the conformal GJMS operators, \textit{Nonlinear elliptic partial differential equations}, 241--259, Contemporary Mathematics, 540, Amer. Math. Soc., Providence, RI, 2011. 

\bibitem[Str18]{Struwe}
\textsc{M. Struwe},
``Bubbling'' of the prescribed curvature flow on the torus, \textit{J. Eur. Math. Soc.}, to appear, 2018.

\bibitem[WX98]{wx1998}
\textsc{J. Wei, X. Xu},
On conformal deformations of metrics on $\mathbb S^n$,
\textit{J. Funct. Anal.} \textbf{157} (1998) 292--325.

\end{thebibliography}
\end{document}